\documentclass{article}
\usepackage[margin=3cm]{geometry}

\usepackage{amsmath,amsthm,amssymb, mathrsfs}
\usepackage{mathtools}
\usepackage{cases}
\usepackage{xcolor}
\usepackage{graphics}
\usepackage{tikz}
\tikzstyle{inside}=[fill=white, inner sep=0pt, circle]
\usetikzlibrary{calc, spy}
\usetikzlibrary{patterns}

\usepackage{hyperref}
\usepackage{zref-clever}
\zcsetup{nameinlink=true,noabbrev,cap=true}
\AddToHook{env/theorem/begin}{%
    \zcsetup{countertype={lemma=theorem}}}
\AddToHook{env/corollary/begin}{%
    \zcsetup{countertype={lemma=corollary}}}
\AddToHook{env/definition/begin}{%
    \zcsetup{countertype={lemma=definition}}}

\newtheorem{lemma}{Lemma}[section]
\newtheorem{theorem}[lemma]{Theorem}
\newtheorem{corollary}[lemma]{Corollary}

\theoremstyle{plain}
\newtheorem{definition}[lemma]{Definition}
\newtheorem{remark}[lemma]{Remark}

\newcommand{\red}[1]{\textcolor{red}{#1}}
\newcommand{\blue}[1]{\textcolor{blue}{#1}}

\title{The Dehn function of Thompson's group $V$ is at most sextic}
\author{Andreas Lorrain}

\begin{document}
\maketitle

\begin{abstract}
A remarkable result of Guba shows that the Dehn function of Thompson's group $F$ is quadratic \cite[Theorem 1]{GubaFQuadratic}. In another paper \cite{GubaFTV}, Guba showed that Thompson's group $T$ and $V$ have Dehn functions bounded above by $n^7$ and $n^{11}$, respectively. Recently, Migliorini \cite[Theorem A]{Migliorini_2025} showed that the Dehn function of Thompson's group $T$ is also quadratic. In this paper, we improve the upper bound of the Dehn function of Thompson's group $V$ to $n^6$ and explore the intricacies of its relations via explicit van Kampen diagrams.
\end{abstract}

\tableofcontents

\section{Introduction}
The Dehn function of a group gives a measure of how easy (or hard) it is to
explicitly work with a (or all) presentations of that group. In essence, it
describes a complexity. Therefore it is always interesting when seemingly complex groups
turn out to have tame Dehn functions. An example of this, is the family of Thompson
groups ($F$, $T$, and $V$) introduced by Richard Thompson in the sixties
(see \cite{IntroNotesF} for an introduction to these groups).
Even though these groups are known to have unintuitive properties and complex dynamics,
V. Guba has shown that they have polynomial Dehn functions (\cite{GubaFTV}). Six years later in \cite[Theorem 1]{GubaFQuadratic},
he refined his original methods to prove that $F$ has a quadratic Dehn function;
the smallest it can get for non-hyperbolic groups.

The next group in line was $T$. In \cite{WangT} K. Wang, Z. Zheng, and J. Zhang pushed the arguments of Guba a bit further,
and achieved a better upper bound for the Dehn function of $T$. They showed that it was at most quintic.
10 years later, M. Migliorini \cite[Theorem A]{Migliorini_2025} circumvented the limitations of the techniques employed by Guba, and Wang et al
to show that $T$ too has a quadratic Dehn function.

The last remaining group in this line is $V$. This paper tackles this problem by
staying close to the original arguments by Guba in \cite{GubaFTV}. As such, some of the results we prove are slight alterations of Guba's proofs. Where this paper improves, is the
art of puzzling with van Kampen diagrams in order to push certain arguments to their limits.

In \zcref{sec:rel_in_T} we lay some groundwork by handling relations fully in $T$. In
\zcref{sec:rel_in_V} we tackle the relations that are exclusive to $V$. Specifically how generators
interact in pairs. \zcref{sec:pi_words} is where things get more complex as we consider what happens with $\pi$-words, and
discuss their dynamics. Finally, everything comes together in \zcref{sec:dehn_of_V} where the main theorem is proven.

\subsection{Presentations for the Thompson groups}\label{subs:presentations_Thompson_groups}
Let us fix some standard definitions and choose which conventions to use.
For the Thompson groups $F$, $T$, and $V$, we use the standard group presentations found in \cite{GubaFTV} and---with slightly different notation---in \cite{IntroNotesF}. For conjugations we choose the convention that $x^y = y^{-1}xy$.
\begin{equation*}
    F = \langle x_0,x_1 \mid x_2^{x_1} = x_3,\; x_3^{x_1} = x_4\rangle
\end{equation*}
\begin{align*}
    T = \langle x_0,x_1,c_1 \mid & \; x_2^{x_1} = x_3,\; x_3^{x_1} = x_4,\; c_2x_2 = x_1c_3,\\
                                &\; c_1x_0 = c_2^2,\; x_1c_2 = c_1,\; c_1^3 = 1\rangle
\end{align*}
\begin{align*}
    V =  \langle x_0,x_1,c_1,\pi_0 \mid & \; x_2^{x_1} = x_3, x_3^{x_1} = x_4,\\
										& \; c_2x_2 = x_1c_3,\; c_1x_0 = c_2^2,\\
										& \; x_1c_2 = c_1,\; c_1^3 = 1,\\
										& \; \pi_1^2 = 1,\; \pi_1\pi_3 = \pi_3\pi_1,\\
				        				& \; (\pi_2\pi_1)^2 = 1,\; x_3\pi_1 = \pi_1x_3,\\
										& \; \pi_1x_2 = x_1\pi_2\pi_1,\; \pi_2^{x_1} = \pi_3,\\
										& \; \pi_1^{c_3} = \pi_2,\; (\pi_1c_2)^3 = 1\rangle
\end{align*}
where the following definitions hold
\begin{equation}\label{def:definitions_of_infinite_letters}
    x_n = x_1^{x_0^{n-1}},\; c_n = x_0^{-(n-1)}c_1x_1^{n-1},\; \pi_1 = \pi_0^{c_2},\; \pi_n = \pi_1^{x_0^{n-1}}
\end{equation}
These are the finite presentations we will work with for the Dehn functions. This also means
that when we talk about the \textbf{length} of a word without specifying, we refer
to the word length with respect to the generators of these presentations.
Thompson groups also have infinite presentations that are perhaps more
natural to work with. We will denote them with a scripted symbol.

\begin{equation*}
    \mathscr{F} = \langle x_0,x_1,x_2,\dots\mid x_n^{x_k} = x_{n+1},\quad (\forall k<n) \rangle
\end{equation*}
\begin{align*}
    \mathscr{T} = \langle x_0,x_1,x_2,\dots,c_1,c_2,c_3,\dots\mid & x_n^{x_k} = x_{n+1},\quad (k<n)\\
                                                                & c_nx_k = x_{k-1}c_{n+1}\quad (k\leq n)\\
                                                                & c_nx_0 = c_{n+1}^2\\
                                                                & x_nc_{n+1} = c_n\\
                                                                & c_n^{n+2} = 1\rangle
\end{align*}
\begin{align*}
    \mathscr{V} = \langle x_0,x_1,x_2,\dots,c_1,c_2,c_3,\dots,\pi_0,\pi_1,\pi_2,\dots\mid & x_n^{x_k} = x_{n+1}\quad (k<n),\\
                                                                                        & c_nx_k = x_{k-1}c_{n+1}\quad (k\leq n),\\
                                                                                        & c_nx_0 = c_{n+1}^2,\\
                                                                                        & x_nc_{n+1} = c_n,\\
                                                                                        & c_n^{n+2} = 1,\\
                                                                                        & \pi_n^2 = 1,\\
                                                                                        & \pi_k\pi_n = \pi_n\pi_k\quad (k\leq n-2),\\
                                                                                        & (\pi_{n+1}\pi_n)^3 = 1,\\
                                                                                        & x_n\pi_k = \pi_kx_n\quad (k\leq n-2),\\
                                                                                        & \pi_nx_{n+1} = x_n\pi_{n+1}\pi_n,\\
                                                                                        & \pi_n^{x_k} = \pi_{n+1}\quad (k<n),\\
                                                                                        & \pi_k^{c_n} = \pi_{k+1}\quad (k\leq n-2),\\
                                                                                        & (\pi_kc_n)^{n+1} = 1\quad (k<n)\rangle
\end{align*}
These give us some more freedom when trying to reduce words to the identity. Therefore,
we will take inspiration from these for our choices of relations to construct
van Kampen diagrams for. Furthermore when working with these infinite generating sets,
we will employ the following length.
\begin{definition}[length with respect to the infinite generating set]\label{def:length_inf_gen_set}
    We denote $|w|_\infty$, the length of word $w$  in the generating set
    \begin{equation*}
        \{x_0,x_1,x_2,\dots,c_1,c_2,c_3,\dots,\pi_0,\pi_1,\pi_2,\dots\}
    \end{equation*}
\end{definition}

\subsection{Dehn functions and van Kampen diagrams}

We assume the reader is already somewhat familiar with Dehn functions and van Kampen diagrams. Therefore, we only briefly recall some core definitions. For a more elaborate introduction see \cite[Part II, Chapter 1]{BradyRileyShortGeomOfWordProblem} and \cite[Office hour 8]{OfficeHours}.

\begin{definition}[Dehn function]\label{def:Dehn_function}
    Given a finite group presentation $G = \langle X\mid R\rangle$, an \textbf{isoperimetric function} $f\colon \mathbb{N}\to \mathbb{N}$ for $G$ is a function such that any word $w$ of length $n$ in the generators $X$ that represents $1$ in $G$, can be reduced to $1$ with at most $f(n)$ applications of relators in $R$. The \textbf{Dehn function} of $G$ is the pointwise minimal isoperimetric function.
\end{definition}
\begin{definition}[van Kampen diagram]
    A \textbf{van Kampen diagram} over a group presentation $G=\langle X\mid R\rangle$ is a simply connected directed planar graph where each edge is labeled by a letter in $X$ and every elementary cycle---called a \textbf{cell}---bears a relator in $R$. The word that can be read on the boundary of a van Kampen diagram
    (in one of the two directions) is called the \textbf{boundary word}, and 
    it can be shown, through the so-called van Kampen lemma that  $w=1$ in $G$ if and only if there exists a van Kampen diagram over $G$ with $w$ as its boundary word \cite{EgbertVanKampen}. This means that the boundary of a van Kampen diagram bears a true equality in the group, and thus we sometimes refer to it as the \textbf{boundary equation}. A subgraph of a van Kampen diagram that is itself also a van Kampen diagram, is called a \textbf{subdiagram} or a \textbf{region}.
\end{definition}
\begin{definition}[Dehn functions via van Kampen diagrams]
    The number of cells in a van Kampen diagram $\Delta$ is called its \textbf{area} and is denoted by $\mathrm{Area}(\Delta)$. For $w$ a word in $G=\langle X\mid R\rangle$, we define
    \begin{equation*}
        \|w = 1\| = \min\big\{\mathrm{Area}(\Delta)\mid \Delta \text{ a van Kampen diagram with boundary equation }w=1\big\}
    \end{equation*}The Dehn function $f$ over $G=\langle X\mid R\rangle$ defined in \zcref{def:Dehn_function}, can also be formulated as
    \begin{equation*}
        f(n) = \max\big\{\|w=1\| \mid w \text{ word in $G$ of length $n$}\big\}
    \end{equation*}
\end{definition}
\begin{definition}[dipoles]
    A \textbf{dipole} is a pair of adjacent cells which are mirror images of each other. An example is highlighted on \zcref{fig:def_van_Kampen_diagrams}. A dipole can be \textbf{reduced} (or \textbf{collapsed}) by replacing it by one copy of the path not shared both its cells. For example, in \zcref{fig:def_van_Kampen_diagrams} the two cells in the dipole share the edge with label $c_4$, and gets reduced to the path with edges labeled $x_3$ and $c_3$, with the orientation preserved. This reduces the area of the ambient van Kampen diagram, without changing its boundary equation.
\end{definition}

\subsubsection{Conventions for van Kampen diagrams}
\label{ssec:conventions}
Van kampen diagrams are by definition labeled and directed graphs. But, as the reader
will realise further down, van Kampen diagrams will get big and complex very quickly.
To keep some semblance of legibility, we will often omit to mark the direction of an
edge, and sometimes also omit the labels. When this happens, the orientation can be deduced from the context of the surrounding edges and labels as regions must still have a word on the boundary that represents the identity in the group.

\begin{figure}
    \begin{center}
    \begin{tikzpicture}
        \fill[blue!20] (2,0) -- (2,2) -- (4,3) -- (4,1) -- cycle;
        \begin{scope}[every path/.style = {-stealth}, every node/.style={inside}]
            \draw (0,0) -- node {$x_2$} (0,2);
            \draw (0,2) -- node {$x_1$} (2,2);
            \draw (0,0) -- node {$x_1$} (2,0);
            \draw (2,0) -- node {$x_3$} (2,2);

            \draw (2,0) -- node {$c_3$} (4,1);
            \draw (2,2) -- node {$c_4$} (4,1);
            \draw (4,3) -- node {$x_3$} (2,2);
            \draw (4,3) -- node {$c_3$} (4,1);

            \draw (4,3) -- node {$x_1$} (6,3);
            \draw (6,3) -- node {$c_4$} (6,1);
            \draw (4,1) -- node {$x_2$} (6,1);
            \draw[ultra thick] (6.5,2) -- (7.5,2);

            \draw (8,2) -- node {$x_2$} (8,4);
            \draw (8,4) -- node {$x_1$} (10,4);
            \draw (8,2) -- node {$x_1$} (10,2);
            \draw (10,2) -- node {$x_3$} (10,4);

            \draw (10,2) -- node {$c_3$} (10,0);
            \draw (10,0) -- node {$x_2$} (12,0);
            \draw (10,2) -- node {$x_1$} (12,2);
            \draw (12,2) -- node {$c_4$} (12,0);
        \end{scope}
        \node (N) at (-1,3) {cell};
        \draw[thin,->] (N.south) to[bend right] (0.8,1.2);
        \node (D) at (3,4) {dipole};
        \draw[thin,->] (D.south) to[bend left] +(0.2,-1.8);
        \node (C) at (7,-1) {collapse};
        \draw[thin,->] (C.north) to[bend right] (7,1.8);
    \end{tikzpicture}
    \end{center}
    \caption{Collapsing a dipole in a van Kampen diagram over Thompson's group $F$.}
    \label{fig:def_van_Kampen_diagrams}
\end{figure}
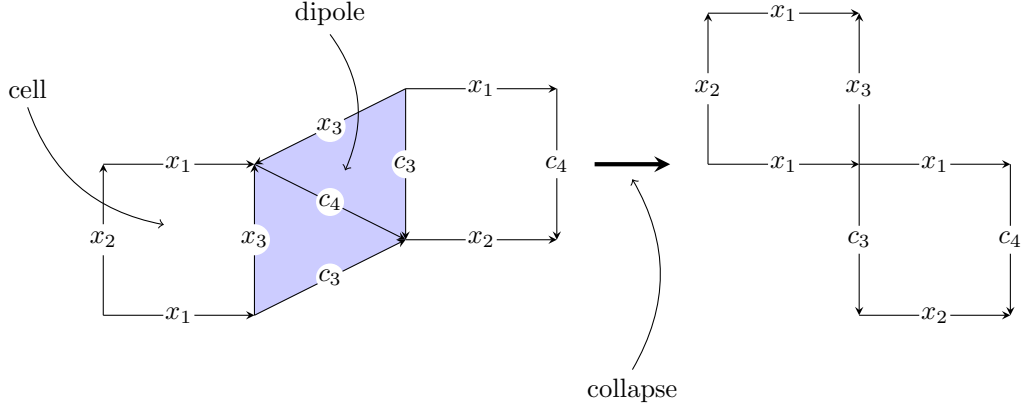

\subsubsection{The cost of definitions}
The presentations for $F$, $T$, and $V$ defined in \zcref{subs:presentations_Thompson_groups} freely use the notations of $x_n$, $c_n$, $\pi_n$ for all $n$ via the definitions in \zcref{def:definitions_of_infinite_letters}. We repeat them again for readability:
\begin{equation*}
    x_n = x_1^{x_0^{n-1}},\; c_n = x_0^{-(n-1)}c_1x_1^{n-1},\; \pi_1 = \pi_0^{c_2},\; \pi_n = \pi_1^{x_0^{n-1}}
\end{equation*}
It is important to keep in mind that these definitions are nothing more than notations to make our work readable. These are not relations from the presentations and thus cost zero when computing the area of a diagram. Nevertheless, we will often draw these as cells in van Kampen diagrams. For example, the following van Kampen diagram only contains 1 cell (the central one). All the rest are constructed via definitions.
\begin{equation*}
    \begin{tikzpicture}
        \draw (-2,-2) -- node[left] {$x_n$} (-2,2)
                        -- node[above] {$x_{n+2}$} (2,2)
                        -- node[right] {$x_n$} (2,-2)
                        -- node[below] {$x_{n+1}$} cycle;
        \draw (-1,-1) -- node[left] {$x_1$} (-1,1)
                        -- node[above] {$x_{3}$} (1,1)
                        -- node[right] {$x_1$} (1,-1)
                        -- node[below] {$x_{2}$} cycle;
        \draw (-2,-2) -- node[inside] {$x_0^{n-1}$} (-1,-1);
        \draw (-2,2) -- node[inside] {$x_0^{n-1}$} (-1,1);
        \draw (2,-2) -- node[inside] {$x_0^{n-1}$} (1,-1);
        \draw (2,2) -- node[inside] {$x_0^{n-1}$} (1,1);
    \end{tikzpicture}
\end{equation*}

\subsubsection{Dehn functions and big O notation}
The Dehn function is not a quasi-isometry invariant of groups. 
That is, different presentations of the same group can yield different Dehn functions. This issue is circumvented by putting an equivalence relation on Dehn functions. We say that $f\preceq g$ if there exists a positive integer $C$ such that
\begin{equation*}
    \forall n\in\mathbb{Z}_{>0},\quad f(n)\leq Cg(Cn) + C
\end{equation*}
Then, $f\sim g$ if $f\preceq g$ and $g\preceq f$.
\begin{remark}
    In the literature (e.g. \cite[Definition 1.3.2]{BridsonGeometryOfWordProblem}, \cite[p.83]{BradyRileyShortGeomOfWordProblem}), it is more common to use the definition
    \begin{equation*}
        f\preceq g \iff \exists C \text{ such that }\forall n\in\mathbb{Z}_{>0},\quad f(n)\leq Cg(Cn) + Cn + C
    \end{equation*}
    which results in the same equivalence classes, except that the equivalence class of constant and linear functions is the same.
    For the purpose of Dehn functions this does not matter as there are no groups with a constant Dehn function. The reason we employ the finer one, is that it will be useful to be able to descern constants for certain computations.
\end{remark}

With this equivalence on Dehn functions, the following is true.
\begin{theorem}
    Let $f$ and $g$ be two Dehn functions for two presentations of the same group. Then,
    \begin{equation*}
        f\sim g
    \end{equation*}
\end{theorem}
\begin{proof}
    A proof can be found in \cite[Proposition 1.3.3]{BridsonGeometryOfWordProblem}.
\end{proof}
To denote that a Dehn function is equivalent to a certain function, we will borrow the "big O" notation from computer science.
\begin{definition}
    If $f\sim g$, we may write $f=O(g)$. For example, $f=O(n^2)$ for ``quadratic'' functions and $g=O(1)$ for constant functions. By abuse of notation, we will also sometimes use $f\leq O(g)$ if we want to highlight that $f\preceq g$
    but maybe $g\not\preceq f$.
\end{definition}

\subsection{Normal forms in Thompson groups}
In order to determine the Dehn function of a group, it is sufficient to
know how many relations it takes to rewrite a word to a well chosen form. Ideally a unique normal form, but we will relax this for $T$ and $V$.

\begin{definition}[normal form for $F$]
    A \textbf{positive} word in $F$ is a word on the alphabet $\{x_0,x_1,x_2,\dots\}$. A \textbf{monotone positive} word (or MP-word) in $F$ is a word of the form
    \begin{equation*}
        x_{i_1}^{d_1}x_{i_2}^{d_2}\cdots x_{i_n}^{d_n}
    \end{equation*}
    such that $i_1<i_2<\cdots<i_n$ and $\forall j\colon d_j>0$.
    A word in $F$ is in \textbf{normal form} if it is of the form
    \begin{equation*}
        pq^{-1} = (x_{i_1}^{d_1}x_{i_2}^{d_2}\cdots x_{i_n}^{d_n})(x_{j_m}^{-b_m}\cdots x_{j_2}^{-b_2}x_{j_1}^{-b_1})
    \end{equation*} where $p,q$ are MP-words and if $x_{i_k}$ and $x_{j_k}^{-1}$ both appear, then either $x_{i_k+1}$ or $x_{j_k+1}^{-1}$ must also appear.
    This form is unique (\cite[Corollary 2.7]{IntroNotesF}).
\end{definition}

\begin{definition}[$pcq$ form in $T$]
    Every element in $T$ can be written in the form
    \begin{equation*}
        pc_{n+1}^mq^{-1}
    \end{equation*}
    with $p,q$ in $F$, MP-words and $0\leq m\leq n+2$. The $pcq$ form is not unique but can be slightly adapted to become unique. The interested reader can find this in \cite[Theorem 4.2]{BurClearSteinTabackcombMetrpropT}.
\end{definition}

\begin{definition}[$p\pi cq$ form in $V$]
    Every element in $V$ can be written in the form 
    \begin{equation*}
        p\pi c_{n+1}^kq^{-1}
    \end{equation*}
    with $p,q$ in $F$ MP-words,
    $0\leq m\leq n+2$, and $\pi$ an irreducible $\pi$-word
    (this will be defined in \zcref{def:irreducible_pi_word}).
    Just like the $pcq$ form in $T$, this is not unique.
\end{definition}

\subsection{General strategy}
As stated above, it is much easier to reduce words over the infinite presentations than over the finite ones. Therefore, we will find ways to realise van Kampen diagrams for the relators of $\mathscr{T}$ and $\mathscr{V}$ over the finite presentations of $T$ and $V$. The first example of this is \zcref{lem:cn_xncn+1}. It features the relator $c_n=x_nx_{n+1}$ from $\mathscr{V}$, but we construct a van Kampen diagram (\ref{eq:cn_xncn+1}) that contains only definitions and relators/cells over the finite presentation.

\section{\texorpdfstring{Relations in Thompson's group $T$}{Relations in Thompson's group T}}\label{sec:rel_in_T}
We start by discussing some relations in $T$. Migliorini has shown in
\cite[Theorem A]{Migliorini_2025} that the Dehn function of $T$ is quadratic.
Since his proof does not rely on constructing concrete van Kampen diagrams in $T$, and we will need some of these for our constructions in $V$, we will spend some time
discussing them.
Inspired by the infinite presentation of $T$,
we start with a relation that will come in handy to reduce the subscript of
a $c$-letter.

\begin{lemma}\label{lem:cn_xncn+1}
    For any integer $n>0$ one has
    \begin{equation*}
        \|c_n = x_nc_{n+1}\| = 1
    \end{equation*}
\end{lemma}
\begin{proof}
    Consider the following equalities and corresponding van Kampen diagram.
    \begin{align*}
        c_n &= x_0^{-(n-1)}c_1x_1^{n-1}\\
            &= x_0^{-(n-1)}x_1c_2x_1^{n-1}\\
            &= x_0^{-(n-1)}(x_1x_0^{-1}c_1x_1)x_1^{n-1}\\
            &= x_0^{-(n-1)}x_1x_0^{n-1}x_0^{-n}c_1x_1^n = x_nc_{n+1} 
    \end{align*}
    where only the second equality requires a relation. The van Kampen diagram corresponding to this chain of equalities is drawn in (\ref{eq:cn_xncn+1}). 
    We notice that one might count 4 cells in this diagram, but the 3 outer regions are purely definitions, not relations, so they cost nothing.
    \begin{equation}\label{eq:cn_xncn+1}
        \begin{tikzpicture}
            \draw (0:1) -- node[inside] {$c_1$} (120:1) -- node[inside] {$x_1$} (240:1) --  node[inside] {$c_2$} cycle;
            \draw (0:2) -- node[above right] {$c_n$} (120:2) -- node[left] {$x_n$} (240:2) -- node[below right] {$c_{n+1}$} cycle;
            \draw (0:1) -- node[inside, scale=0.5] {$x_1^{n-1}$} (0:2);
            \draw (120:1) -- node[inside, scale=0.5] {$x_0^{n-1}$} (120:2);
            \draw (240:1) -- node[inside, scale=0.5] {$x_0^{n-1}$} (240:2);
        \end{tikzpicture}
    \end{equation}
\end{proof}
Now that we have a way to reduce the index of $c$-letters at the low price of 1 relation, we can do the following.

\begin{lemma}\label{lem:cnxk_xk-1cn+1}
    For any integers $0<k\leq n$ one has
    \begin{equation*}
        \|c_nx_k = x_{k-1}c_{n+1}\| = O(n^2)
    \end{equation*}
\end{lemma}
\begin{proof}
    For $k=1$, we have
    \begin{equation*}
        x_0^{-1}c_nx_1 = c_{n+1}
    \end{equation*}
    by definition.
    For $k=2$ and $n=2$, we find one of the defining relations of $T$.
    \begin{equation*}
        c_2x_2 = x_1c_3
    \end{equation*}
    Now suppose that $k\geq 3$ and consider the van Kampen diagram
    \begin{equation*}
        \begin{tikzpicture}
            \fill[blue!20] (3,3) -- (3,-3) -- (2,-2) -- (2,2) -- cycle;
            \fill[blue!20] (-2,2) -- (2,2) -- (2,-2) -- (-2,-2) -- cycle;
            \draw (-3,3) -- node[above] {$c_{n+1}$} (3,3) -- node[right] {$x_k$} (3,-3) -- node[below] {$c_n$} (-3,-3) -- node[left] {$x_{k-1}$} cycle;
            \draw (-2,2) -- node[above] {$c_{n-k+3}$} (2,2) -- node[right] {$x_2$} (2,-2) -- node[below] {$c_{n-k+2}$} (-2,-2) -- node[left] {$x_1$} cycle;
            \draw (-3,3) -- node[inside] {$x_0^{k-2}$} (-2,2);
            \draw (3,3) -- node[inside] {$x_1^{k-2}$} (2,2);
            \draw (3,-3) -- node[inside] {$x_1^{k-2}$} (2,-2);
            \draw (-3,-3) -- node[inside] {$x_0^{k-2}$} (-2,-2);
        \end{tikzpicture}
    \end{equation*}
    where the highlighted parts are currently the only parts containing cells; the
    rest consists only of definitions.

    Now, using multiple copies of the diagram in \zcref{lem:cn_xncn+1}, the index
    of the $c$-letters can be lowered. The central region can then
    be bridged by a relation and a van Kampen diagram in $F$ with boundary equation
    \begin{equation*}
        (x_2x_3\cdots x_{n-k+1})^{x_1} = x_3x_4\cdots x_{n-k+2}
    \end{equation*}
    Visually, this is given by
    \begin{equation*}\label{eq:x1cn+1=cnx2}
        \begin{tikzpicture}[spy using outlines={circle, magnification=1.75, size=7cm, connect spies}]
            \fill[blue!20] (3,3) -- (3,-3) -- (2,-2) -- (2,2) -- cycle;
            \fill[blue!20] (-2,-2) -- (-2,2) -- (1,1) -- (1,-1) -- cycle;
            \draw (-3,3) -- node[above] {$c_{n+1}$} (3,3) -- node[right] {$x_k$} (3,-3) -- node[below] {$c_n$} (-3,-3) -- node[left] {$x_{k-1}$} cycle;
            \draw (-2,2) -- node[above] {$c_{n-k+3}$} (2,2) -- node[right] {$x_2$} (2,-2) -- node[below] {$c_{n-k+2}$} (-2,-2) -- node[left] {$x_1$} cycle;
            \draw (-3,3) -- node[inside] {$x_0^{k-2}$} (-2,2);
            \draw (3,3) -- node[inside] {$x_1^{k-2}$} (2,2);
            \draw (3,-3) -- node[inside] {$x_1^{k-2}$} (2,-2);
            \draw (-3,-3) -- node[inside] {$x_0^{k-2}$} (-2,-2);
            \begin{scope}[every node/.style={scale=0.5}]
                \draw (-2,2) -- node[below, pos=0.1, sloped] {$x_{n-k+2}$} node[below,pos=0.3, sloped] {$x_{n-k+1}$} node[above, pos=0.7, sloped] {$\cdots$} node[below, pos=0.9] {$x_3$} (1,1);
                \draw (-2,-2) -- node[above, pos=0.1, sloped] {$x_{n-k+1}$} node[above,pos=0.3, sloped] {$x_{n-k}$} node[below, pos=0.7, sloped] {$\cdots$} node[above, pos=0.9] {$x_2$} (1,-1);
                \foreach \p in {0.2,0.4,0.8} {
                    \draw (2,2) -- ($(-2,2)!\p!(1,1)$);
                    \draw (2,-2) -- ($(-2,-2)!\p!(1,-1)$);
                }
                \draw (2,2) -- node[inside] {$c_3$} (1,1);              
                \draw (2,-2) -- node[inside] {$c_2$} (1,-1); 
                \draw (1,1) -- node[right] {$x_1$} (1,-1);             
            \end{scope}
            \spy [red] on (-0.25,1.5) in node at (7.5,0);
        \end{tikzpicture}
    \end{equation*}
    This is a van Kampen diagram with
    \begin{equation}\label{eq:first_step_cnxk_xk-1cn+1}
        2(n-k) + 1 + \|(x_2x_3\cdots x_{n-k+1})^{x_1} = x_3x_4\cdots x_{n-k+2}\| + \|x_2^{x_1^{k-2}} = x_k\|
    \end{equation}
    cells. To utilise the fact that the Dehn function of Thompson's group $F$ is quadratic,
    we need to determine the lengths of the boundary equations of the highlighted regions. We rewrite
    \begin{equation*}
        (x_2x_3\cdots x_{n-k+1})^{x_1} = x_3x_4\cdots x_{n-k+2}
    \end{equation*}
    to
    \begin{equation*}
        x_1^{-1}(x_0^{-1}x_1)^{n-k}x_0^{n-k}x_1 = x_0^{-1}(x_0^{-1}x_1)^{n-k}x_0^{n-k+1}
    \end{equation*}
    via the definitions of the generators (see \zcref{def:definitions_of_infinite_letters}). Thus, this boundary equation has length
    \begin{equation*}
        1 + 2(n-k)+(n-k) + 1 + 1 +2(n-k) + (n-k+1) = 6(n-k)+4
    \end{equation*}
    Analogously, we can rewrite $x_2^{x_1^{k-2}} = x_k$ to
    \begin{equation*}
        x_1^{-(k-2)}x_0^{-1}x_1x_0x_1^{k-2} = x_0^{-(k-1)}x_1x_0^{k-1}
    \end{equation*}
    such that we can conclude a length of $4k-2$ for $x_2^{x_1^{k-2}}=x_k$.
    Therefore, we find that
    \begin{equation*}
        \|c_nx_k = x_{k-1}c_{n+1}\| \leq 2(n-k) + 1 + O((6(n-k)+4)^2) + O((4k-2)^2) = O(n^2)
    \end{equation*}
    where the first inequality follows from \zcref{eq:first_step_cnxk_xk-1cn+1} and the fact that the Dehn function of Thompson's group $F$ is quadratic \cite[Theorem 1]{GubaFQuadratic}.
    For the second equality, we use the fact that $0<k\leq n$.
\end{proof}
Now that we have shown that exchanging $c$-letters and $x$-letters is not too costly, we can tackle
the generalisation of the relation $c_1x_0 = c_2^2$.

\begin{lemma}\label{lem:cnx0_cn+1^2}
    For any integer $0<n$ one has
    \begin{equation*}
                \|c_nx_0 = c_{n+1}^2\| = O(n^2)
    \end{equation*}
\end{lemma}
\begin{proof}
    If $n=1$, this is one of the defining relations of $V$. So suppose that $n\geq 2$ and
    consider the following outline of a van Kampen diagram.
    \begin{center}
        \begin{tikzpicture}[scale=0.8]
            \fill[blue!20] (1,1) -- (3,3) -- (3,-3) -- (1,-1) -- cycle;
            \draw (-3,3) -- node[above] {$x_0$} (3,3) -- node[right] {$c_{n+1}$} (3,-3) -- node[below] {$c_{n+1}$} (-3,-3) -- node[left] {$c_n$} cycle;
            \draw (-1,1) -- node[above] {$x_0$} (1,1) -- node[right] {$c_2$} (1,-1) -- node[below] {$c_2$} (-1,-1) -- node[left] {$c_1$} cycle;
            \draw (-3,3) -- node[inside] {$x_1^{n-1}$} (-1,1);
            \draw (3,3) -- node[inside] {$x_2^{n-1}$} (1,1);
            \draw (3,-3) -- node[inside] {$x_1^{n-1}$} (1,-1);
            \draw (-3,-3) -- node[inside] {$x_0^{n-1}$} (-1,-1);
        \end{tikzpicture}
    \end{center}
    The central region is just the cell $c_1x_0=c_2^2$. Now we show that we can fill each of the surrounding
    regions in an efficient way. The top, left, and bottom regions can be realised without any
    cells as they simply represent definitions. It remains to fill the region on the right.
    To achieve this, we concatenate multiple copies of the diagram in \zcref{lem:cnxk_xk-1cn+1}.
    As we need $n-1$ instances of such a diagram, one would assume to obtain $\sum_{i=2}^nO(i^2)=O(n^3)$
    cells. But it turns out that some cells cancel. To understand this, we observe what happens when
    we concatenate two such diagrams (for some $k\in\mathbb{N}$ such that $3\leq k\leq n$). Observe the diagram
    on the left.
    \begin{center}
        \begin{tikzpicture}[scale=0.8, every node/.style={scale=0.9}]
            \fill[blue!20] ($(-3,-3)!0.8!(2,-1)$) -- (3,-3) -- (2,-5) -- (-3,-3) -- cycle;

            \draw (-3,3) -- node[above] {$c_{k+1}$} (3,3) -- node[right] {$x_2$} (3,-3) -- (-3,-3) -- node[left] {$x_1$} cycle;
            \draw (3,3) -- node[right, scale=0.7] {$c_3$} (2,1) -- node[right] {$x_1$} (2,-1) -- node[right, scale=0.7] {$c_2$} (3,-3);
            \draw (-3,3) -- node[below,pos=0.1] {$x_k$} node[below,pos=0.3] {$x_{k-1}$} node[above,pos=0.7, sloped] {$\cdots$} node[below,pos=0.9] {$x_3$} (2,1);
            \foreach \pos in {0.2,0.4,0.8} {
                \draw (3,3) -- ($(-3,3)!\pos!(2,1)$);
            }
            \draw (-3,-3) -- node[above,pos=0.1] {$x_{k-1}$} node[above,pos=0.3] {$x_{k-2}$} node[below,pos=0.7, sloped] {$\cdots$} node[above,pos=0.9] {$x_2$} (2,-1);
            \foreach \pos in {0.2,0.4,0.8} {
                \draw (3,-3) -- ($(-3,-3)!\pos!(2,-1)$);
            }

            \draw (-3,-3) -- node[fill=blue!20, circle, inner sep=0.2] {$c_{k}$} (3,-3) -- node[right] {$x_2$} (3,-9) -- node[below] {$c_{k-1}$} (-3,-9) -- node[left] {$x_1$} cycle;
            \draw (3,-3) -- node[right, scale=0.7] {$c_3$} (2,-5) -- node[right] {$x_1$} (2,-7) -- node[right, scale=0.7] {$c_2$} (3,-9);
            \draw (-3,-3) -- node[below,pos=0.1] {$x_{k-1}$} node[below,pos=0.3] {$x_{k-2}$} node[above,pos=0.7, sloped] {$\cdots$} node[below,pos=0.9] {$x_3$} (2,-5);
            \foreach \pos in {0.2,0.4,0.8} {
                \draw (3,-3) -- ($(-3,-3)!\pos!(2,-5)$);
            }
            \draw (-3,-9) -- node[above,pos=0.1] {$x_{k-2}$} node[above,pos=0.3] {$x_{k-3}$} node[below,pos=0.7, sloped] {$\cdots$} node[above,pos=0.9] {$x_2$} (2,-7);
            \foreach \pos in {0.2,0.4,0.8} {
                \draw (3,-9) -- ($(-3,-9)!\pos!(2,-7)$);
            }
        \end{tikzpicture}
        \begin{tikzpicture}[scale=0.8, every node/.style={scale=0.9}]
            \fill[pattern=dots, pattern color=gray!40] (-3,3) -- (2,1) -- (2,-7) -- (-3,-9) -- cycle;
            \draw (-3,-3) -- node[left] {$x_1$} (-3,3) -- node[above] {$c_{k+1}$} (3,3) -- node[right] {$x_2$} (3,-3);
            \draw (3,3) -- node[right, scale=0.7] {$c_3$} (2,1) -- node[right] {$x_1$} (2,-1) -- node[right, scale=0.7] {$c_2$} (3,-3);
            \draw (-3,3) -- node[below,pos=0.1] {$x_k$} node[below,pos=0.3] {$x_{k-1}$} node[above,pos=0.7, sloped] {$\cdots$} node[below,pos=0.9] {$x_3$} (2,1);
            \foreach \pos in {0.2,0.4,0.8} {
                \draw (3,3) -- ($(-3,3)!\pos!(2,1)$);
            }

            \draw (3,-3) -- node[right, scale=0.7] {$c_3$} (2,-5) -- node[right] {$x_1$} (2,-7) -- node[right, scale=0.7] {$c_2$} (3,-9);
            \draw (-3,-5) -- node[above,pos=0.1] {$x_{k-1}$} node[above,pos=0.3] {$x_{k-2}$} node[above,pos=0.7, sloped] {$\cdots$} node[above,pos=0.9] {$x_3$} (2,-5);
            \foreach \pos in {0.2,0.4,0.8} {
                \draw (3,-9) -- ($(-3,-9)!\pos!(2,-7)$);
            }
            \draw (-3,-9) -- node[above,pos=0.1] {$x_{k-2}$} node[above,pos=0.3] {$x_{k-3}$} node[below,pos=0.7, sloped] {$\cdots$} node[above,pos=0.9] {$x_2$} (2,-7);
            \draw (-3,-3) -- node[left,pos=0.6] {$x_1$} (-3,-9) -- node[below] {$c_{k-1}$} (3,-9) -- node[right] {$x_2$} (3,-3);
            \draw (2,-1) -- node[right] {$x_2$} (2,-5);
        \end{tikzpicture}
    \end{center}
    The highlighted part is completely symmetric with respect to the edge with label $c_k$,
    so we can collapse it, obtaining the diagram on the right. Now the trick is to consider the central dotted
    region as one diagram in $F$ instead of two. We shall show later that the boundary of this region does
    not grow too large.
    
    Concatenating $n-1$ diagrams together, we can construct the following.
    \begin{center}
        \begin{tikzpicture}
            \draw (-3,3) -- node[above] {$c_{n+1}$} (3,3) -- node[right] {$x_2$} (3,-3);
            \draw (-3,-3) -- node[left] {$x_1$} (-3,3);
            \draw (3,3) -- node[right] {$\scriptstyle c_3$} (2,1) -- node[left] {$x_1$} (2,-1) -- node[right] {$\scriptstyle c_2$} (3,-3);
            \draw (-3,3) -- node[below,pos=0.1] {$x_n$} node[below,pos=0.3] {$x_{n-1}$} node[above,pos=0.7, sloped] {$\cdots$} node[below,pos=0.9] {$x_3$} (2,1);
            \foreach \pos in {0.2,0.4,0.8} {
                \draw (3,3) -- ($(-3,3)!\pos!(2,1)$);
            }

            \draw (3,-3) -- node[right, pos=0.166] {$x_2$} node[right,pos=0.5] {$x_2$} node[right, pos=0.833] {$x_2$} (3,-6);

            \draw (3,-3) -- ($(2,-3)!0.083!(2,-6)$)
                -- ($(2,-3)!0.25!(2,-6)$)
                    -- ($(3,-3)!0.333!(3,-6)$)
                -- ($(2,-3)!0.416!(2,-6)$)
                -- ($(2,-3)!0.583!(2,-6)$)
                    -- ($(3,-3)!0.666!(3,-6)$)
                -- ($(2,-3)!0.75!(2,-6)$)
                -- ($(2,-3)!0.916!(2,-6)$)
                -- (3,-6);

            \draw (2,-1) -- node[left] {$x_2$} (2,-3) -- node[left, pos=0.166] {$x_1$} node[left, pos=0.33] {$x_2$} node[left,pos=0.5] {$x_1$} node[left, pos=0.66] {$x_2$} node[left, pos=0.833] {$x_1$} (2,-6);

            \draw[dotted] (3,-6) -- (3,-7);
            \draw[dotted] (-3,-6) -- (-3,-7);
            \draw[dotted] (2,-6) -- (2,-7);

            \draw (-3,-3) -- node[left, pos=0.166] {$x_1$} node[left,pos=0.5] {$x_1$} node[left, pos=0.833] {$x_1$} (-3,-6);

            \draw (3,-7) -- ($(2,-7)!0.083!(2,-10)$)
                -- ($(2,-7)!0.25!(2,-10)$)
                    -- ($(3,-7)!0.333!(3,-10)$)
                -- ($(2,-7)!0.416!(2,-10)$)
                -- ($(2,-7)!0.583!(2,-10)$)
                    -- (3,-10); 
            \draw (2,-7) -- node[left, pos=0.26] {$x_1$} node[left, pos=0.55] {$x_2$} node[left,pos=0.85] {$x_1$} ($(2,-7)!0.583!(2,-10)$);
            \draw (3,-7) -- node[right, pos=0.166] {$x_2$} node[right,pos=0.666] {$x_2$} (3,-10);

            \draw ($(2,-7)!0.583!(2,-10)$) -- node[above] {$x_2$} (0,-10);
            \draw (3,-10) -- node[above] {$c_3$} (0,-10);
            \draw (3,-10) -- node[right,pos=0.25] {$x_2$} node[right,pos=0.75] {$x_2$} (3,-12)
                -- node[below] {$c_2$} (-3,-12)
                -- node[left,pos=0.25] {$x_1$} node[left, pos=0.75] {$x_1$} (-3,-10);
            \draw (-3,-11) -- node[below] {$c_3$} (3,-11);
            \draw (-3,-11) -- node[above] {$x_2$} (0,-10.5) -- node[above] {$c_2$} (3,-11);
            \draw (0,-10.5) -- node[left] {$x_1$} (0,-10);

            \draw (-3,-7) -- (-3,-10);
        \end{tikzpicture}
    \end{center}
    The area of this diagram is 
    \begin{equation*}
        (n-2) + 2(n-2) +1+ \|x_1^{n-2} = (x_2^{-1}x_1)^{n-2}x_3x_4\cdots x_n\|   
    \end{equation*}
    First, we check that the length of the boundary of the subdiagram in $F$ is not too large, by rewriting the $x_i$'s via the definitions in \zcref{def:definitions_of_infinite_letters}:
    \begin{align*}
        x_1^{n-2} &= (x_2^{-1}x_1)^{n-2}x_3x_4\cdots x_n\tag{\theequation}\label{eq:inner_part_cnx0_cn+12}\\
        \iff x_1^{n-2} &= (x_0^{-1}x_1^{-1}x_0x_1)^{n-2}x_0^{-1}(x_0^{-1}x_1)^{n-2}x_0^{n-1}
    \end{align*}
    We can read that the length of $x_1^{-(n-2)}(x_2^{-1}x_1)^{n-2}x_3x_4\cdots x_n$
    is $8n-14$.
    By Guba's result on Thompson's group $F$, we obtain that the total area of the diagram is bounded by
    \begin{equation*}
        3n-5 + O((8n-14)^2) = O(n^2)
    \end{equation*}
    and thus,
    \begin{equation*}
        \|c_nx_0 = c_{n+1}^2\| = 1 + O(n^2) = O(n^2)
    \end{equation*}
\end{proof}

Now we look at the first hard problem. The ``natural'' van Kampen diagrams will
not give satisfying bounds, so we will need to do some extensive manipulations.

\begin{lemma}\label{lem:c_n^n+2_1}
    For all integers $n\geq 1$,
    \begin{equation*}
        \|c_{n}^{n+2} = 1\| = O(n^3)
    \end{equation*}
\end{lemma}
\begin{proof}
    Let us show the induction step by working out how to prove $c_{n+1}^{n+3} = 1$, starting from $c_{n}^{n+2}=1$.
    For this, consider the van kampen diagram
    \begin{equation*}
        \begin{tikzpicture}
            \begin{scope}[radius=4cm]
                \draw[<-] (0,0) arc[start angle=200, end angle=110] node[pos=0.5, left] {$c_{n+1}$} node (A) {};
                \draw[<-] (A.center) arc[start angle=110, end angle=80] node[pos=0.5, above] {$c_{n+1}$} node (B) {};
                \draw[<-] (B.center) arc[start angle=80, end angle=50] node[pos=0.5, above] {$c_{n+1}$} node (C) {};
                \draw[dotted] (C.center) arc[start angle=50, end angle=-20] node (D) {};
                \draw[<-] (D.center) arc[start angle=-20, end angle=-50] node[pos=0.5, right] {$c_{n+1}$} node (E) {};
                \draw[<-] (E.center) arc[start angle=-50, end angle=-80] node[pos=0.5, below right] {$c_{n+1}$} node (F) {};
                \draw[<-] (F.center) arc[start angle=-80, end angle=-110] node[pos=0.5, below] {$c_{n+1}$} node (G) {};
                \draw[<-] (G.center) arc[start angle=-110, end angle=-160] node[pos=0.5, left] {$c_{n+1}$};
            \end{scope}
            \begin{scope}[radius=2.5cm]
                \draw[<-] (0,0) arc[start angle=200, end angle=90] node[pos=0.5, below right] {$c_n$} node (a) {};
                \draw[<-] (a.center) arc[start angle=90, end angle=60] node[pos=0.5, below] {$c_n$} node (b) {};
                \draw[<-] (b.center) arc[start angle=60, end angle=30] node[pos=0.5, below left] {$c_n$} node (c) {};
                \draw[dotted] (c.center) arc[start angle=30, end angle=-10] node (d) {};
                \draw[<-] (d.center) arc[start angle=-10, end angle=-35] node[pos=0.5, left] {$c_n$} node (e) {};
                \draw[<-] (e.center) arc[start angle=-35, end angle=-60] node[pos=0.5, above left] {$c_n$} node (f) {};
                \draw[<-] (f.center) arc[start angle=-60, end angle=-160] node[pos=0.5, above] {$c_n$};
            \end{scope}
            \begin{scope}[every node/.style={inside}]
                \draw (A.center) -- node {$x_n$} (a.center);
                \draw (B.center) -- node {$x_{n-1}$} (b.center);
                \draw (C.center) -- node {$x_{n-2}$} (c.center);
                \draw (D.center) -- node {$x_2$} (d.center);
                \draw (E.center) -- node {$x_1$} (e.center);
                \draw (F.center) -- node {$x_0$} (f.center);
            \end{scope}
        \end{tikzpicture}
    \end{equation*}
    \vbox{Now, repeating this structure inwards and ``unfolding'' the diagram, one obtains the blueprint for
    $c_n^{n+2}=1$.}
    \begin{equation*}
        \begin{tikzpicture}
            \fill[blue!20] (4,4) -- ++(0,2) -- ++(2,0) -- ++(0,2) -- ++(2,2) -- ++(2,0) -- ++(0,2) -- ++(2,0) -- ++(0,2) -- ++(2,0) -- (14,14) -- (14,4) -- cycle;
            \foreach \x/\y in {2/4,4/6,8/10,10/12} {
                \fill[pattern=checkerboard, pattern color=gray!50] (\x,\y) rectangle ++(2,2);
            }
            \fill[pattern=dots, pattern color=gray] (2,10) rectangle +(4,4);
            \fill[pattern=dots, pattern color=gray] (2,6) rectangle +(2,2);
            \fill[pattern=dots, pattern color=gray] (8,12) rectangle +(2,2);
            
            \begin{scope}[every node/.style={inside}]
                \foreach \x in {2,4,8,10,12} {
                    \draw (\x,14) -- node {$c_n$} +(2,0);
                }
                \foreach \x in {2,4,8,10} {
                    \draw (\x,12) -- node {$c_{n-1}$} +(2,0);
                }
                \foreach \x in {2,4,8} {
                    \draw (\x,10) -- node {$c_{n-2}$} +(2,0);
                }
                \draw (2,8) -- node {$c_3$} (4,8) -- node {$c_3$} (6,8);
                \draw (2,6) -- node {$c_2$} (4,6) -- node {$c_2$} (6,6);
                \draw (2,4) -- node {$c_1$} (4,4);

                \draw (2,4) -- node {$x_1$} (2,6) -- node {$x_2$} (2,8); \draw (2,10) -- node {$x_{n-2}$} (2,12) -- node {$x_{n-1}$} (2,14);
                \draw (4,4) -- node {$x_0$} (4,6) -- node {$x_1$} (4,8); \draw (4,10) -- node {$x_{n-3}$} (4,12) -- node {$x_{n-2}$} (4,14);
                \draw (6,6) -- node {$x_0$} (6,8); \draw (6,10) -- node {$x_{n-4}$} (6,12) -- node {$x_{n-3}$} (6,14);
                \draw (8,10) -- node {$x_1$} (8,12) -- node {$x_2$} (8,14);
                \draw (10,10) -- node {$x_0$} (10,12) -- node {$x_1$} (10,14);
                \draw (12,12) -- node {$x_0$} (12,14);

                \foreach \k/\c in {4/c_1,6/c_2,10/c_{n-2},12/c_{n-1},14/c_n} {
                    \draw (\k,\k) -- node {$\c$} (14,4);
                }
                \foreach \k/\c in {4/c_1,6/c_2,8/c_3,10/c_{n-2},12/c_{n-1},14/c_n} {
                    \draw (2,\k) -- node {$\c$} (0,9);
                }
            \end{scope}
            \begin{scope}[every path/.style={dotted}]
                \draw (2,8) -- (2,10);
                \draw (6,8) -- (6,10);
                \draw (6,10) -- (8,10);
                \draw (6,14) -- (8,14);
                \draw (6,6) -- (10,10);
            \end{scope}
        \end{tikzpicture}
    \end{equation*}
    where the bottom path reads $c_1^3$ which can be closed onto itself since $c_1^3=1$ is a relation in $T$.
    Each of the undecorated triangles on the left side costs 1 relation according to \zcref{lem:cn_xncn+1}, totalling $n-1$ relations.
    The highlighted region is comprised of subdiagrams with boundaries
    \begin{equation*}
        c_{i}x_0=c_{i+1},\qquad (i\in\{1,\dots,n-1\})
    \end{equation*}
    It costs $O(n^3)$, as each subdiagram costs $O(i^2)$ as seen in
    \zcref{lem:cnx0_cn+1^2}. 
    The checkered squares are all definitions. We are left with the triangle of dotted 
    square subdiagrams with boundaries
    \begin{equation*}
        c_ix_k=x_{k-1}c_{i+1},\qquad (i\in\{2,\dots,n-1\},\; k\in\{2,\dots,i\})
    \end{equation*}
    each costing at most $O(n^2)$ (\zcref{lem:cnxk_xk-1cn+1}).
    Simply doing a counting argument on each of these would result in a bound of $O(n^4)$.
    Thus, we need to put in some more effort.
    We fill the dotted square subdiagrams as in \zcref{lem:cnxk_xk-1cn+1}.
    First, we
    take care of the outer rims. To illustrate our argument, we consider the 4 upper left dotted squares.
    \begin{equation*}
        \begin{tikzpicture}
            \fill[blue!20] (-4,0) -- (-3,1) -- (-1,1) -- (0,0) -- (-1,-1) -- (-3,-1) -- cycle;
            \fill[blue!20] (4,0) -- (3,1) -- (1,1) -- (0,0) -- (1,-1) -- (3,-1) -- cycle;
            \draw (-4,4) -- node[above] {$c_n$} (0,4) -- node[above] {$c_n$} (4,4);
            \draw (4,4) -- node[right] {$x_{n-3}$} (4,0) -- node[right] {$x_{n-4}$} (4,-4);
            \draw (4,-4) -- node[below] {$c_{n-2}$} (0,-4) -- node[below] {$c_{n-2}$} (-4,-4);
            \draw (-4,-4) -- node[left] {$x_{n-2}$} (-4,0) -- node[left] {$x_{n-1}$} (-4,4);
            
            \begin{scope}[every node/.style={inside}]
                \draw (-4,4) -- node {$x_1^{n-3}$} (-3,3);
                \draw (-4,0) -- node {$x_1^{n-3}$} (-3,1);
                \draw (0,4) -- node {$x_0^{n-3}$} (-1,3);
                \draw (0,0) -- node {$x_0^{n-3}$} (-1,1);
                \draw (-3,3) -- node {$c_3$} (-1,3) -- node {$x_1$} (-1,1) -- node {$c_2$} (-3,1) -- node {$x_2$} cycle;

                \draw (0,4) -- node {$x_1^{n-4}$} (1,3);
                \draw (4,4) -- node {$x_0^{n-4}$} (3,3);
                \draw (4,0) -- node {$x_0^{n-4}$} (3,1);
                \draw (0,0) -- node {$x_1^{n-4}$} (1,1);
                \draw (1,3) -- node {$c_4$} (3,3) -- node {$x_1$} (3,1) -- node {$c_3$} (1,1) -- node {$x_2$} cycle;

                \draw (-4,0) -- node {$x_1^{n-4}$} (-3,-1);
                \draw (0,0) -- node {$x_0^{n-4}$} (-1,-1);
                \draw (0,-4) -- node {$x_0^{n-4}$} (-1,-3);
                \draw (-4,-4) -- node {$x_1^{n-4}$} (-3,-3);
                \draw (-3,-1) -- node {$c_3$} (-1,-1) -- node {$x_1$} (-1,-3) -- node {$c_2$} (-3,-3) -- node {$x_2$} cycle;

                \draw (0,0) -- node {$x_1^{n-5}$} (1,-1);
                \draw (4,0) -- node {$x_0^{n-5}$} (3,-1);
                \draw (4,-4) -- node {$x_0^{n-5}$} (3,-3);
                \draw (0,-4) -- node {$x_1^{n-5}$} (1,-3);
                \draw (1,-1) -- node {$c_4$} (3,-1) -- node {$x_1$} (3,-3) -- node {$c_3$} (1,-3) -- node {$x_2$} cycle;
            \end{scope}
        \end{tikzpicture}
    \end{equation*}
    Both highlighted subregions are almost fully collapsible. The highlighted regions themselves do not contain any relations, but collapsing them will make the central region shown in the next diagram.
    \begin{equation*}
        \begin{tikzpicture}
            \draw (-4,4) -- node[above] {$c_n$} (0,4) -- node[above] {$c_n$} (4,4);
            \draw (4,4) -- node[right] {$x_{n-3}$} (4,0) -- node[right] {$x_{n-4}$} (4,-4);
            \draw (4,-4) -- node[below] {$c_{n-2}$} (0,-4) -- node[below] {$c_{n-2}$} (-4,-4);
            \draw (-4,-4) -- node[left] {$x_{n-2}$} (-4,0) -- node[left] {$x_{n-1}$} (-4,4);
            
            \begin{scope}[every node/.style={inside}]
                \fill[pattern=dots] (0,4) -- (1,3) -- (1,-3) -- (0,-4) -- (-1,-3) -- (-1,3) -- cycle;
                \draw (-4,4) -- node {$x_1^{n-3}$} (-3,3);
                \draw (0,4) -- node {$x_0^{n-3}$} (-1,3);
                \draw (-3,3) -- node {$c_3$} (-1,3) -- node {$x_1$} (-1,1) -- node {$c_2$} (-3,1) -- node {$x_2$} cycle;

                \draw (0,4) -- node {$x_1^{n-4}$} (1,3);
                \draw (4,4) -- node {$x_0^{n-4}$} (3,3);
                \draw (1,3) -- node {$c_4$} (3,3) -- node {$x_1$} (3,1) -- node {$c_3$} (1,1) -- node {$x_2$} cycle;

                \draw (0,-4) -- node {$x_0^{n-4}$} (-1,-3);
                \draw (-4,-4) -- node {$x_1^{n-4}$} (-3,-3);
                \draw (-3,-1) -- node {$c_3$} (-1,-1) -- node {$x_1$} (-1,-3) -- node {$c_2$} (-3,-3) -- node {$x_2$} cycle;

                \draw (4,-4) -- node {$x_0^{n-5}$} (3,-3);
                \draw (0,-4) -- node {$x_1^{n-5}$} (1,-3);
                \draw (1,-1) -- node {$c_4$} (3,-1) -- node {$x_1$} (3,-3) -- node {$c_3$} (1,-3) -- node {$x_2$} cycle;

                \draw (-3,1) -- node {$x_1$} (-3,-1);
                \draw (-1,1) -- node {$x_0$} (-1,-1);
                \draw (1,1) -- node {$x_1$} (1,-1);
                \draw (3,1) -- node {$x_0$} (3,-1);
            \end{scope}
        \end{tikzpicture}
    \end{equation*}
    The dotted subdiagram has boundary fully in $F$. We will leverage this later. Propagating this
    structure to the whole of the triangular region gives \zcref{fig:c-triangle}.
    \begin{figure}
        \includegraphics[scale=0.825]{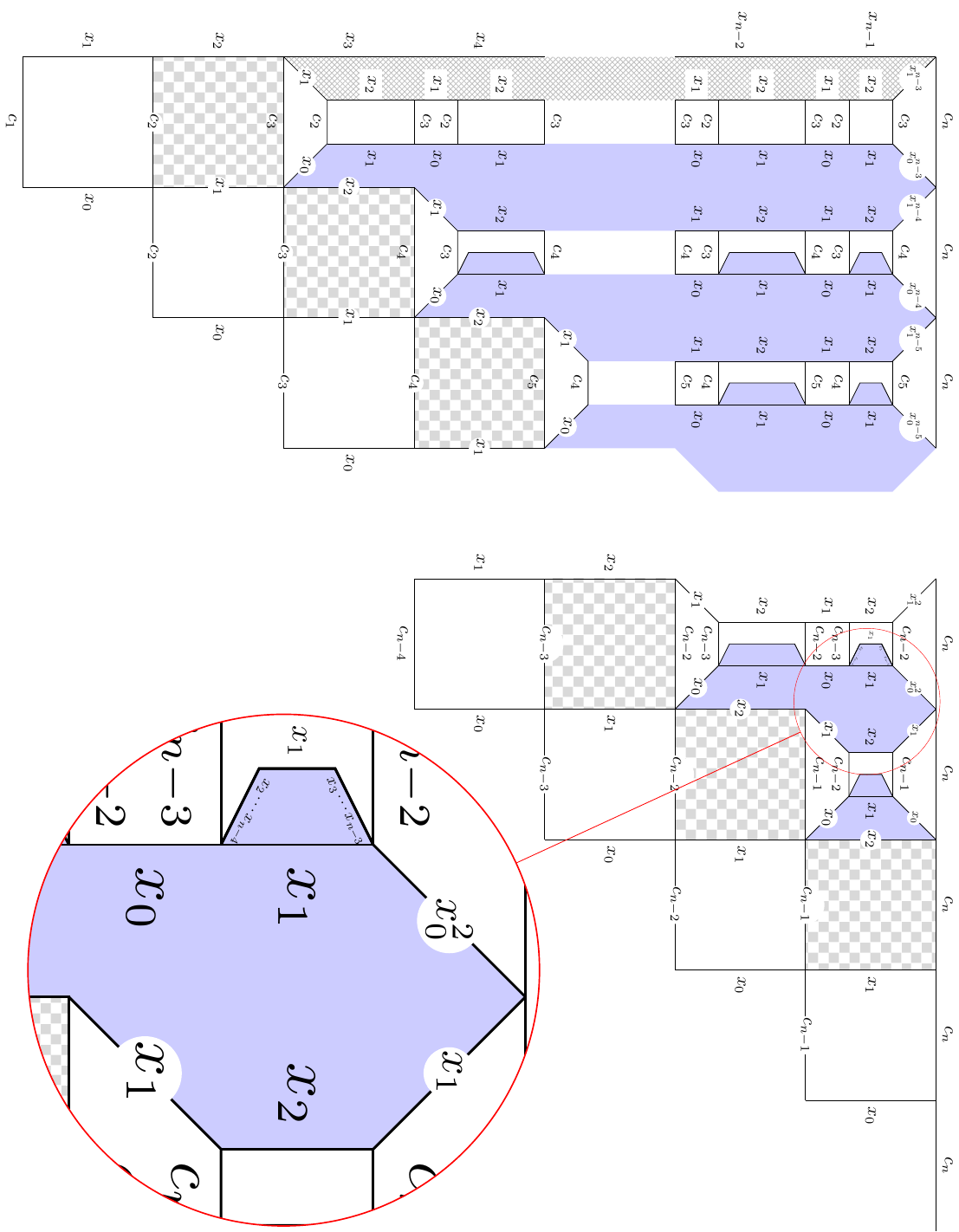}
        \caption{Triangular region in van Kampen diagram for $c_n^{n+2}=1$.}\label{fig:c-triangle}
    \end{figure}
    
    The highlighted subdiagrams are the combinations
    of the dotted subdiagrams from the previous diagram, and the subregions in $F$ found in
    \zcref{lem:cnxk_xk-1cn+1}. The subdiagrams with boundary equations
    \begin{equation*}
        c_{i-1}x_2 = x_1c_i,\qquad (i\in\{3,\dots,n\})
    \end{equation*}
    have not been drawn fully to avoid an illegible picture.
    The relevant info can be found in the magnifying
    glass. First, let us bound the area of the highlighted regions. These have boundary equations
    \begin{equation*}
        \Big(x_0^{-1}(x_2\cdots x_{i-2})^{-1}x_1(x_3\cdots x_{i-1})\Big)^{n-i}x_0^{n-i} = x_2(x_1^{-1}x_2)^{n-i-1}x_1^{n-i-1},\qquad (i\in\{3,\dots,n-1\})
    \end{equation*}
    Sliding the $x_0^{-1}$-letter in the left hand side through the term $(x_2\cdots x_{i-2})^{-1}$ costs no
    relations as it only requires definitions, and yields
    \begin{equation*}
        \Big((x_3\cdots x_{i-1})^{-1}x_0^{-1}x_1(x_3\cdots x_{i-1})\Big)^{n-i}x_0^{n-i} = x_2(x_1^{-1}x_2)^{n-i-1}x_1^{n-i-1},\qquad (i\in\{3,\dots,n-1\})
    \end{equation*}
    such that the expression collapses to
    \begin{equation*}
        (x_3\cdots x_{i-1})^{-1}\big(x_0^{-1}x_1\big)^{n-i}(x_3\cdots x_{i-1})x_0^{n-i} = x_2(x_1^{-1}x_2)^{n-i-1}x_1^{n-i-1},\qquad (i\in\{3,\dots,n-1\})
    \end{equation*}
    Finally, expanding the definitions fully, we get
    \begin{equation*}
        \Big(x_0^{-1}(x_0^{-1}x_1)^{i-2}x_0^{i-2}\Big)^{-1}(x_0^{-1}x_1)^{n-i}\Big(x_0^{-1}(x_0^{-1}x_1)^{i-2}x_0^{i-2}\Big)x_0^{n-i} = x_0^{-1}x_1x_0(x_1^{-1}x_0^{-1}x_1x_0)^{n-i-1}x_1^{n-i-1}
    \end{equation*}
    which has length bounded by
    \begin{equation*}
        2(3(i-2)+1)+3(n-i)+3+5(n-i-1)= 8n-2i-12
    \end{equation*}
    Since the Dehn function of $F$ is quadratic, the combined area of all
    highlighted regions is bounded by
    \begin{equation*}
        \sum_{i=3}^{n-1}O((8n-2i-12)^2) = O(n^3)
    \end{equation*}
    The checkered regions in \zcref{fig:c-triangle} have areas $O((i)^2),\;(i\in\{2,\dots,n-1\})$ according to \zcref{lem:cnxk_xk-1cn+1}.
    Together this gives
    \begin{equation*}
        \sum_{i=2}^{n-1}O(i^2)\leq O(n^3)
    \end{equation*}
    The crosshatched region has boundary equation
    \begin{equation*}
        x_3\cdots x_{n-1} = (x_1^{-1}x_2)^{n-3}x_1^{n-3}
    \end{equation*}
    which is in Thompson's group $F$ and has length
    \begin{equation*}
        1+2(n-3)+n-2 + 4(n-3)+(n-3) = 8n-22
    \end{equation*}
    Therefore, it contributes $O(n^2)$ to the whole diagram.
    It remains to count the cells in the subdiagram with boundary equations
    $c_{i-1}x_2 = x_1c_i$ that were not
    counted in the highlighted and checkered regions. For example, in the second column ($i=4$), 
    there are $n-i$ subdiagrams with boundary equation $c_3x_2 = x_1c_4$.
    Looking back at \zcref{lem:cnxk_xk-1cn+1}, there are $(i-3)+1+(i-3)$ cells
    not counted yet per subdiagram. Therefore totalling $(n-i)(2i-5)$ cells not previously counted.
    Repeating this for every column, we count
    \begin{equation*}
        \sum_{i=3}^{n-1}(n-i)(2i-5) = \sum_{i=3}^{n-1}2ni-2i^2-5n+5i = O(n^3)
    \end{equation*}
    With this, we have shown that the whole triangular region depicted in \zcref{fig:c-triangle} has area
    bounded by $O(n^3)$. This concludes the proof.
\end{proof}

We end this section with some useful relations obtained by chaining multiple relations in $T_\infty$. We will not need
their corresponding van Kampen diagrams explicitly.

\begin{lemma}\label{lem:cm+1sxj}
    For all integers $0\leq m,j$, $0<s$ such that $s\leq m+2$ and $j\leq m+1$, we have
    \begin{equation*}
        c_{m+1}^s = \begin{cases*}
            x_{j-s}c_{m+2}^s & $j\geq s$\\
            c_{m+1}^{s+1} & $j=s-1$\\
            x_{j+m+3-s}c_{m+1}^{s+1} & $j<s-1$
        \end{cases*}
    \end{equation*}
\end{lemma}
\begin{proof}
    This is shown in \cite[Lemma 5.6.i)]{IntroNotesF}.
\end{proof}

\section{\texorpdfstring{Relations in Thompson's group $V$}{Relations in Thompson's group V}}\label{sec:rel_in_V}
The preparations in $T$ are done and now we can tackle $V$.
The amount of relations we need to consider is substantially higher.
The majority of this paper will therefore be devoted to lowering the area
of each of these relations as much as possible.

\subsection{\texorpdfstring{Relations with $x$-letters and $\pi$-letters}{Relations with x-letters and pi-letters}}

We start with the relations mixing $x$-letters and $\pi$-letters. We will often
use induction-like arguments and the basecases sometimes become a bit too complex to leave
within the proof of the general case. They will sometimes be split into a seperate
lemma as is the case for the following lemma.

\begin{lemma}\label{lem:x1pi3x1_pi4}
    \begin{equation*}
        \|\pi_3^{x_1} = \pi_4\| = 6
    \end{equation*}
\end{lemma}
\begin{proof}
    Consider the following van Kampen diagram where every highlighted region is a cell and every white region is
    a definition.

    \begin{equation*}
        \begin{tikzpicture}
            \fill[blue!20] (-1.5,2.7) -- (1.5,2.7) -- (1.5,2.3) -- (-1.5,2.3) -- cycle;
            \fill[blue!20] (-2,0) -- (2,0) -- (2,-2) -- (-2,-2) -- cycle;
            \fill[blue!20] (-4,4) -- (-2,0) -- (-2,-2) -- (-4,-4) -- cycle;
            \fill[blue!20] (4,4) -- (2,0) -- (2,-2) -- (4,-4) -- cycle;
            \fill[blue!20] (-2,-2.5) -- (2,-2.5) -- (2,-3.5) -- (-2,-3.5) -- cycle;

            \draw (-4,4) -- node[above] {$\pi_4$} (4,4) -- node[right] {$x_1$} (4,-4) -- node[below] {$\pi_3$} (-4,-4) -- node[left] {$x_1$} cycle;
            \draw (-2,2) -- node[pos=0.25, below] {$\pi_2$} node[below, pos=0.75] {$x_3$} (2,2)
                -- node[left, pos=0.25] {$x_0$} node[right,pos=0.75]{$x_1$} (2,-2)
                -- node[above, pos=0.25] {$x_3$} node[above, pos=0.75] {$\pi_2$} (-2,-2)
                -- node[left, pos=0.25] {$x_1$} node[right,pos=0.75] {$x_0$} cycle;
            \draw (0,2) -- node[right,pos=0.25] {$x_0$} node[right,pos=0.75] {$x_1$} (0,-2);
            \draw (-2,0) -- node[above,pos=0.25] {$\pi_3$} node[above,pos=0.75] {$x_4$} (2,0);

            \draw (-4,4) -- node[left] {$x_3$} (-2,0);
            \draw (-4,4) -- node[above] {$x_0$} (-2,3);
            \draw (4,4) -- node[right] {$\pi_3$} (2,0);
            \draw (4,4) -- node[above] {$x_0$} (2,3);
            \draw (-4,-4) -- node[above, sloped] {$x_2$} (-2,-2);
            \draw (2,-2) -- node[above, sloped] {$\pi_2$} (4,-4);

            \draw (-2,2) -- node[left] {$x_2$} (-2,3) -- node[above] {$\pi_3$} (2,3) -- node[right] {$\pi_2$} (2,2);

            \draw (-2,-3.5) -- node[left] {$x_1$} (-2,-2.5) -- node[below,pos=0.25] {$\pi_1$} node[below,pos=0.75] {$x_2$} (2,-2.5) -- node[right] {$\pi_1$} (2,-3.5) -- node[above] {$\pi_2$} cycle;
            \draw (-4,-4) -- node[above, sloped] {$x_0$} (-2,-3.5);
            \draw (-2,-2) -- node[right] {$x_0$} (-2,-2.5);
            \draw (0,-2) -- node[right] {$x_0$} (0,-2.5);
            \draw (2,-2) -- node[left] {$x_0$} (2,-2.5);
            \draw (4,-4) -- node[sloped, above] {$x_0$} (2,-3.5);

            \begin{scope}[every node/.style={fill=white, inner sep=0.1, scale=0.7}]
                \draw (-1.5,2.7) -- node {$\pi_2$} (1.5,2.7)
                    -- node {$\pi_1$} (1.5,2.3)
                    -- node[pos=0.25] {$x_2$} node[pos=0.75] {$\scriptstyle \pi_1$} (-1.5,2.3)
                    -- node {$x_1$} cycle;
                \draw (-2,3) -- node {$x_0$} (-1.5,2.7);
                \draw (2,3) -- node {$x_0$} (1.5,2.7);
                \draw (-2,2) -- node {$x_0$} (-1.5,2.3);
                \draw (0,2) -- node {$x_0$} (0,2.3);
                \draw (2,2) -- node {$x_0$} (1.5,2.3);
            \end{scope}
        \end{tikzpicture}
    \end{equation*}
\end{proof}

\begin{lemma}\label{lem:x1pin+1x1_pin+2}
    For any integer $n\geq 1$,
    \begin{equation*}
        \|\pi_{n+1}^{x_1} = \pi_{n+2}\| = O(n^2)
    \end{equation*}
\end{lemma}
\begin{proof}
    For $n=1$, we have a defining relation, for $n=2$, we have \zcref{lem:x1pi3x1_pi4}.
    Now let $n\geq 3$ and consider the following outline of a van Kampen diagram.
    \begin{equation}\label{eq:x1pin+1x1_pin+2_outline}
        \begin{tikzpicture}
            \draw (-4,4) -- node[above] {$\pi_{n+2}$} (4,4)
            -- node[right] {$x_1$} (4,-4)
            -- node[below] {$\pi_{n+1}$} (-4,-4)
            -- node[left] {$x_1$} cycle;
            \draw (-4,4) -- node[below,sloped] {$x_nx_{n-1}\cdots x_4x_3$} (-1,1) -- node[above] {$\pi_4$} (1,1) -- node[below,sloped] {$x_3x_4\cdots x_{n-1}x_{n}$} (4,4);
            \draw (-4,-4) -- node[above,sloped] {$x_{n-1}x_{n-2}\cdots x_3x_2$} (-1,-1) -- node[below] {$\pi_3$} (1,-1) -- node[above,sloped] {$x_2x_3\cdots x_{n-2}x_{n-1}$} (4,-4);
            \draw (-1,-1) -- node[left] {$x_1$} (-1,1);
            \draw (1,-1) -- node[right] {$x_1$} (1,1);
        \end{tikzpicture}
    \end{equation}
    We show that each of the drawn regions can be filled.
    The center is no problem as it is exactly the statement of \zcref{lem:x1pi3x1_pi4}. Therefore,
    only 6 cells are required.

    The top and bottom regions can be filled by stacking consecutive cells with boundary
    \begin{equation*}
        \pi_{i+1}^{x_i} = \pi_{i+2},\quad i\in\{2,\dots,n\}
    \end{equation*}
    \vbox{These stem directly from $\pi_2^{x_1} = \pi_3$ via conjugation with $x_0$.
    This time, no cells can be collapsed but parts of the van Kampen diagram that only contain definitions
    can be kept to a minimum by going from the left diagram to right one in the following pictures.}
    
    \begin{equation*}
        \begin{tikzpicture}[scale=0.75, every node/.style={scale=0.9}]
            \draw (-3,-7) -- node[left] {$x_3$} (-3,-5) -- node[inside] {$\pi_5$} (3,-5) -- node[right] {$x_3$} (3,-7) -- node[below] {$\pi_4$} cycle;
            \draw (-2,-6.5) -- node[left] {$x_1$} (-2,-5.5) -- node[inside] {$\pi_3$} (2,-5.5) -- node[right] {$x_1$} (2,-6.5) -- node[inside] {$\pi_2$} cycle;
            \draw (-3,-7) -- node[inside, scale=0.7] {$x_0^2$} (-2,-6.5);
            \draw (-3,-5) -- node[inside, scale=0.7] {$x_0^2$} (-2,-5.5);
            \draw (3,-5) -- node[inside, scale=0.7] {$x_0^2$} (2,-5.5);
            \draw (3,-7) -- node[inside, scale=0.7] {$x_0^2$} (2,-6.5);

            \draw (-3,-5) -- node[left] {$x_4$} (-3,-3) -- node[inside] {$\pi_6$} (3,-3) -- node[right] {$x_4$} (3,-5);
            \draw (-2,-4.5) -- node[left] {$x_1$} (-2,-3.5) -- node[inside] {$\pi_3$} (2,-3.5) -- node[right] {$x_1$} (2,-4.5) -- node[inside] {$\pi_2$} cycle;
            \draw (-3,-5) -- node[inside, scale=0.7] {$x_0^3$} (-2,-4.5);
            \draw (-3,-3) -- node[inside, scale=0.7] {$x_0^3$} (-2,-3.5);
            \draw (3,-3) -- node[inside, scale=0.7] {$x_0^3$} (2,-3.5);
            \draw (3,-5) -- node[inside, scale=0.7] {$x_0^3$} (2,-4.5);

            \draw (-3,7) -- node[left] {$x_n$} (-3,5) -- node[inside] {$\pi_{n+1}$} (3,5) -- node[right] {$x_n$} (3,7) -- node[above] {$\pi_{n+2}$} cycle;
            \draw (-2,6.5) -- node[left] {$x_1$} (-2,5.5) -- node[inside] {$\pi_2$} (2,5.5) -- node[right] {$x_1$} (2,6.5) -- node[inside] {$\pi_3$} cycle;
            \draw (-3,7) -- node[inside, scale=0.7, rectangle] {$x_0^{n-1}$} (-2,6.5);
            \draw (-3,5) -- node[inside, scale=0.7, rectangle] {$x_0^{n-1}$} (-2,5.5);
            \draw (3,5) -- node[inside, scale=0.7, rectangle] {$x_0^{n-1}$} (2,5.5);
            \draw (3,7) -- node[inside, scale=0.7, rectangle] {$x_0^{n-1}$} (2,6.5);

            \draw (-3,5) -- node[left] {$x_{n-1}$} (-3,3) -- node[inside] {$\pi_n$} (3,3) -- node[right] {$x_{n-1}$} (3,5);
            \draw (-2,4.5) -- node[left] {$x_1$} (-2,3.5) -- node[inside] {$\pi_2$} (2,3.5) -- node[right] {$x_1$} (2,4.5) -- node[inside] {$\pi_3$} cycle;
            \draw (-3,5) -- node[inside, scale=0.7, rectangle] {$x_0^{n-2}$} (-2,4.5);
            \draw (-3,3) -- node[inside, scale=0.7, rectangle] {$x_0^{n-2}$} (-2,3.5);
            \draw (3,3) -- node[inside, scale=0.7, rectangle] {$x_0^{n-2}$} (2,3.5);
            \draw (3,5) -- node[inside, scale=0.7, rectangle] {$x_0^{n-2}$} (2,4.5);

            \draw[dotted] (-3,-3) -- (-3,3);
            \draw[dotted] (3,-3) -- (3,3);
        \end{tikzpicture}
        \begin{tikzpicture}[scale=0.75, every node/.style={scale=0.9}]
            \draw (-3,-7) -- node[above, sloped, rectangle] {$x_3x_4\cdots x_{n-1}x_n$} (-3,7)
                -- node[above] {$\pi_{n+2}$} (3,7)
                -- node[above, sloped, rectangle] {$x_nx_{n-1}\cdots x_4x_3$} (3,-7)
                -- node[below] {$\pi_4$} cycle;
            \draw (-3,-7) -- node[inside, scale=0.7] {$x_0^2$} (-2,-6.5);
            \draw (-3,7) -- node[inside, scale=0.7, rectangle] {$x_0^{n-1}$} (-2,6.5);
            \draw (3,7) -- node[inside, scale=0.7, rectangle] {$x_0^{n-1}$} (2,6.5);
            \draw (3,-7) -- node[inside, scale=0.7, rectangle] {$x_0^2$} (2,-6.5);
            \draw (-2,6.5) -- node[inside] {$\pi_3$} (2,6.5);
            \draw (-2,-6.5) -- node[inside] {$\pi_2$} (2,-6.5);
            \draw (-2,-6.5) -- (-2,6.5);
            \draw (2,-6.5) -- (2,6.5);

            \foreach \p in {0.1,0.3,0.6,0.8} {
                \draw ($(-2,-6.5)!\p!(-2,6.5)$) -- node[inside] {$\pi_3$} ($(2,-6.5)!\p!(2,6.5)$);
            }
            \foreach \p in {0.2,0.4,0.7,0.9} {
                \draw ($(-2,-6.5)!\p!(-2,6.5)$) -- node[inside] {$\pi_2$} ($(2,-6.5)!\p!(2,6.5)$);
            }
            \foreach \p in {0.05,0.25,0.95,0.75} {
                \node[scale=0.8] at ($(-2.3,-6.5)!\p!(-2.3,6.5)$) {$x_1$};
                \node[scale=0.8] at ($(2.3,-6.5)!\p!(2.3,6.5)$) {$x_1$};
            }
            \foreach \p in {0.15,0.35,0.85,0.65} {
                \node[scale=0.8] at ($(-2.3,-6.5)!\p!(-2.3,6.5)$) {$x_0^{-1}$};
                \node[scale=0.8] at ($(2.3,-6.5)!\p!(2.3,6.5)$) {$x_0^{-1}$};
            }

            \node at (0,0) {$\vdots$};
        \end{tikzpicture}
    \end{equation*}
    where every vertical edge is oriented from bottom to top.
    There are $n-2$ cells in this diagram, and by (shifted) symmetry, $n-2$ in the diagram for
    the bottom region of (\ref{eq:x1pin+1x1_pin+2_outline}).
    
    The left and right region of (\ref{eq:x1pin+1x1_pin+2_outline}) have boundaries in $F$.
    To use the result for $F$ by Guba, we determine the length of the boundary
    \begin{align*}
        (x_2x_3\cdots x_{n-1})^{x_1} &= x_3x_4\cdots x_n\\
        \iff ((x_0^{-1}x_1)^{n-2}x_0^{n-2})^{x_1} &= x_0^{-1}(x_0^{-1}x_1)^{n-2}x_0^{n-1}
    \end{align*}
    This boundary equation has length
    \begin{equation*}
        2(n-2) + (n-2) + 2 + 1 +2(n-2) + (n-1) = 6n-8
    \end{equation*}
    By Guba's result, we know that the left (and the right) region can be filled with $O((6n-8)^2) = O(n^2)$
    cells. Therefore, we can conclude that the whole van Kampen diagram (\ref{eq:x1pin+1x1_pin+2_outline})
    can be realised with
    \begin{equation*}
        6 + 2(n-2) + 2O(n^2) = O(n^2)
    \end{equation*}
    cells. A visual recap is given by the following diagram.
    
    \begin{equation*}
        \begin{tikzpicture}
            \fill[blue!20] (-5,5) -- (-1,1) -- (-1,-1) -- (-5,-5) -- cycle;
            \fill[blue!20] (5,5) -- (1,1) -- (1,-1) -- (5,-5) -- cycle;

            \draw (-5,5) -- node[above] {$\pi_{n+2}$} (5,5) -- node[right] {$x_1$} (5,-5) -- node[below] {$\pi_{n+1}$} (-5,-5) -- node[left] {$x_1$} cycle;
            \draw (-1,1) -- node[below] {$\pi_4$} (1,1) -- node[right] {$x_1$} (1,-1) -- node[above] {$\pi_3$} (-1,-1) -- node[left] {$x_1$} cycle;
            \draw (-5,5) -- node[sloped, below] {$x_n\cdots x_3$} (-1,1);
            \draw (5,5) -- node[sloped, below] {$x_3\cdots x_n$} (1,1);
            \draw (5,-5) -- node[sloped, above] {$x_2\cdots x_{n-1}$} (1,-1);
            \draw (-5,-5) -- node[sloped, above] {$x_{n-1}\cdots x_2$} (-1,-1);

            \draw (-1,1) -- node[inside] {$\scriptstyle x_0^2$} (-0.5,1.5) -- node[inside] {$\scriptstyle \pi_2$} (0.5,1.5) -- node[inside] {$\scriptstyle x_0^2$} (1,1);
            \draw (-0.5,1.5) -- node[sloped, above] {$\scriptstyle x_1(x_0^{-1}x_1)^{n-3}$} (-0.5,4) -- node[above] {$\pi_3$} (0.5,4);
            \draw (0.5,1.5) -- node[sloped, below] {$\scriptstyle x_1(x_0^{-1}x_1)^{n-3}$} (0.5,4);
            \draw (-5,5) -- node[below] {$x_0^{n-1}$} (-0.5,4);
            \draw (5,5) -- node[below] {$x_0^{n-1}$} (0.5,4);

            \draw (-1,-1) -- node[inside] {$\scriptstyle x_0$} (-0.5,-1.5) -- node[inside] {$\scriptstyle \pi_2$} (0.5,-1.5) -- node[inside] {$\scriptstyle x_0$} (1,-1);
            \draw (-0.5,-1.5) -- node[sloped, below] {$\scriptstyle x_1(x_0^{-1}x_1)^{n-3}$} (-0.5,-4) -- node[below] {$\pi_3$} (0.5,-4);
            \draw (0.5,-1.5) -- node[sloped, above] {$\scriptstyle x_1(x_0^{-1}x_1)^{n-3}$} (0.5,-4);
            \draw (-5,-5) -- node[above] {$x_0^{n-2}$} (-0.5,-4);
            \draw (5,-5) -- node[above] {$x_0^{n-2}$} (0.5,-4);

            \foreach \y in {2,2.5,3,3.5} {
                \draw (-0.5,\y) -- (0.5,\y);
                \draw (-0.5,-\y) -- (0.5,-\y);
            }
        \end{tikzpicture}
    \end{equation*}
\end{proof}

\begin{lemma}\label{lem:pijxi_pij+1}
    For any integers $0\leq i<j$,
    \begin{equation*}
        \|\pi_j^{x_i} = \pi_{j+1}\| = O(n^2)
    \end{equation*}
    where $n = j-i$.
\end{lemma}
\begin{proof}
    If $i=0$, this is a definition. If $i=1$, $j=2$, this is a defining relation of $V$. Therefore, we suppose
    that $2\leq i<j$. Now consider the van Kampen diagram
    \begin{equation*}
        \begin{tikzpicture}
            \draw (-2,2) -- node[above] {$\pi_{j+1}$} (2,2)
            -- node[right] {$x_i$} (2,-2)
            -- node[below] {$\pi_{j}$} (-2,-2)
            -- node[left] {$x_i$} cycle;
            \draw (-1,1) -- node[inside, scale=0.8] {$\pi_{j-i+2}$} (1,1)
            -- node[inside, scale=0.8] {$x_1$} (1,-1)
            -- node[inside, scale=0.8] {$\pi_{j-i+1}$} (-1,-1)
            -- node[inside, scale=0.8] {$x_1$} cycle;
            \draw (-2,-2) -- node[inside, scale=0.8] {$x_0^{i-1}$} (-1,-1);
            \draw (-2,2) -- node[inside, scale=0.8] {$x_0^{i-1}$} (-1,1);
            \draw (2,2) -- node[inside, scale=0.8] {$x_0^{i-1}$} (1,1);
            \draw (2,-2) -- node[inside, scale=0.8] {$x_0^{i-1}$} (1,-1);
        \end{tikzpicture}
    \end{equation*}
    and notice that the regions around the central region simply stem from the definitions and thus contain no
    real cells. The central region can then be filled with \zcref{lem:x1pin+1x1_pin+2} in a quadratic
    amount of cells.
\end{proof}

\begin{lemma}\label{lem:x4pi1_pi1x4}
    \begin{equation*}
        \|x_4\pi_1 = \pi_1x_4\| = 6
    \end{equation*}
\end{lemma}
\begin{proof}
    Again the following van Kampen diagram suffices. The only cells are highlighted.
        \begin{equation*}
        \begin{tikzpicture}
            \fill[blue!20] (-3.5,2.5) -- (-2,1.5) -- (-2,-1.5) -- (-3.5,-2.5) -- cycle;
            \fill[blue!20] (3.5,2.5) -- (2,1.5) -- (2,-1.5) -- (3.5,-2.5) -- cycle;
            \fill[blue!20] (-1.5,-2) rectangle (1.5,2);
            \fill[blue!20] (-4,-4) -- (-1.5,-2) -- (1.5,-2) -- (4,-4) -- cycle;
            \fill[blue!20] (-4,4) -- (-1.5,2) -- (1.5,2) -- (4,4) -- cycle;

            \draw (-4,4) -- node[above] {$\pi_1$} (4,4) -- node[right] {$x_4$} (4,-4) -- node[below] {$\pi_1$} (-4,-4) -- node[left] {$x_4$} cycle;
            \draw (-4,4) -- node[inside] {$\pi_2$} (-1.5,2) -- node[above] {$x_1$} (0,2) -- node[above] {$\pi_1$} (1.5,2) -- node[inside] {$x_2$} (4,4);
            \draw (-4,-4) -- node[inside] {$\pi_2$} (-1.5,-2) -- node[below] {$x_1$} (0,-2) -- node[below] {$\pi_1$} (1.5,-2) -- node[inside] {$x_2$} (4,-4);
            \draw (-1.5,2) -- node[right] {$x_4$} (-1.5,-2);
            \draw (0,2) -- node[inside, fill=blue!20] {$x_3$} (0,-2);
            \draw (1.5,2) -- node[left] {$x_3$} (1.5,-2);
            \draw (-3.5,2.5) -- node[below] {$\pi_1$} (-2,1.5) -- node[left] {$x_3$} (-2,-1.5) -- node[above] {$\pi_1$} (-3.5,-2.5) -- node[right] {$x_3$} cycle;
            \draw (-4,4) -- node[inside] {$x_0$} (-3.5,2.5);
            \draw (-1.5,2) -- node[inside] {$x_0$} (-2,1.5);
            \draw (-1.5,-2) -- node[inside] {$x_0$} (-2,-1.5);
            \draw (-4,-4) -- node[inside] {$x_0$} (-3.5,-2.5);

            \draw (3.5,2.5) -- node[below] {$\pi_1$} (2,1.5) -- node[right] {$x_3$} (2,-1.5) -- node[above] {$\pi_1$} (3.5,-2.5) -- node[left] {$x_3$} cycle;
            \draw (4,4) -- node[inside] {$x_0$} (3.5,2.5);
            \draw (1.5,2) -- node[inside] {$x_0$} (2,1.5);
            \draw (1.5,-2) -- node[inside] {$x_0$} (2,-1.5);
            \draw (4,-4) -- node[inside] {$x_0$} (3.5,-2.5);
        \end{tikzpicture}
    \end{equation*}
\end{proof}

\begin{lemma}\label{lem:xnpi1_pi1xn}
    For any integer $n\geq 3$,
    \begin{equation*}
        \|x_n\pi_1 = \pi_1x_n\| = O(n^2)
    \end{equation*}
\end{lemma}
\begin{proof}
    If $n=3$, we have the defining relation $x_3\pi_1 = \pi_1x_3$. If $n=4$, we have
    \zcref{lem:x4pi1_pi1x4}. Now let $n\geq 5$ and consider the van Kampen diagram
    \begin{equation*}
        \begin{tikzpicture}[scale=0.7]
            \fill[blue!20] (-5,5) -- (5,5) -- (3,3) -- (-3,3) -- cycle;
            \fill[blue!20] (-5,-5) -- (5,-5) -- (3,-3) -- (-3,-3) -- cycle;

            \draw (-5,5) -- node[above] {$x_n$} (5,5) -- node[right] {$\pi_1$} (5,-5) -- node[below] {$x_n$} (-5,-5) -- node[left] {$\pi_1$} cycle;
            \draw (-3,3) -- node[below] {$x_4$} (3,3) -- node[left] {$\pi_1$} (3,-3) -- node[above] {$x_4$} (-3,-3) -- node[right] {$\pi_1$} cycle;
            \draw (-5,5) -- node[above right, scale=0.8] {$x_3^{n-4}$} (-3,3);
            \draw (5,5) -- node[above left, scale=0.8] {$x_3^{n-4}$} (3,3);
            \draw (5,-5) -- node[below left, scale=0.8] {$x_3^{n-4}$} (3,-3);
            \draw (-5,-5) -- node[below right, scale=0.8] {$x_3^{n-4}$} (-3,-3);

            \foreach \p in {0.2,0.4,...,0.8} {
                \draw ($(-5,5)!\p!(-3,3)$) -- ($(-5,-5)!\p!(-3,-3)$);
                \draw ($(5,5)!\p!(3,3)$) -- ($(5,-5)!\p!(3,-3)$);
            }
        \end{tikzpicture}
    \end{equation*}
    The left and right region can be realised with each $n-4$ copies of the cell with boundary $\pi_1x_3=x_3\pi_1$.
    The middle region takes 6 cells as seen in \zcref{lem:x4pi1_pi1x4}. It remains to consider the top
    and bottom region. They have the same boundary equation, fully in $F$. We determine its length via the definitions in \zcref{def:definitions_of_infinite_letters}.
    \begin{align*}
        (x_3^{n-4})^{-1}x_4(x_3^{n-4}) &= x_n\\
        \iff (x_0^{-2}x_1^{n-4}x_0^{2})^{-1}(x_0^{-3}x_1x_0^{3})(x_0^{-2}x_1^{n-4}x_0^{2}) &= x_0^{-(n-1)}x_1x_0^{n-1}\\
        \iff x_1^{n-4}x_0^{-1}x_1x_0x_1^{n-4} &= x_0^{-(n-3)}x_1x_0^{n-3}
    \end{align*}
    The length is thus
    \begin{equation*}
        (n-4) + 1 +1 +1 +(n-4) + (n-3) + 1 +(n-3) = 4n-10
    \end{equation*}
    As $F$ has quadratic Dehn function, it takes $O((4n-10)^2) = O(n^2)$ cells to fill each
    of the remaining regions. In total, we get
    \begin{equation*}
        2(n-4) + 6 + 2O(n^2) = O(n^2)
    \end{equation*}
    cells for the whole van Kampen diagram.
\end{proof}

\begin{lemma}\label{lem:xjpii_piixj}
    For any integers $j\geq i+2>2$,
    \begin{equation*}
        \|x_j\pi_i = \pi_ix_j\| = O(n^2)
    \end{equation*}
    with $n = j-i+1$.
\end{lemma}
\begin{proof}
    The case $i=1$ is handled by \zcref{lem:xnpi1_pi1xn}, so suppose that $i\geq 2$.
    Consider the van Kampen diagram
    \begin{equation*}
        \begin{tikzpicture}
            \draw (-2,2) -- node[above] {$x_j$} (2,2)
            -- node[right] {$\pi_i$} (2,-2)
            -- node[below] {$x_j$} (-2,-2)
            -- node[left] {$\pi_i$} cycle;
            \draw (-2,2) -- node[inside, rectangle] {$\scriptstyle x_0^{i-1}$} (-1,1) -- node[above] {$x_{j-i+1}$} (1,1) -- node[inside, rectangle] {$\scriptstyle x_0^{i-1}$} (2,2);
            \draw (-2,-2) -- node[inside, rectangle] {$\scriptstyle x_0^{i-1}$} (-1,-1) -- node[below] {$x_{j-i+1}$} (1,-1) -- node[inside, rectangle] {$\scriptstyle x_0^{i-1}$} (2,-2);
            \draw (-1,-1) -- node[left] {$\pi_1$} (-1,1);
            \draw (1,-1) -- node[right] {$\pi_1$} (1,1);
        \end{tikzpicture}
    \end{equation*}
    We remark again that the only cells in this van Kampen diagram are in the center region.
    \zcref{lem:xnpi1_pi1xn}
    tells us that we can fill this with $O(n^2)$ cells.
\end{proof}

\subsection{\texorpdfstring{Relations with $c$-letters and $\pi$-letters}{Relations with c-letters and pi-letters}}
We just considered relations containing $x$-letters and $\pi$-letters.
Now we continue with relations
containing $\pi$-letters and $c$-letters.

\begin{lemma}\label{lem:cnpi2_pi1cn}
    For any integer $n\geq 3$,
    \begin{equation*}
        \|\pi_1^{c_n} = \pi_2\| = O(n^2)
    \end{equation*}
\end{lemma}
\begin{proof}
    If $n=3$, this is a defining relation of $V$. Let $n\geq 4$ and consider the following van Kampen diagram.

    \begin{equation*}
        \begin{tikzpicture}[scale=0.7, every node/.style={scale=0.7}]
            \draw (-4,4) -- node[above] {$\pi_2$} (4,4)
            -- node[right] {$c_n$} (4,-4)
            -- node[below] {$\pi_1$} (-4,-4)
            -- node[left] {$c_n$} cycle;
            \draw (-4,4) -- node[above] {$c_3$} (-1,1) -- node[above] {$\pi_1$} (1,1) -- node[above] {$c_3$} (4,4);
            \draw (-4,-4) -- node[below, sloped] {$x_n\cdots x_4x_3$} (-1,1);
            \draw (4,-4) -- node[below, sloped] {$x_3x_4\cdots x_n$} (1,1);
            \draw (-4,4) -- node[right] {$c_{n-1}$} ($(-4,-4)!0.2!(-1,1)$);
            \draw (-4,4) -- node[left] {$c_4$} ($(-4,-4)!0.8!(-1,1)$);

            \draw (4,4) -- node[left] {$c_{n-1}$} ($(4,-4)!0.2!(1,1)$);
            \draw (4,4) -- node[right] {$c_4$} ($(4,-4)!0.8!(1,1)$);

            \node at (-3,1.1) {$\cdots$};
            \node at (3,1.1) {$\cdots$};
        \end{tikzpicture}
    \end{equation*}
    The top region is simply a cell. The left and right regions are built using $n-3$ cells each (\zcref{lem:cn_xncn+1}).
    It remains to construct the bottom region. We will do this via the van Kampen diagrams
    seen in \zcref{lem:xnpi1_pi1xn}.
    If we stack two such diagrams on top of each other, we get
    \begin{equation*}
        \begin{tikzpicture}
            \fill[blue!20] (-2,-2) -- (-0.5,-0.5) -- (0.5, -0.5) -- (2,-2) -- (0.5,-3.5) -- (-0.5,-3.5) -- cycle;

            \draw (-2,2) -- node[above] {$\pi_1$} (2,2) -- node[right] {$x_i$} (2,-2) -- node[below] {$\pi_1$} (-2,-2) -- node[left] {$x_i$} cycle;
            \draw (-2,2) -- node[inside] {$x_3^{i-4}$} (-0.5,0.5) -- node[above] {$\pi_1$} (0.5,0.5) -- node[inside] {$x_3^{i-4}$} (2,2);
            \draw (-2,-2) -- node[inside] {$x_3^{i-4}$} (-0.5,-0.5) -- node[below] {$\pi_1$} (0.5,-0.5) -- node[inside] {$x_3^{i-4}$} (2,-2);
            \draw (-0.5,0.5) -- node[left] {$x_4$} (-0.5,-0.5);
            \draw (0.5,0.5) -- node[right] {$x_4$} (0.5,-0.5);

            \draw (-2,-2) -- node[inside] {$x_3^{i-5}$} (-0.5,-3.5) -- node[above] {$\pi_1$} (0.5,-3.5) -- node[inside] {$x_3^{i-5}$} (2,-2);
            \draw (-2,-6) -- node[inside] {$x_3^{i-5}$} (-0.5,-4.5) -- node[below] {$\pi_1$} (0.5,-4.5) -- node[inside] {$x_3^{i-5}$} (2,-6);
            \draw (-0.5,-3.5) -- node[left] {$x_4$} (-0.5,-4.5);
            \draw (0.5,-3.5) -- node[right] {$x_4$} (0.5,-4.5);
            \draw (-2,-2) -- node[left] {$x_{i-1}$} (-2,-6) -- node[below] {$\pi_1$} (2,-6) -- node[right] {$x_{i-1}$} (2,-2);
            \node at (3,-2) {$\mbox{\Large $\leadsto$}$};
        \end{tikzpicture}
        \begin{tikzpicture}
            \fill[blue!20] (-0.5,-0.5) -- (0.5,-0.5) -- (0.5,-2) -- (-0.5,-2) -- cycle;

            \draw (-2,2) -- node[above] {$\pi_1$} (2,2) -- node[right] {$x_i$} (2,-2) -- node[below] {$\pi_1$} (-2,-2) -- node[left] {$x_i$} cycle;
            \draw (-2,2) -- node[inside] {$x_3^{i-4}$} (-0.5,0.5) -- node[above] {$\pi_1$} (0.5,0.5) -- node[inside] {$x_3^{i-4}$} (2,2);
            \draw (-0.5,0.5) -- node[left] {$x_4$} (-0.5,-0.5);
            \draw (0.5,0.5) -- node[right] {$x_4$} (0.5,-0.5);

            \draw (-0.5,-0.5) -- node[below] {$\pi_1$} (0.5,-0.5) -- node[right] {$x_3^{-1}$} (0.5,-2) -- (-0.5,-2) -- node[left] {$x_3^{-1}$} cycle;

            \draw (-2,-6) -- node[inside] {$x_3^{i-5}$} (-0.5,-4.5) -- node[below] {$\pi_1$} (0.5,-4.5) -- node[inside] {$x_3^{i-5}$} (2,-6);
            \draw (-0.5,-2) -- node[left] {$x_4$} (-0.5,-4.5);
            \draw (0.5,-2) -- node[right] {$x_4$} (0.5,-4.5);
            \draw (-2,-2) -- node[left] {$x_{i-1}$} (-2,-6) -- node[below] {$\pi_{1}$} (2,-6) -- node[right] {$x_{i-1}$} (2,-2);

            \draw (-2,-2) -- node[below] {$x_3^{i-5}$} (-0.5,-2);
            \draw (0.5,-2) -- node[below] {$x_3^{i-5}$} (2,-2);
        \end{tikzpicture}
    \end{equation*}
    where every vertical edge is oriented upwards and we cancelled all but one of the cells in the highlighted region.
    
    Repeating this pattern, one can get the desired van Kampen diagram.
    \begin{equation*}
        \begin{tikzpicture}
            \fill[pattern=dots] (-2,5) -- (2,5) -- (0.5,3) -- (-0.5,3) -- cycle;
            \fill[pattern=checkerboard, pattern color=gray!30] (2,5) -- (2,-2.33) -- (0.5,-2.33) -- (0.5,3) -- cycle;
            \fill[pattern=checkerboard, pattern color=gray!30] (-2,5) -- (-2,-2.33) -- (-0.5,-2.33) -- (-0.5,3) -- cycle;
            \fill[blue!20] (-0.5,3) rectangle (0.5,-4);
            \draw (-2,5) -- node[above] {$\pi_1$} (2,5);
            \draw (-2,5) -- node[inside] {$x_3^{n-4}$} (-0.5,3) -- node[above, inside] {$\pi_1$} (0.5,3) -- node[inside] {$x_3^{n-4}$} (2,5);

            \draw (-0.5,3) -- node[left, pos=0.2] {$x_4$} node[left, pos=0.5] {$x_3^{-1}$} node[left, pos=0.8] {$x_4$} (-0.5,1);
            \draw (0.5,3) -- node[right,pos=0.2] {$x_4$} node[right,pos=0.5] {$x_3^{-1}$} node[right,pos=0.8] {$x_4$} (0.5,1);
            \draw ($(-0.5,3)!0.33!(-0.5,1)$) -- node[above=-3] {$\pi_1$} ($(0.5,3)!0.33!(0.5,1)$);
            \draw ($(-0.5,3)!0.66!(-0.5,1)$) -- node[above=-3] {$\pi_1$} ($(0.5,3)!0.66!(0.5,1)$);

            \draw[dotted] (-0.5,1) -- (-0.5,-1);
            \draw[dotted] (0.5,1) -- (0.5,-1);

            \draw (-0.5,-1) -- node[left, pos=0.2] {$x_4$} node[left, pos=0.5] {$x_3^{-1}$} node[left, pos=0.8] {$x_4$} (-0.5,-3);
            \draw (0.5,-1) -- node[right,pos=0.2] {$x_4$} node[right,pos=0.5] {$x_3^{-1}$} node[right,pos=0.8] {$x_4$} (0.5,-3);
            \draw (-0.5,-3) -- node[below] {$\pi_1$} (0.5,-3);
            \draw ($(-0.5,-3)!0.33!(-0.5,-1)$) -- node[below=-3] {$\pi_1$} ($(0.5,-3)!0.33!(0.5,-1)$);
            \draw ($(-0.5,-3)!0.66!(-0.5,-1)$) -- node[below=-3] {$\pi_1$} ($(0.5,-3)!0.66!(0.5,-1)$);
            \draw (-0.5,-3) -- node[left] {$x_3$} (-0.5,-4) -- node[below] {$\pi_1$} (0.5,-4) -- node[right] {$x_3$} (0.5,-3);
            \draw[swap] ($(-0.5,-3)!0.33!(-0.5,-1)$) -- ++(-1.5,0) -- node[above, rotate=180, sloped] {$x_5x_6\cdots x_n$} (-2,5);
            \draw ($(0.5,-3)!0.33!(0.5,-1)$) -- ++(1.5,0) -- node[above, rotate=180, sloped] {$x_n\cdots x_6x_5$} (2,5);
        \end{tikzpicture}
    \end{equation*}
    Again, every vertical edge is oriented upwards.
    The left and right checkered regions have boundaries in $F$. We deduce their length.
    \begin{align*}
        &x_5x_6\cdots x_n = (x_3^{-1}x_4)^{n-4}x_3^{n-4}\tag{\theequation}\label{eq:sides_of_pi1cn_pi2}\\
        \iff &x_0^{-3}(x_0^{-1}x_1)^{n-4}x_0^{n-1} = x_0^{-2}(x_1^{-1}x_0^{-1}x_1x_0)^{n-4}x_0^2x_0^{-2}x_1^{n-4}x_0^2\\
        \iff &x_0^{-1}(x_0^{-1}x_1)^{n-4}x_0^{n-3} = (x_1^{-1}x_0^{-1}x_1x_0)^{n-4}x_1^{n-4}
    \end{align*}
    which gives a length of
    \begin{equation*}
        1 + 2(n-4) + (n-3) + 4(n-4) + (n-4) = 8n-30
    \end{equation*}
    Again, Guba's result gives an area of $O((8n-30)^2) = O(n^2)$ for each region.
    The highlighted strip is formed by alternating the cells with boundary $x_3\pi_1=\pi_1x_3$ and the regions with boundary $x_4\pi_1=\pi_1x_4$, each appearing $(n-3)$ times.
    The cells of course take 1 relation, and using \zcref{lem:x4pi1_pi1x4} we can fill these subregions with $6$ cells. We conclude that the highlighted strip takes
    $(n-3) + 6(n-3) = 7n-21$ cells. Finally, the dotted region takes $n-4$ cells.
    The van Kampen diagram with boundary equation $(x_3\cdots x_n)^{\pi_1} = x_3\cdots x_n$
    is thus constructed with
    \begin{equation*}
        (n-4) + (7n-21) + 2O(n^2) = O(n^2)
    \end{equation*}
    cells. Plugging this diagram in the one with boundary $\pi_1^{c_n} = \pi_2$,
    gives the diagram below with area
    \begin{equation*}
        1 + 2(n-3) + (n-4) + (7n-21) + 2O(n^2) = O(n^2)
    \end{equation*}
        \begin{equation}\label{diag:pi1cn_pi2}
        \begin{tikzpicture}
            \fill[pattern=dots, pattern color=gray!70] (-5,-5) -- ($(-5,-5)!0.8!(-1,1)$) -- (-1,-1) -- (-1,-4) -- cycle;
            \fill[pattern=dots, pattern color=gray!70] (5,-5) -- ($(5,-5)!0.8!(1,1)$) -- (1,-1) -- (1,-4) -- cycle;
            \draw (-5,5) -- node[above] {$\pi_2$} (5,5)
            -- node[right] {$c_n$} (5,-5)
            -- node[below] {$\pi_1$} (-5,-5)
            -- node[left] {$c_n$} cycle;
            \draw (-5,5) -- node[above] {$c_3$} (-1,1) -- node[above] {$\pi_1$} (1,1) -- node[above] {$c_3$} (5,5);
            \draw (-5,-5) -- node[pos=0.85,inside] {$x_4$} node[inside,pos=0.95] {$x_3$} node[below, sloped] {$x_n\cdots x_6x_5$} (-1,1);
            \draw (5,-5) -- node[pos=0.85,inside] {$x_4$} node[inside,pos=0.95] {$x_3$} node[below, sloped] {$x_5x_6\cdots x_n$} (1,1);
            \draw (-5,5) -- node[right] {$c_{n-1}$} ($(-5,-5)!0.2!(-1,1)$);
            \draw (-5,5) -- node[inside] {$c_4$} ($(-5,-5)!0.9!(-1,1)$);
            \draw (-5,5) -- node[inside] {$c_5$} ($(-5,-5)!0.8!(-1,1)$);

            \draw (5,5) -- node[left] {$c_{n-1}$} ($(5,-5)!0.2!(1,1)$);
            \draw (5,5) -- node[inside] {$c_4$} ($(5,-5)!0.9!(1,1)$);
            \draw (5,5) -- node[inside] {$c_5$} ($(5,-5)!0.8!(1,1)$);

            \draw ($(-5,-5)!0.9!(-1,1)$) -- node[inside] {$\pi_1$} ($(5,-5)!0.9!(1,1)$);
            \draw ($(-5,-5)!0.8!(-1,1)$) -- node[inside] {$\pi_1$} ($(5,-5)!0.8!(1,1)$);

            \draw ($(-5,-5)!0.8!(-1,1)$) -- node[inside] {$x_3$} (-1,-1) -- node[inside] {$\pi_1$} (1,-1) -- node[inside] {$x_3$} ($(5,-5)!0.8!(1,1)$);
            \draw (-1,-1) -- node[sloped,below] {$(x_3^{-1}x_4)^{n-4}$}  (-1,-4);
            \draw (1,-1) -- node [sloped, above] {$(x_3^{-1}x_4)^{n-4}$} (1,-4);
            \foreach \y in {-1.5,-2,...,-3.5} {
                \draw (-1,\y) -- (1,\y);
            }
            \draw (-5,-5) -- node[above] {$x_3^{n-4}$} (-1,-4) -- node[inside] {$\pi_1$} (1,-4) -- node[above] {$x_3^{n-4}$} (5,-5);
            \foreach \p in {0.2,0.4,...,0.8} {
                \draw ($(-5,-5)!\p!(-1,-4)$) -- ($(5,-5)!\p!(1,-4)$);
            }
        \end{tikzpicture}
    \end{equation}
\end{proof}
We would like to tackle the relations $\pi_j^{c_i} = \pi_{j+1}\;(0\leq j<i-1)$ next. But the definition
of $\pi_1 = \pi_0^{c_2}$ makes this deceptively hard to manage in the case $j=0$. We devote two entire lemmas to
this first case. We start by considering the simplest diagrams.

\begin{lemma}\label{lem:low_cases_pi0_to_pi1}
    \begin{equation*}
        \|\pi_0^{c_3} = \pi_1\| = 5
    \end{equation*}
    \begin{equation*}
        \|\pi_0^{c_4} = \pi_1\| = 8
    \end{equation*}
    \begin{equation*}
        \|\pi_0^{c_5} = \pi_1\| \leq 44
    \end{equation*}
\end{lemma}
\begin{proof}
    We construct a van Kampen diagram with boundary $\pi_0^{c_5}=\pi_1$ that also
    contains subdiagrams proving the other claims.
    
    \begin{equation*}
        \begin{tikzpicture}
            \fill[pattern=dots, pattern color=gray!60] (-3,3) -- (-1,1) -- (1,1) -- (3,3) -- cycle;
            \fill[pattern=dots, pattern color=gray!60] (-3,-3) -- (-1,-1) -- (-1,0) -- (-3,3) -- cycle;
            \fill[pattern=dots, pattern color=gray!60] (3,-3) -- (1,-1) -- (1,0) -- (3,3) -- cycle;
            \fill[pattern=dots, pattern color=gray!60] (-3,-3) -- (-1,-1.5) -- (-1,-2.5) -- cycle;
            \fill[pattern=dots, pattern color=gray!60] (3,-3) -- (1,-1.5) -- (1,-2.5) -- cycle;
            \fill[blue!30] (-1,-1) rectangle (1,-2);
            \fill[blue!30] (-3,3) -- (-1,1) -- (1,1) -- (3,3) -- (1,0) -- (-1,0) -- cycle;
            \fill[blue!30] (-3,-3) -- (-1,-1) -- (-1,-1.5) -- cycle;
            \fill[blue!30] (3,-3) -- (1,-1) -- (1,-1.5) -- cycle;
            
            \draw (-3,3) -- node[above] {$\pi_1$} (3,3) -- node[right] {$c_5$} (3,-3) -- node[below] {$\pi_0$} (-3,-3) -- node[left] {$c_5$} cycle;
            \draw (-1,1) -- node[inside] {$\pi_1$} (1,1) -- node[inside] {$c_4$} (1,0) -- node[inside] {$x_0$} (1,-1) -- node[inside] {$\pi_1$} (-1,-1) -- node[inside] {$x_0$} (-1,0) -- node[inside] {$c_4$} cycle;
            \draw (-3,3) -- node[above right] {$x_4$} (-1,1);
            \draw (-3,3) -- node[inside] {$c_5$} (-1,0);
            \draw (3,3) -- node[above left] {$x_4$} (1,1);
            \draw (3,3) -- node[inside] {$c_5$} (1,0);
            \draw (-3,-3) -- node[inside] {$c_4$} (-1,-1);
            \draw (3,-3) -- node[inside] {$c_4$} (1,-1);
            \draw (-1,0) -- node[inside] {$\pi_2$} (1,0);
            \draw (-1,1) -- node[inside, scale=0.7] {$x_3$} (-0.5,0.5) -- node[inside, scale=0.7] {$c_3$} (-1,0);
            \draw (1,1) -- node[inside, scale=0.7] {$x_3$} (0.5,0.5) -- node[inside, scale=0.7] {$c_3$} (1,0);
            \draw (-0.5,0.5) -- node[inside] {$\pi_1$} (0.5,0.5);
            \draw (-1,-1) -- node[inside, scale=0.7] {$x_3$} (-1,-1.5) -- node[inside, scale=0.7] {$\pi_1$} (1,-1.5) -- node[inside, scale=0.7] {$x_3$} (1,-1);
            \draw (-3,-3) -- node[inside] {$c_3$} (-1,-1.5);
            \draw (3,-3) -- node[inside] {$c_3$} (1,-1.5);
            \draw (-1,-1.5) -- node[inside, scale=0.7] {$c_3$} (-1,-2) -- node[inside] {$\pi_2$} (1,-2) -- node[inside,scale=0.7] {$c_3$} (1,-1.5);
            \draw (-1,-2) -- node[inside, scale=0.7] {$x_0$} (-1,-2.5) -- node[inside] {$\pi_1$} (1,-2.5) -- node[inside, scale=0.7] {$x_0$} (1,-2);
            \draw (-3,-3) -- node[inside] {$c_2$} (-1,-2.5);
            \draw (3,-3) -- node[inside] {$c_2$} (1,-2.5);
        \end{tikzpicture}
    \end{equation*}
    The cells (or conjugations of cells) are highlighted. The regions containing multiple cells are dotted. The definitions are left undecorated.
    From this diagram, one can read that
    \begin{equation*}
        \|\pi_0^{c_3} = \pi_1\| = 1 + 2\|c_2x_0 = c_3^2\| = 1+4 = 5,
    \end{equation*}
    and thus
    \begin{equation*}
        \|\pi_0^{c_4} = \pi_1\| = 3 + \|\pi_0^{c_3} = \pi_1\| = 8,
    \end{equation*}
    and finally
    \begin{align*}
        \|\pi_0^{c_5} = \pi_1\| &\leq 6 + \|\pi_0^{c_4} = \pi_1\| + \|\pi_1x_4 = x_4\pi_1\| + 2\|c_4x_0 = c_5^2\|\\
            &\leq 6 + 8 + 6 + 2(11) = 44
    \end{align*}
    Here we used \zcref{lem:x4pi1_pi1x4} and the construction of \zcref{lem:cnx0_cn+1^2}.
\end{proof}

\begin{lemma}\label{lem:pi0^ci_pi1}
    For any integer $n\geq 2$,
    \begin{equation*}
        \|\pi_0^{c_n} = \pi_{1}\| = O(n^2)
    \end{equation*}
\end{lemma}
\begin{proof}
    \zcref{lem:low_cases_pi0_to_pi1} takes care of the cases $n=3,4,5$. So suppose that $n\geq 6$. This makes it possible to draw following diagram.
    \begin{equation}\label{eq:full_pi0^ci_pi1}
        \begin{tikzpicture}
            \fill[pattern=checkerboard, pattern color=gray!30] (-4,3) -- (4,3) -- (2,2) -- (-2,2) -- cycle;
            \fill[pattern=checkerboard, pattern color=gray!30] (-2,1) -- (2,1) -- (1.5,0) -- (-1.5,0) -- cycle;
            \fill[blue!20] (-4,3) -- (-2,2) -- (-2,1) -- (-1.5,0) -- (-1.5,-1) -- (-4,-3) -- cycle;
            \fill[blue!20] (4,3) -- (2,2) -- (2,1) -- (1.5,0) -- (1.5,-1) -- (4,-3) -- cycle;

            \draw (-4,3) -- node[above] {$\pi_1$} (4,3) -- node[right] {$c_n$} (4,-3) -- node[below] {$\pi_0$} (-4,-3) -- node[left] {$c_n$} cycle;
            \draw (-4,3) -- node[inside] {$c_n$} (-2,2) -- node[below] {$\pi_2$} (2,2) -- node[inside] {$c_n$} (4,3);
            \draw (-2,2) -- node[left] {$x_0$} (-2,1);
            \draw (2,2) -- node[right] {$x_0$} (2,1);
            \draw (-4,-3) -- node[inside] {$c_{n-1}$} (-2,1) -- node[above] {$\pi_1$} (2,1) -- node[inside] {$c_{n-1}$} (4,-3);

            \draw (-2,1) -- node[inside] {$c_{n-1}$} (-1.5,0) -- node[below] {$\pi_2$} (1.5,0) -- node[inside] {$c_{n-1}$} (2,1);
            \draw (-1.5,0) -- node[inside] {$x_0$} (-1.5,-1) -- node[above] {$\pi_1$} (1.5,-1) -- node[inside] {$x_0$} (1.5,0);
            \draw (-4,-3) -- node[inside] {$c_{n-2}$} (-1.5,-1);
            \draw (4,-3) -- node[inside] {$c_{n-2}$} (1.5,-1);

            \draw (-4,-3) -- node[inside] {$c_5$} (-0.5,-2.5) -- node[above] {$\pi_1$} (0.5,-2.5) -- node[inside] {$c_5$} (4,-3);
            \draw[thick, dotted] (-2,-1.5) -- (-1,-2.5);
            \draw[thick, dotted] (2,-1.5) -- (1,-2.5);
        \end{tikzpicture}
        \begin{tikzpicture}[scale=0.2]
            \draw (-4,3) -- (4,3) -- (4,-3) -- (-4,-3) -- cycle;
            \draw (-4,3) -- (-2,2) -- (2,2) -- (4,3);
            \draw (-2,2) -- (-2,1);
            \draw (2,2) -- (2,1);
            \draw (-4,-3) -- (-2,1) -- (2,1) -- (4,-3);

            \draw (-2,1) -- (-1.5,0) -- (1.5,0) -- (2,1);
            \draw (-1.5,0) -- (-1.5,-1) -- (1.5,-1) -- (1.5,0);
            \draw (-4,-3) -- (-1.5,-1);
            \draw (4,-3) -- (1.5,-1);

            \draw (-4,-3) -- (-0.5,-2.5) -- (0.5,-2.5) -- (4,-3);
            \draw[thick, dotted] (-2,-1.5) -- (-1,-2.5);
            \draw[thick, dotted] (2,-1.5) -- (1,-2.5);
        \end{tikzpicture}
    \end{equation}
    Note the shrunk version of the diagram in the lower right corner of the picture. This indicator will appear in some of the following diagrams to keep track of the part of diagram (\ref{eq:full_pi0^ci_pi1}) we are working in.
    
    To get a good bound on the area of this diagram, we need to understand the junctions between the
    highlighted subdiagrams and the checkered subdiagrams, and the junctions between two highlighted subdiagrams.
    \zcref{lem:cnpi2_pi1cn} gives us the recipe for filling the checkered regions. This involves
    two regions---the ones dotted in diagram (\ref{diag:pi1cn_pi2})---fully in $F$ with boundary equation of the form
    \begin{equation*}
        x_5x_6\cdots x_{k-4} = (x_3^{-1}x_4)^{k-4}x_3^{k-4},\qquad (k\in\{6,\dots,n\})\tag{from \zcref{eq:sides_of_pi1cn_pi2}}
    \end{equation*}
    If these do not cancel in some way, they will lead to a cubic area in our total diagram.
    Therefore, we need these $F$-subdiagrams to also appear in our highlighted subdiagrams.
    Unfortunately, our current way of realising the highlighted subdiagrams (see \zcref{lem:cnx0_cn+1^2})
    does not offer this. Instead, they contain subdiagrams with boundary equation
    \begin{equation*}
        x_1^{k-3} = (x_2^{-1}x_1)^{k-3}x_3x_4\cdots x_{k-1},\qquad (k\in\{6,\dots,n\})\tag{from \zcref{eq:inner_part_cnx0_cn+12}}
    \end{equation*}
    Very similar, but not the same. To solve this, we propose a method of filling $c_{k-1}x_0 = c_k^2$ analogous to \zcref{lem:cnx0_cn+1^2}, but different.
    The sketch of the diagram is as follows.
    \begin{equation*}
        \begin{tikzpicture}[rotate=90]
            \fill[pattern=dots] (-5,5) -- (-4,4) -- (-4,-4) -- (-3,-4) -- (5,-5) -- (-5,-5) -- cycle;
            \draw (-5,5) -- node[left] {$x_0$} (5,5) -- node[above] {$c_k$} (5,-5) -- node[right] {$c_k$} (-5,-5) -- node[below] {$c_{k-1}$} cycle;
            \draw (-5,5) -- node[inside, scale=0.7] {$x_3^{k-5}$} (-4,4) -- node[inside] {$c_4$} (-4,-4) -- node[inside, scale=0.7] {$x_2^{k-5}$} (-5,-5);
            \draw (-4,4) -- node[inside] {$x_0$} (-3,4) -- node[inside] {$c_5$} (-3,-4) -- node[inside] {$c_5$} (-4,-4);

            \filldraw[fill=blue!20] (5,-5) -- node[sloped, rotate=90, below] {$x_5\cdots x_{k-1}$} (3,3)
                -- node[rotate=90, below] {$x_3(x_4^{-1}x_3)^{k-7}$} (0,3)
                -- node[inside] {$x_4$}(-1,0) -- node[inside] {$x_3$} (-1.6,0)
                -- node[inside] {$x_4$} (-3,-4);
            \draw (0,3) -- node[inside] {$c_4$} ($(-3,4)!0.25!(5,5)$);
            \draw (3,3) -- node[inside] {$c_5$} (5,5);
            \foreach \p in {0.3,0.5,...,0.9} {
                \draw (5,5) -- ($(5,-5)!\p!(3,3)$);
            }
            \draw (-3,4) -- node[inside, pos=0.15] {$x_4$} node[inside] {$x_4^{k-6}$} (5,5);
            \draw (-3,-4) -- node[inside, pos=0.5] {$x_3^{k-5}$} (5,-5);

            \draw (-3,4) -- node[inside] {$c_4$} (-1.6,0);
            \draw (-1,0) -- node[inside] {$c_5$} ($(-3,4)!0.25!(5,5)$);
        \end{tikzpicture}
        \begin{tikzpicture}[scale=0.2]
            \fill[green!20] (4,3) -- (2,2) -- (2,1) -- (4,-3) -- cycle;
            \draw (-4,3) -- (4,3) -- (4,-3) -- (-4,-3) -- cycle;
            \draw (-4,3) -- (-2,2) -- (2,2) -- (4,3);
            \draw (-2,2) -- (-2,1);
            \draw (2,2) -- (2,1);
            \draw (-4,-3) -- (-2,1) -- (2,1) -- (4,-3);

            \draw (-2,1) -- (-1.5,0) -- (1.5,0) -- (2,1);
            \draw (-1.5,0) -- (-1.5,-1) -- (1.5,-1) -- (1.5,0);
            \draw (-4,-3) -- (-1.5,-1);
            \draw (4,-3) -- (1.5,-1);

            \draw (-4,-3) -- (-0.5,-2.5) -- (0.5,-2.5) -- (4,-3);
            \draw[thick, dotted] (-2,-1.5) -- (-1,-2.5);
            \draw[thick, dotted] (2,-1.5) -- (1,-2.5);
        \end{tikzpicture}
    \end{equation*}
    We remark that the biggest change (compared to \zcref{lem:cnx0_cn+1^2}) is in the
    dotted regions. Instead of reducing the subscript of $c$-letters via the definition of
    said $c$-letters, we opted for $x_2$- and $x_3$-letters.
    With these alterations, we gained something and we lost something. The highlighted region now has the
    exact boundary we found in the checkered subdiagrams of (\ref{eq:full_pi0^ci_pi1}). But this came at the price of the dotted regions now being
    ``bad''. Indeed, they are themselves non-trivial diagrams in $T$.
    Fortunately, we will see that they cancel when we concatenate a diagram with boundary $c_{k-1}x_0 = c_k^2$
    and $c_{k-2}x_0 = c_{k-1}^2$.

    But first we look at the junction between this new diagram and the checkered diagram.

    \begin{equation*}
        \begin{tikzpicture}[rotate=90, scale=0.9]
            \fill[pattern=checkerboard, pattern color=gray!30] (5,5) -- ($(8,3)!0.1!(5,-5)$) -- (5,-5) -- (2,2) -- cycle;
            \fill[pattern=dots] (-5,5) -- (-4,4) -- (-4,-4) -- (-3,-4) -- (5,-5) -- (-5,-5) -- cycle;
            \draw (-5,5) -- node[left] {$x_0$} (5,5) -- node[above] {$c_k$} (5,-5) -- node[right] {$c_k$} (-5,-5) -- node[below] {$c_{k-1}$} cycle;
            \draw (-5,5) -- node[inside, scale=0.7] {$x_3^{k-5}$} (-4,4) -- node[inside] {$c_4$} (-4,-4) -- node[inside, scale=0.7] {$x_2^{k-5}$} (-5,-5);
            \draw (-4,4) -- node[inside] {$x_0$} (-3,4) -- node[inside] {$c_5$} (-3,-4) -- node[inside] {$c_5$} (-4,-4);

            \filldraw[fill=blue!20] (5,-5) -- node[sloped, rotate=90, below] {$x_5\cdots x_{k-1}$} (3,3)
                -- node[rotate=90, below] {$x_3(x_4^{-1}x_3)^{k-7}$} (0,3)
                -- node[inside] {$x_4$}(-1,0) -- node[inside] {$x_3$} (-1.6,0)
                -- node[inside] {$x_4$} (-3,-4);
            \draw (0,3) -- node[inside] {$c_4$} ($(-3,4)!0.25!(5,5)$);
            \draw (3,3) -- node[inside] {$c_5$} (5,5);
            \foreach \p in {0.3,0.5,...,0.9} {
                \draw (5,5) -- ($(5,-5)!\p!(3,3)$);
            }
            \draw (-3,4) -- node[inside, pos=0.15] {$x_4$} node[inside] {$x_4^{k-6}$} (5,5);
            \draw (-3,-4) -- node[inside, pos=0.5] {$x_3^{k-5}$} (5,-5);

            \draw (-3,4) -- node[inside] {$c_4$} (-1.6,0);
            \draw (-1,0) -- node[inside] {$c_5$} ($(-3,4)!0.25!(5,5)$);

            \fill[blue!20] ($(8,3)!0.1!(5,-5)$) -- (5,-5) -- (9,-4) -- cycle;

            \draw (5,5) -- node[left] {$\pi_2$} (15,5) -- node[above] {$c_k$} (15,-5) -- node[right] {$\pi_1$} (5,-5);
            \draw (5,5) -- node[left] {$c_3$} (8,3) -- node[left] {$\pi_1$} (12,3) -- node[left] {$c_3$} (15,5);
            \draw (8,3) -- node[sloped, rotate=-90, above] {$x_5\cdots x_{k-1}$} (5,-5);
            \foreach \p in {0.05,0.1,0.2,0.3,0.5,0.7} {
                \draw (5,5) -- ($(8,3)!\p!(5,-5)$);
            }
            \draw (5,-5) -- node[left] {$x_3^{k-5}$} (9,-4);
            \draw ($(8,3)!0.1!(5,-5)$) -- node[below] {$(x_3^{-1}x_4)^{k-5}$} (9,-4);

            \draw (12,3) -- node[sloped, rotate=90, below] {$x_5\cdots x_{k-1}$} (15,-5);
            \foreach \p in {0.05,0.1,0.2,0.3,0.5,0.7} {
                \draw (15,5) -- ($(12,3)!\p!(15,-5)$);
            }
            \draw (15,-5) -- node[left] {$x_3^{k-5}$} (11,-4);
            \draw ($(12,3)!0.1!(15,-5)$) -- node[above] {$(x_3^{-1}x_4)^{k-5}$} (11,-4);

            \draw ($(8,3)!0.05!(5,-5)$) -- node[inside, scale=0.7] {$\pi_1$} ($(12,3)!0.05!(15,-5)$);
            \draw ($(8,3)!0.1!(5,-5)$) -- node[inside, scale=0.7] {$\pi_1$} ($(12,3)!0.1!(15,-5)$);
            \draw (9,-4) -- node[inside] {$\pi_1$} (11,-4);
        \end{tikzpicture}
        \begin{tikzpicture}[scale=0.2]
            \fill[green!20] (4,3) -- (-4,3) -- (-2,2) -- (2,2) -- (2,1) -- (4,-3) -- cycle;
            \draw (-4,3) -- (4,3) -- (4,-3) -- (-4,-3) -- cycle;
            \draw (-4,3) -- (-2,2) -- (2,2) -- (4,3);
            \draw (-2,2) -- (-2,1);
            \draw (2,2) -- (2,1);
            \draw (-4,-3) -- (-2,1) -- (2,1) -- (4,-3);

            \draw (-2,1) -- (-1.5,0) -- (1.5,0) -- (2,1);
            \draw (-1.5,0) -- (-1.5,-1) -- (1.5,-1) -- (1.5,0);
            \draw (-4,-3) -- (-1.5,-1);
            \draw (4,-3) -- (1.5,-1);

            \draw (-4,-3) -- (-0.5,-2.5) -- (0.5,-2.5) -- (4,-3);
            \draw[thick, dotted] (-2,-1.5) -- (-1,-2.5);
            \draw[thick, dotted] (2,-1.5) -- (1,-2.5);
        \end{tikzpicture}
    \end{equation*}
    After cancellation of the dipoles in the checkered region and the highlighted regions, we obtain a simpler diagram.
    
    \begin{equation*}
        \begin{tikzpicture}[rotate=90, scale=0.88]
            \fill[pattern=dots] (-5,5) -- (4,-3.5) -- (6,-4) -- (5,-5) -- (-5,-5) -- cycle;
            \draw (5,-5) -- node[right] {$c_k$} (-5,-5) -- node[below] {$c_{k-1}$} (-5,5) -- node[left] {$x_0$} (5,5);

            \draw (6,2.16) -- node[inside] {$c_5$} (5,5);

            \draw (5,5) -- node[above] {$c_3$} (6,3) -- node[above] {$(x_3^{-1}x_4)^{k-5}$} (6,-4) -- node[inside] {$x_3^{k-5}$} (5,-5);
            \draw (5,-3) -- node[inside, scale=0.7] {$c_4$} (6,-2.5);
            \draw (5,-1) -- node[inside, scale=0.7] {$c_5$} (6,-1.5);
            \draw (5,-1) -- node[inside, scale=0.7] {$c_4$} (6,-0.5);
            \node at (5.5,0) {$\cdots$};
            \draw (5,2) -- node[inside, scale=0.7] {$c_4$} (6,2);

            \draw (6,-4) -- node[inside] {$c_5$} (4,-3.5) -- node[inside] {$x_2^{k-5}$} (-5,-5);
            \draw (6,-4) -- node[inside] {$c_5$} (5,-3) -- node[left] {$x_0$} (3,-2.5) -- node[inside] {$c_4$} (4,-3.5);
            \draw (5,5) -- node[below] {$x_4^{k-5}$} (5,-3);

            \draw (5,5) -- node[left] {$\pi_2$} (15,5) -- node[above] {$c_k$} (15,-5) -- node[right] {$\pi_1$} (5,-5);
            \draw[ultra thick] (15,5) -- (15,-5);
            \draw (6,3) -- node[left] {$\pi_1$} (12,3) -- node[left] {$c_3$} (15,5);
            \draw (-5,5) -- node[left] {$x_3^{k-5}$} (3,-2.5);

            \draw (12,3) -- node[sloped, rotate=90, below] {$x_5\cdots x_{k-1}$} (15,-5);
            \foreach \p in {0.05,0.1,0.2,0.3,0.5,0.7} {
                \draw (15,5) -- ($(12,3)!\p!(15,-5)$);
            }
            \draw (15,-5) -- node[left] {$x_3^{k-5}$} (11,-4);
            \draw ($(12,3)!0.1!(15,-5)$) -- node[above] {$(x_3^{-1}x_4)^{k-5}$} (11,-4);

            \draw (6,3) |- node[inside, pos=0.8, scale=0.7] {$\pi_1$} ($(12,3)!0.05!(15,-5)$);
            \draw (6,3) |- node[inside, pos=0.8, scale=0.7] {$\pi_1$} ($(12,3)!0.1!(15,-5)$);
            \draw (6,-4) -- node[inside] {$\pi_1$} (11,-4);
        \end{tikzpicture}
        \begin{tikzpicture}[scale=0.2]
            \fill[green!20] (4,3) -- (-4,3) -- (-2,2) -- (2,2) -- (2,1) -- (4,-3) -- cycle;
            \draw (-4,3) -- (4,3) -- (4,-3) -- (-4,-3) -- cycle;
            \draw (-4,3) -- (-2,2) -- (2,2) -- (4,3);
            \draw[ultra thick] (-4,3) -- (-2,2);
            \draw (-2,2) -- (-2,1);
            \draw (2,2) -- (2,1);
            \draw (-4,-3) -- (-2,1) -- (2,1) -- (4,-3);

            \draw (-2,1) -- (-1.5,0) -- (1.5,0) -- (2,1);
            \draw (-1.5,0) -- (-1.5,-1) -- (1.5,-1) -- (1.5,0);
            \draw (-4,-3) -- (-1.5,-1);
            \draw (4,-3) -- (1.5,-1);

            \draw (-4,-3) -- (-0.5,-2.5) -- (0.5,-2.5) -- (4,-3);
            \draw[thick, dotted] (-2,-1.5) -- (-1,-2.5);
            \draw[thick, dotted] (2,-1.5) -- (1,-2.5);
        \end{tikzpicture}
    \end{equation*}
    This shows what happens in the top right corner of diagram (\ref{eq:full_pi0^ci_pi1}). Now we remark that the top left corner of diagram (\ref{eq:full_pi0^ci_pi1}) is a mirror image of the top right one. Attaching the second subdiagram with boundary
    equation $c_{n-1}x_0 = c_n^2$ to the thick edge,
    results in the exact same cancellation of cells on the other side.
    Rotating the result 90 degrees counter-clockwise, yields the following diagram (where some cells have not been drawn for legibility reasons).
    
    \begin{equation*}
        \begin{tikzpicture}
            \fill[blue!20] (-1,2.5) -- (1,2.5) -- (2,1.5) -- (1,-4) -- (-1,-4) -- (-2,1.5) -- (-1,2.5) -- cycle; 
            \fill[pattern=dots] (1,4) -- (5,4) -- (5,-4) -- (3,1.5) -- (2,2.5) -- (1,2.5) -- cycle;
            \fill[pattern=dots] (-1,4) -- (-5,4) -- (-5,-4) -- (-3,1.5) -- (-2,2.5) -- (-1,2.5) -- cycle;
            \draw (-1,-4) -- node[right, pos=0.1] {$c_3$} 
                            node[below, sloped, pos=0.55] {$(x_3^{-1}x_4)^{k-5}$}
                            node[inside, pos=0.9] {$x_3^{k-5}$} (-1,4)
                        -- node[above] {$\pi_1$} (1,4)
                        -- node[left, pos=0.9] {$c_3$}
                            node[above, sloped, rotate=180, pos=0.45] {$(x_3^{-1}x_4)^{k-5}$}
                            node[inside, pos=0.1] {$x_3^{k-5}$} (1,-4) -- node[below] {$\pi_2$} cycle;
            \draw (-1,-2.5) -- node[inside] {$\pi_1$} (1,-2.5);
            \draw (-1,-2) -- node[inside] {$\pi_1$} (1,-2);
            \draw (-1,2.5) -- node[inside] {$\pi_1$} (1,2.5);

            \draw (1,2.5) -- node[inside] {$c_5$} (2,1.5) -- node[inside] {$x_0$} (3,1.5);
            \draw (1,2.5) -- node[inside] {$c_5$} (2,2.5) -- node[inside] {$c_4$} (3,1.5);

            \draw (2,1.5) -- node[right] {$x_4^{k-5}$} (1,-4);
            \draw (1,-4) -- node[below] {$x_0$} (5,-4) -- node[right] {$c_{k-1}$} (5,4) -- node[above] {$c_k$} (1,4);
            \draw (5,-4) -- node[inside] {$x_3^{k-5}$} (3,1.5);
            \draw (5,4) -- node[inside] {$x_2^{k-5}$} (2,2.5);

            \draw (-1,2.5) -- node[inside] {$c_5$} (-2,1.5) -- node[inside] {$x_0$} (-3,1.5);
            \draw (-1,2.5) -- node[inside] {$c_5$} (-2,2.5) -- node[inside] {$c_4$} (-3,1.5);

            \draw (-2,1.5) -- node[left] {$x_4^{k-5}$} (-1,-4);
            \draw (-1,-4) -- node[below] {$x_0$} (-5,-4) -- node[left] {$c_{k-1}$} (-5,4) -- node[above] {$c_k$} (-1,4);
            \draw (-5,-4) -- node[inside] {$x_3^{k-5}$} (-3,1.5);
            \draw (-5,4) -- node[inside] {$x_2^{k-5}$} (-2,2.5);
        \end{tikzpicture}
        \begin{tikzpicture}[scale=0.2]
            \fill[green!20] (4,3) -- (-4,3) -- (-4,-3) -- (-2,1) -- (-2,2) -- (2,2) -- (2,1) -- (4,-3) -- cycle;
            \draw (-4,3) -- (4,3) -- (4,-3) -- (-4,-3) -- cycle;
            \draw (-4,3) -- (-2,2) -- (2,2) -- (4,3);
            \draw (-2,2) -- (-2,1);
            \draw (2,2) -- (2,1);
            \draw (-4,-3) -- (-2,1) -- (2,1) -- (4,-3);

            \draw (-2,1) -- (-1.5,0) -- (1.5,0) -- (2,1);
            \draw (-1.5,0) -- (-1.5,-1) -- (1.5,-1) -- (1.5,0);
            \draw (-4,-3) -- (-1.5,-1);
            \draw (4,-3) -- (1.5,-1);

            \draw (-4,-3) -- (-0.5,-2.5) -- (0.5,-2.5) -- (4,-3);
            \draw[thick, dotted] (-2,-1.5) -- (-1,-2.5);
            \draw[thick, dotted] (2,-1.5) -- (1,-2.5);
        \end{tikzpicture}
    \end{equation*}
    We can do one more simplification. We remove the highlighted subdiagram and replace it with another bearing the same boundary.
    \begin{equation*}
        \begin{tikzpicture}
            \fill[pattern=dots] (1,4) -- (5,4) -- (5,-4) -- (2,1.5) -- (2,2.5) -- (1,2.5) -- cycle;
            \fill[pattern=dots] (-1,4) -- (-5,4) -- (-5,-4) -- (-2,1.5) -- (-2,2.5) -- (-1,2.5) -- cycle;
            \draw (-1,-4) -- node[left, pos=0.4] {$x_4^{k-5}$}
                            node[inside, pos=0.9] {$x_3^{k-5}$} (-1,4)
                            node[right, pos=0.75] {$c_5$}
                        -- node[above] {$\pi_1$} (1,4)
                        -- node[right, pos=0.6] {$x_4^{k-5}$}
                            node[inside, pos=0.1] {$x_3^{k-5}$}
                            node[left, pos=0.25] {$c_5$} (1,-4)
                        -- node[below] {$\pi_2$} cycle;

            \draw (1,2.5) -- node[inside] {$c_5$} (2,2.5) -- node[inside] {$c_4$} (2,1.5) -- node[inside] {$x_0$} (1,1.5);

            \draw (1,-4) -- node[below] {$x_0$} (5,-4) -- node[right] {$c_{k-1}$} (5,4) -- node[above] {$c_k$} (1,4);
            \draw (5,-4) -- node[inside] {$x_3^{k-5}$} (2,1.5);
            \draw (5,4) -- node[inside] {$x_2^{k-5}$} (2,2.5);

            \draw (-1,2.5) -- node[inside] {$c_5$} (-2,2.5) -- node[inside] {$c_4$} (-2,1.5) -- node[inside] {$x_0$} (-1,1.5);

            \draw (-1,-4) -- node[below] {$x_0$} (-5,-4) -- node[left] {$c_{k-1}$} (-5,4) -- node[above] {$c_k$} (-1,4);
            \draw (-5,-4) -- node[inside] {$x_3^{k-5}$} (-2,1.5);
            \draw (-5,4) -- node[inside] {$x_2^{k-5}$} (-2,2.5);

            \draw (-1,2.5) -- node[inside] {$\pi_1$} (1,2.5);
            \draw (-1,1.5) -- node[inside] {$\pi_2$} (1,1.5);
            \foreach \y in {1,0.5,0,-2,-2.5,-3,-3.5} {
                \draw (-1,\y) -- node[inside] {$\pi_2$} (1,\y);
            }
            \node at (0,-1) {$\vdots$};
        \end{tikzpicture}
        \begin{tikzpicture}[scale=0.2]
            \fill[green!20] (4,3) -- (-4,3) -- (-4,-3) -- (-2,1) -- (-2,2) -- (2,2) -- (2,1) -- (4,-3) -- cycle;
            \draw (-4,3) -- (4,3) -- (4,-3) -- (-4,-3) -- cycle;
            \draw (-4,3) -- (-2,2) -- (2,2) -- (4,3);
            \draw (-2,2) -- (-2,1);
            \draw (2,2) -- (2,1);
            \draw (-4,-3) -- (-2,1) -- (2,1) -- (4,-3);

            \draw (-2,1) -- (-1.5,0) -- (1.5,0) -- (2,1);
            \draw (-1.5,0) -- (-1.5,-1) -- (1.5,-1) -- (1.5,0);
            \draw (-4,-3) -- (-1.5,-1);
            \draw (4,-3) -- (1.5,-1);

            \draw (-4,-3) -- (-0.5,-2.5) -- (0.5,-2.5) -- (4,-3);
            \draw[thick, dotted] (-2,-1.5) -- (-1,-2.5);
            \draw[thick, dotted] (2,-1.5) -- (1,-2.5);
        \end{tikzpicture}
    \end{equation*}
    Finally, we have obtained one ``layer'' of diagram (\ref{eq:full_pi0^ci_pi1}). It remains to determine what happens when such layers are stacked
    upon one another (with the checkered definition in between). To draw this, some stretching is necessary
    .
    \begin{equation*}
        \begin{tikzpicture}
            \fill[pattern=checkerboard, pattern color=gray!30] (-1,1) rectangle (1,0);
            \fill[pattern=dots] (1,0) -- ($(1,0)!0.7!(5,3)$) -- (5,0) -- (2,-1) -- (1,-1) -- cycle;
            \fill[pattern=dots, pattern color=gray!30] ($(1,0)!0.7!(5,3)$) --++ (0,1) -- (5,4) -- (5,0) -- cycle;
            \fill[pattern=dots, pattern color=gray!30] (2,-1) -- (5,0) -- (5,-4) -- (2,-2) -- cycle;
            \fill[pattern=dots] (-1,0) -- ($(-1,0)!0.7!(-5,3)$) -- (-5,0) -- (-2,-1) -- (-1,-1) -- cycle;
            \fill[pattern=dots, pattern color=gray!30] ($(-1,0)!0.7!(-5,3)$) --++ (0,1) -- (-5,4) -- (-5,0) -- cycle;
            \fill[pattern=dots, pattern color=gray!30] (-2,-1) -- (-5,0) -- (-5,-4) -- (-2,-2) -- cycle;
            \draw (-5,-4) -- node[left] {$c_{k-2}$} (-5,0) -- node[left] {$c_k$} (-5,4) -- node[above] {$\pi_1$} (5,4) -- node[right] {$c_k$} (5,0) -- node[right] {$c_{k-2}$} (5,-4);
            \draw (-5,0) -- node[inside] {$c_{k-1}$} (-1,0) -- node[inside] {$\pi_1$} (1,0) -- node[inside] {$c_{k-1}$} (5,0);
            \draw (-1,0) -- node[inside] {$x_0$} (-1,1) -- node[inside] {$\pi_2$} (1,1) -- node[inside] {$x_0$} (1,0);

            \draw (1,1) -- node[inside, pos=0.2] {$x_4^{k-5}$} node[inside, pos=0.6] {$c_5$} node[inside, pos=0.85] {$x_3^{k-5}$} (5,4);
            \draw ($(1,1)!0.5!(5,4)$) -- node[inside] {$x_0$} ($(1,0)!0.5!(5,3)$);
            \draw ($(1,1)!0.7!(5,4)$) -- node[inside] {$c_5$} ($(1,0)!0.7!(5,3)$);
            \draw ($(1,0)!0.5!(5,3)$) -- node[inside] {$c_4$} ($(1,0)!0.7!(5,3)$);
            \draw (1,0) -- node[inside] {$x_3^{k-5}$} ($(1,0)!0.5!(5,3)$);
            \draw (5,0) -- node[inside] {$x_2^{k-5}$} ($(1,0)!0.7!(5,3)$);

            \draw (-1,1) -- node[inside, pos=0.2] {$x_4^{k-5}$} node[inside, pos=0.6] {$c_5$} node[inside, pos=0.85] {$x_3^{k-5}$} (-5,4);
            \draw ($(-1,1)!0.5!(-5,4)$) -- node[inside] {$x_0$} ($(-1,0)!0.5!(-5,3)$);
            \draw ($(-1,1)!0.7!(-5,4)$) -- node[inside] {$c_5$} ($(-1,0)!0.7!(-5,3)$);
            \draw ($(-1,0)!0.5!(-5,3)$) -- node[inside] {$c_4$} ($(-1,0)!0.7!(-5,3)$);
            \draw (-1,0) -- node[inside] {$x_3^{k-5}$} ($(-1,0)!0.5!(-5,3)$);
            \draw (-5,0) -- node[inside] {$x_2^{k-5}$} ($(-1,0)!0.7!(-5,3)$);

            \draw ($(-1,1)!0.5!(-5,4)$) -- node[inside] {$\pi_2$} ($(1,1)!0.5!(5,4)$);
            \draw ($(-1,1)!0.7!(-5,4)$) -- node[inside] {$\pi_1$} ($(1,1)!0.7!(5,4)$);

            \draw (1,0) -- node[inside, scale=0.7, pos=0.125] {$x_3^{k-6}$} node[inside, pos=0.35] {$c_5$} node[inside, pos=0.75] {$x_4^{k-6}$} (1,-4);
            \draw (1,-1) -- node[inside] {$c_5$} (2,-1) -- node[inside] {$c_4$} (2,-2) -- node[inside] {$x_0$} (1,-2);
            \draw (2,-1) -- node[below] {$x_2^{k-6}$} (5,0);
            \draw (2,-2) -- node[inside] {$x_3^{k-6}$} (5,-4);

            \draw (-1,0) -- node[inside, scale=0.7, pos=0.125] {$x_3^{k-6}$} node[inside, pos=0.35] {$c_5$} node[inside, pos=0.75] {$x_4^{k-6}$} (-1,-4);
            \draw (-1,-1) -- node[inside] {$c_5$} (-2,-1) -- node[inside] {$c_4$} (-2,-2) -- node[inside] {$x_0$} (-1,-2);
            \draw (-2,-1) -- node[below] {$x_2^{k-6}$} (-5,0);
            \draw (-2,-2) -- node[inside] {$x_3^{k-6}$} (-5,-4);

            \draw (-1,-1) -- node[inside] {$\pi_1$} (1,-1);
            \draw (-1,-2) -- node[inside] {$\pi_2$} (1,-2);
            \draw (-5,-4) -- node[below] {$x_0$} (-1,-4) -- node[below] {$\pi_2$} (1,-4) -- node[below] {$x_0$} (5,-4);
        \end{tikzpicture}
        \begin{tikzpicture}[scale=0.2]
            \fill[green!20] (4,3) -- (-4,3) -- (-4,-3) -- (-1.5,-1) -- (-1.5,0) -- (1.5,0) -- (1.5,-1) -- (4,-3) -- cycle;
            \fill[pattern=checkerboard, pattern color=green] (-2,2) rectangle (2,1);
            \draw (-4,3) -- (4,3) -- (4,-3) -- (-4,-3) -- cycle;
            \draw (-4,3) -- (-2,2) -- (2,2) -- (4,3);
            \draw (-2,2) -- (-2,1);
            \draw (2,2) -- (2,1);
            \draw (-4,-3) -- (-2,1) -- (2,1) -- (4,-3);

            \draw (-2,1) -- (-1.5,0) -- (1.5,0) -- (2,1);
            \draw (-1.5,0) -- (-1.5,-1) -- (1.5,-1) -- (1.5,0);
            \draw (-4,-3) -- (-1.5,-1);
            \draw (4,-3) -- (1.5,-1);

            \draw (-4,-3) -- (-0.5,-2.5) -- (0.5,-2.5) -- (4,-3);
            \draw[thick, dotted] (-2,-1.5) -- (-1,-2.5);
            \draw[thick, dotted] (2,-1.5) -- (1,-2.5);
        \end{tikzpicture}
    \end{equation*}
    The four dotted subdiagrams with high opacity form two sets of almost mirror images. Both pairs can thus be
    collapsed into a subdiagram with boundary equation $c_4x_3 = x_2c_5$. Repeating this for each ``layer'' in diagram (\ref{eq:full_pi0^ci_pi1})
    ensures that the end result only contains two dotted subdiagrams. Indeed, we obtain the following diagram.
    \begin{equation*}
        \begin{tikzpicture}[scale=0.7, every node/.style={scale=0.7}]
            \fill[pattern=dots, pattern color=gray!60] (8,3) -- ($(2,-4)!0.8!(8,3)$) -- ($(2,-5)!0.8!(8,2)$) to[out=-45, in=0] ($(2,-5)!0.6!(8,2)$) to[out=-45, in=0] ($(2,-5)!0.4!(8,2)$) -- ($(2,-5)!0.2!(8,2)$) -- (8,-8) -- cycle;
            \fill[pattern=dots, pattern color=gray!60] (-8,3) -- ($(-2,-4)!0.8!(-8,3)$) -- ($(-2,-5)!0.8!(-8,2)$) to[out=-135, in=180] ($(-2,-5)!0.6!(-8,2)$) to[out=-135, in=180] ($(-2,-5)!0.4!(-8,2)$) -- ($(-2,-5)!0.2!(-8,2)$) -- (-8,-8) -- cycle;

            \draw (-8,3) -- node[above] {$\pi_1$} (8,3) -- node[right] {$c_n$} (8,-8) -- node[below] {$\pi_0$} (-8,-8) -- node[left] {$c_n$} cycle;
            \draw (-8,-8) -- node[inside] {$c_5$} (-2,-5) -- node[inside] {$\pi_1$} (2,-5) -- node[inside] {$c_5$} (8,-8);
            \draw (-2,-5) -- node[left] {$c_5$} (-2,-6) -- node[inside] {$\pi_2$} (2,-6) -- node[right] {$c_5$} (2,-5);
            \draw (-2,-5) -- node[inside, scale=0.7, pos=0.25] {$x_4$} node[inside, scale=0.7, pos=0.75] {$x_3$} (-1,-5.75) -- node[inside, scale=0.5] {$c_3$} (-2,-6);
            \draw (-2,-6) -- node[inside, scale=0.5] {$c_4$} ($(-2,-5)!0.5!(-1,-5.75)$);
            \draw (2,-5) -- node[inside, scale=0.7, pos=0.25] {$x_4$} node[inside, scale=0.7, pos=0.75] {$x_3$} (1,-5.75) -- node[inside, scale=0.5] {$c_3$} (2,-6);
            \draw (2,-6) -- node[inside, scale=0.5] {$c_4$} ($(2,-5)!0.5!(1,-5.75)$);
            \draw ($(-2,-5)!0.5!(-1,-5.75)$) -- node[inside, scale=0.7] {$\pi_1$} ($(2,-5)!0.5!(1,-5.75)$);
            \draw (-1,-5.75) -- node[inside, scale=0.7] {$\pi_1$} (1,-5.75);
            \draw (-2,-6) -- node[left, pos=0.125] {$x_0$} node[right, pos=0.375] {$x_3$} node[right, pos=0.625] {$c_3$} node[right, pos=0.875] {$x_0$} (-2,-7.5);
            \draw (2,-6) -- node[right, pos=0.125] {$x_0$} node[left, pos=0.375] {$x_3$} node[left, pos=0.625] {$c_3$} node[left, pos=0.875] {$x_0$} (2,-7.5);
            \draw (-8,-8) -- node[inside] {$c_4$} ($(-2,-6)!0.25!(-2,-7.5)$) -- node[inside, scale=0.7] {$\pi_1$} ($(2,-6)!0.25!(2,-7.5)$) -- node[inside] {$c_4$} (8,-8);
            \draw (-8,-8) -- node[inside, pos=0.6] {$c_3$} ($(-2,-6)!0.5!(-2,-7.5)$) -- node[inside, scale=0.7] {$\pi_1$} ($(2,-6)!0.5!(2,-7.5)$) -- node[inside, pos=0.4] {$c_3$} (8,-8);
            \draw ($(-2,-6)!0.75!(-2,-7.5)$) -- node[inside, scale=0.7] {$\pi_2$} ($(2,-6)!0.75!(2,-7.5)$);
            \draw (-8,-8) -- node[inside] {$c_2$} (-2,-7.5) -- node[inside, scale=0.7] {$\pi_1$} (2,-7.5) -- node[inside] {$c_2$} (8,-8);

            \node[rotate=77] at (3.75,-2.25) {$\ddots$};
            \node[rotate=-15] at (-3.75,-2.25) {$\ddots$};
            \node at (0,-1.8) {$\vdots$};

            \draw (8,3) -- node[inside] {$x_3^{n-5}$} ($(2,-4)!0.8!(8,3)$);
            \draw ($(2,-4)!0.8!(8,3)$) -- node[inside] {$c_5$} ($(2,-5)!0.8!(8,2)$) -- node[inside] {$c_4$} ($(2,-5)!0.7!(8,2)$) -- node[inside] {$x_0$} ($(2,-4)!0.7!(8,3)$) -- node[inside] {$c_5$} cycle;
            \draw ($(2,-5)!0.7!(8,2)$) -- node[inside] {$x_3$} ($(2,-4)!0.6!(8,3)$);
            \draw ($(2,-5)!0.8!(8,2)$) to[out=-45, in=0] node[right] {$x_2$} ($(2,-5)!0.6!(8,2)$);
            \draw ($(2,-4)!0.6!(8,3)$) -- node[inside] {$c_5$} ($(2,-5)!0.6!(8,2)$) -- node[inside] {$c_4$} ($(2,-5)!0.5!(8,2)$) -- node[inside] {$x_0$} ($(2,-4)!0.5!(8,3)$) -- node[inside] {$c_5$} cycle;
            \draw ($(2,-5)!0.5!(8,2)$) -- node[inside] {$x_3$} ($(2,-4)!0.4!(8,3)$);
            \draw ($(2,-5)!0.6!(8,2)$) to[out=-45, in=0] node[right] {$x_2$} ($(2,-5)!0.4!(8,2)$);
            \draw ($(2,-5)!0.4!(8,2)$) -- node[inside] {$c_5$} ($(2,-4)!0.4!(8,3)$);

            \draw ($(2,-4)!0.2!(8,3)$) -- node[inside] {$c_5$} ($(2,-5)!0.2!(8,2)$) -- node[inside] {$c_4$} ($(2,-5)!0.1!(8,2)$) -- node[inside] {$x_0$} ($(2,-4)!0.1!(8,3)$) -- node[inside] {$c_5$} cycle;
            \draw ($(2,-5)!0.1!(8,2)$) -- node[inside] {$x_3$} (2,-5);
            \draw ($(2,-5)!0.2!(8,2)$) -- node[above right] {$x_2$} (8,-8);

            \draw (-8,3) -- node[inside] {$x_3^{n-5}$} ($(-2,-4)!0.8!(-8,3)$);
            \draw ($(-2,-4)!0.8!(-8,3)$) -- node[inside] {$c_5$} ($(-2,-5)!0.8!(-8,2)$) -- node[inside] {$c_4$} ($(-2,-5)!0.7!(-8,2)$) -- node[inside] {$x_0$} ($(-2,-4)!0.7!(-8,3)$) -- node[inside] {$c_5$} cycle;
            \draw ($(-2,-5)!0.7!(-8,2)$) -- node[inside] {$x_3$} ($(-2,-4)!0.6!(-8,3)$);
            \draw ($(-2,-5)!0.8!(-8,2)$) to[out=-135, in=180] node[left] {$x_2$} ($(-2,-5)!0.6!(-8,2)$);
            \draw ($(-2,-4)!0.6!(-8,3)$) -- node[inside] {$c_5$} ($(-2,-5)!0.6!(-8,2)$) -- node[inside] {$c_4$} ($(-2,-5)!0.5!(-8,2)$) -- node[inside] {$x_0$} ($(-2,-4)!0.5!(-8,3)$) -- node[inside] {$c_5$} cycle;
            \draw ($(-2,-5)!0.5!(-8,2)$) -- node[inside] {$x_3$} ($(-2,-4)!0.4!(-8,3)$);
            \draw ($(-2,-5)!0.6!(-8,2)$) to[out=-135, in=180] node[left] {$x_2$} ($(-2,-5)!0.4!(-8,2)$);
            \draw ($(-2,-5)!0.4!(-8,2)$) -- node[inside] {$c_5$} ($(-2,-4)!0.4!(-8,3)$);

            \draw ($(-2,-4)!0.2!(-8,3)$) -- node[inside] {$c_5$} ($(-2,-5)!0.2!(-8,2)$) -- node[inside] {$c_4$} ($(-2,-5)!0.1!(-8,2)$) -- node[inside] {$x_0$} ($(-2,-4)!0.1!(-8,3)$) -- node[inside] {$c_5$} cycle;
            \draw ($(-2,-5)!0.1!(-8,2)$) -- node[inside] {$x_3$} (-2,-5);
            \draw ($(-2,-5)!0.2!(-8,2)$) -- node[above left] {$x_2$} (-8,-8);

            \draw ($(-2,-4)!0.8!(-8,3)$) -- node[inside] {$\pi_1$} ($(2,-4)!0.8!(8,3)$);
            \draw ($(-2,-4)!0.7!(-8,3)$) -- node[inside] {$\pi_2$} ($(2,-4)!0.7!(8,3)$);
            \draw ($(-2,-4)!0.6!(-8,3)$) -- node[inside] {$\pi_1$} ($(2,-4)!0.6!(8,3)$);
            \draw ($(-2,-4)!0.5!(-8,3)$) -- node[inside] {$\pi_2$} ($(2,-4)!0.5!(8,3)$);
            \draw ($(-2,-4)!0.4!(-8,3)$) -- node[inside] {$\pi_1$} ($(2,-4)!0.4!(8,3)$);
            \draw ($(-2,-4)!0.2!(-8,3)$) -- node[inside] {$\pi_1$} ($(2,-4)!0.2!(8,3)$);
            \draw ($(-2,-4)!0.1!(-8,3)$) -- node[inside] {$\pi_2$} ($(2,-4)!0.1!(8,3)$);

            \draw ($(-2,-5)!0.1!(-8,2)$) -- node[inside] {$\pi_1$} ($(2,-5)!0.1!(8,2)$);
        \end{tikzpicture}
    \end{equation*}
    We can finish the proof by counting out the number of cells in this diagram. We get
    \begin{align*}
        \|\pi_0^{c_n} = \pi_1\| &\leq 2\|x_2^{-(n-5)}c_5x_3^{n-5} = c_n\| + (n-5)\|\pi_1x_3=x_3\pi_1\|\\
                                &\qquad + 2(n-5)\Big(\|c_4x_0 = c_5^2\| + \|c_4x_3=x_2c_5\|\Big)\\
                                &\qquad + (n-5)\Big(\|\pi_1^{c_5} = \pi_2\| + \|\pi_1x_3=x_3\pi_1\|\Big)\\
                                &\qquad + \|\pi_0^{c_5} = \pi_1\|\\
                                &\leq 2\|x_2^{-(n-5)}c_5x_3^{n-5} = c_n\| + (n-5)\\
                                &\qquad + 2(n-5)\Big(11 + 5\Big)\tag{\zcref{lem:cnx0_cn+1^2,lem:cnxk_xk-1cn+1}}\\
                                &\qquad + (n-5)\Big(12 + 1\Big)\tag{\zcref{lem:pi0^ci_pi1}}\\
                                &\qquad + 44\tag{\zcref{lem:low_cases_pi0_to_pi1}}\\
                                &\leq 2\|x_2^{-(n-5)}c_5x_3^{n-5} = c_n\| + 46(n-5) + 44
    \end{align*}
    The last hurdle is the dotted subdiagram with boundary equation $x_2^{-(n-5)}c_5x_3^{n-5} = c_n$.
    But this lies completely in $T$, which has quadratic Dehn function \cite[Theorem A]{Migliorini_2025}.
    We determine the length of this boundary to get a bound
    on the area. Using the definitions in \zcref{def:definitions_of_infinite_letters},
    \begin{align*}
        x_2^{-(n-5)}c_5x_3^{n-5} &= c_n\\
        \iff x_0^{-1}x_1^{-(n-5)}x_0x_0^{-4}c_1x_1^4x_0^{-2}x_1^{n-5}x_0^2 &= x_0^{-(n-1)}c_1x_1^{n-1}\\
        \iff x_1^{-(n-5)}x_0^{-3}c_1x_1^4x_0^{-2}x_1^{n-5}x_0^2 &= x_0^{-(n-2)}c_1x_1^{n-1}
    \end{align*}
    Thus the length of this boundary is $4n$, and
    \begin{equation*}
        \|x_2^{-(n-5)}c_5x_3^{n-5} = c_n\| = O((4n)^2) = O(n^2)
    \end{equation*}
    In conclusion,
    \begin{equation*}
        \|\pi_0^{c_n} = \pi_1\| \leq 2O(n^2) + 46(n-5) + 44 = O(n^2)
    \end{equation*}
\end{proof}

\begin{lemma}\label{lem:cipij_pij-1ci}
    For $0\leq j<i-1$,
    \begin{equation*}
        \|\pi_{j}^{c_i} = \pi_{j+1}\| = O(n^2)
    \end{equation*}
    where $n = \max\{i-j+1,j-1\}$.
\end{lemma}
\begin{proof}
    The previous lemma handles the case $j=0$.
    The case $j=1$ is handled by \zcref{lem:cnpi2_pi1cn} with $n=i$.
    For the rest of the proof we suppose that $j\geq 2$.
    Consider the following outline of a van Kampen diagram.
    \begin{equation*}
        \begin{tikzpicture}
            \draw (-2,2) -- node[above] {$\pi_{j+1}$} (2,2)
            -- node[right] {$c_i$} (2,-2)
            -- node[below] {$\pi_j$} (-2,-2)
            -- node[left] {$c_i$} cycle;
            \draw (-2,2) -- node[inside, scale=0.8] {$x_1^{j-1}$} (-1,1) -- node[above] {$\pi_2$} (1,1) -- node[inside, scale=0.8] {$x_1^{j-1}$} (2,2);
            \draw (-2,-2) -- node[inside, scale=0.8] {$x_0^{j-1}$} (-1,-1) -- node[below] {$\pi_1$} (1,-1) -- node[inside, scale=0.8] {$x_0^{j-1}$} (2,-2);
            \draw (-1,-1) -- node[inside, rectangle, inner sep=1pt] {$c_{i-j+1}$} (-1,1);
            \draw (1,-1) -- node[inside, rectangle, inner sep=1pt] {$c_{i-j+1}$} (1,1);
        \end{tikzpicture}
    \end{equation*}
    The left, bottom, and right regions are simply definitions and require no cells.
    The central region can be filled with $O((i-j+1)^2)$ cells according to Lemma \ref{lem:cnpi2_pi1cn}.
    It remains to fill the top region. We defer this to \zcref{lem:pi2tox1n-1_pin+1}.
\end{proof}

The proof of \zcref{lem:pi0^ci_pi1} opens up a proof for a relation we avoided
up until now.

\begin{lemma}\label{lem:pi0xn_xnpi0}
    For $n\geq 2$
    \begin{equation*}
        \|\pi_0x_n = x_n\pi_0\| = O(n^2)
    \end{equation*}
\end{lemma}
\begin{proof}
    As we have done before, when dealing with a $\pi_0$, we transform it into a $\pi_1$.
    One way of doing this, is with a $c$-letter.
    For $n=2$, we can find
    \begin{equation*}
        \begin{tikzpicture}
            \fill[blue!30] (-2,-2) -- (-0.5,-0.5) -- (-0.5,0.5) -- (-2,2) -- (2,2) -- (0.5,0.5) -- (0.5,-0.5) -- (2,-2) -- cycle;
            \fill[white] (-0.5,0.5) rectangle +(1,1);
            \fill[white] (-0.5,-0.5) rectangle +(1,-1);
            \draw (-2,2) -- node[above] {$x_2$} (2,2) -- node[right] {$\pi_0$} (2,-2) -- node[below] {$x_2$} (-2,-2) -- node[left] {$\pi_0$} cycle;
            \draw (-0.5,0.5) -- node[inside] {$x_3$} (0.5,0.5) -- node[inside] {$\pi_1$} (0.5,-0.5) -- node[inside] {$x_3$} (-0.5,-0.5) -- node[inside] {$\pi_1$} cycle;
            \draw (-2,2) -- node[inside] {$c_4$} (-0.5,0.5);
            \draw (2,2) -- node[inside] {$c_5$} (0.5,0.5);
            \draw (2,-2) -- node[inside] {$c_5$} (0.5,-0.5);
            \draw (-2,-2) -- node[inside] {$c_4$} (-0.5,-0.5);

            \draw (-2,2) -- node[inside] {$x_3$} (-0.5,1.5) -- node[inside] {$x_2$} (0.5,1.5) -- node[inside] {$x_4$} (2,2);
            \draw (-0.5,0.5) -- node[inside] {$c_3$} (-0.5,1.5);
            \draw (0.5,0.5) -- node[inside] {$c_4$} (0.5,1.5);

            \draw (-2,-2) -- node[inside] {$x_3$} (-0.5,-1.5) -- node[inside] {$x_2$} (0.5,-1.5) -- node[inside] {$x_4$} (2,-2);
            \draw (-0.5,-0.5) -- node[inside] {$c_3$} (-0.5,-1.5);
            \draw (0.5,-0.5) -- node[inside] {$c_4$} (0.5,-1.5);
        \end{tikzpicture}
    \end{equation*}
    where each cell (or conjugation of a cell) is highlighted. The other regions are constructed with multiple cells. The totals areas is
    \begin{align*}
        \|\pi_0x_2 = x_2\pi_0\| &\leq 7 + \|\pi_0^{c_4} = \pi_1\| + \|\pi_0^{c_5}=\pi_1\| +2\|x_2c_4=c_3x_3\|\\
            &\leq 7 + 8 + 44 + 2(2)= 63 \tag{\zcref{lem:low_cases_pi0_to_pi1,lem:cnxk_xk-1cn+1}}
    \end{align*}
    For $n\geq 3$, we contstruct the following.
    \begin{equation*}
        \begin{tikzpicture}
            \draw (-2,2) -- node[above] {$x_n$} (2,2) -- node[right] {$\pi_0$} (2,-2) -- node[below] {$x_n$} (-2,-2) -- node[left] {$\pi_0$} cycle;
            \draw (-0.5,0.5) -- node[inside] {$x_{n+1}$} (0.5,0.5) -- node[inside] {$\pi_1$} (0.5,-0.5) -- node[inside] {$x_{n+1}$} (-0.5,-0.5) -- node[inside] {$\pi_1$} cycle;
            \draw[->] (-2,2) -- node[inside] {$c_{n+2}$} (-0.5,0.5);
            \draw[->] (2,2) -- node[inside] {$c_{n+3}$} (0.5,0.5);
            \draw[->] (2,-2) -- node[inside] {$c_{n+3}$} (0.5,-0.5);
            \draw[->] (-2,-2) -- node[inside] {$c_{n+2}$} (-0.5,-0.5);
        \end{tikzpicture}
    \end{equation*}
    From \zcref{lem:pi0^ci_pi1} we know that the left and right subdiagram have themselves
    subdiagrams on their flanks featuring $x_2$- and $x_3$-letters.
    In order to cancel these, we tweak the diagrams from \zcref{lem:cnxk_xk-1cn+1}
    to also include $x_2$'s and $x_3$'s on their corresponding flanks.
    
    We propose the
    following diagram.
    \begin{equation*}
        \begin{tikzpicture}
            \fill[blue!20] (-2,-2) -- (-0.5,-0.5) -- (-0.5,0.5) -- (-2,2) -- cycle;
            \fill[blue!20] (2,-2) -- (0.5,-0.5) -- (0.5,0.5) -- (2,2) -- cycle;
            \draw (-2,2) -- node[above] {$x_n$} (2,2) -- node[right] {$c_{n+3}$} (2,-2) -- node[below] {$x_{n+1}$} (-2,-2) -- node[left] {$c_{n+2}$} cycle;
            \draw (-0.5,0.5) -- node[inside] {$x_3$} (0.5,0.5) -- node[inside] {$c_6$} (0.5,-0.5) -- node[inside] {$x_4$} (-0.5,-0.5) -- node[inside] {$c_5$} cycle;
            \draw (-2,2) -- node[inside] {$x_2^{n-3}$} (-0.5,0.5);
            \draw (2,2) -- node[inside] {$x_2^{n-3}$} (0.5,0.5);
            \draw (2,-2) -- node[inside] {$x_3^{n-3}$} (0.5,-0.5);
            \draw (-2,-2) -- node[inside] {$x_3^{n-3}$} (-0.5,-0.5);
        \end{tikzpicture}
    \end{equation*}
    This new diagram is worse in area than our original way of completing this diagram, but better
    suited for this proof.
    Plugging two copies of this diagram and two copies of $(\pi_0^{c}=\pi_1)$-diagrams
    in our blueprint, and we obtain.
    \begin{equation*}
        \begin{tikzpicture}
            \fill[blue!20] (-4,4) -- (-1,3) -- (-1,2) -- (-0.5,0.5) -- (-2,1) -- (-3,1) -- cycle;
            \fill[blue!20] (-4,-4) -- (-1,-3) -- (-1,-2) -- (-0.5,-0.5) -- (-2,-1) -- (-3,-1) -- cycle;
            \fill[blue!20] (4,4) -- (1,3) -- (1,2) -- (0.5,0.5) -- ($(0.5,0.5)!0.8!(2,1)$) to[out=90, in=180] ($(3,1)!0.1!(4,4)$) -- cycle;
            \fill[blue!20] (4,-4) -- (1,-3) -- (1,-2) -- (0.5,-0.5) -- ($(0.5,-0.5)!0.8!(2,-1)$) to[out=-90, in=180] ($(3,-1)!0.1!(4,-4)$) -- cycle;
            \draw (-4,4) -- node[above] {$x_n$} (4,4) -- node[right] {$\pi_0$} (4,-4) -- node[below] {$x_n$} (-4,-4) -- node[left] {$\pi_0$} cycle;
            \draw (-0.5,0.5) -- node[inside] {$x_{n+1}$} (0.5,0.5) -- node[inside] {$\pi_1$} (0.5,-0.5) -- node[inside] {$x_{n+1}$} (-0.5,-0.5) -- node[inside] {$\pi_1$} cycle;
            \draw (-4,4) -- node[inside] {$c_{n+2}$} (-0.5,0.5);
            \draw (4,4) -- node[inside] {$c_{n+3}$} (0.5,0.5);
            \draw (4,-4) -- node[inside] {$c_{n+3}$} (0.5,-0.5);
            \draw (-4,-4) -- node[inside] {$c_{n+2}$} (-0.5,-0.5);

            \draw (-1,3) -- node[inside] {$x_3$} (1,3) -- node[inside] {$c_6$} (1,2) -- node[inside] {$x_4$} (-1,2) -- node[inside] {$c_5$} cycle;
            \draw (-4,4) -- node[inside] {$x_2^{n-3}$} (-1,3);
            \draw (4,4) -- node[inside] {$x_2^{n-3}$} (1,3);
            \draw (-0.5,0.5) -- node[inside, scale=0.7] {$x_3^{n-3}$} (-1,2);
            \draw (0.5,0.5) -- node[inside, scale=0.7] {$x_3^{n-3}$} (1,2);

            \draw (-1,-3) -- node[inside] {$x_3$} (1,-3) -- node[inside] {$c_6$} (1,-2) -- node[inside] {$x_4$} (-1,-2) -- node[inside] {$c_5$} cycle;
            \draw (-4,-4) -- node[inside] {$x_2^{n-3}$} (-1,-3);
            \draw (4,-4) -- node[inside] {$x_2^{n-3}$} (1,-3);
            \draw (-0.5,-0.5) -- node[inside, scale=0.7] {$x_3^{n-3}$} (-1,-2);
            \draw (0.5,-0.5) -- node[inside, scale=0.7] {$x_3^{n-3}$} (1,-2);

            \draw (-3,1) -- node[inside] {$c_5$} (-2,1) -- node[inside] {$\pi_1$} (-2,-1) -- node[inside] {$c_5$} (-3,-1) -- node[sloped, above] {$c_4x_3\pi_1x_3^{-1}c_4$} (-3,1);
            \draw (-4,4) -- node[inside] {$x_2^{n-3}$} (-3,1);
            \draw (-4,-4) -- node[inside] {$x_2^{n-3}$} (-3,-1);
            \draw (-0.5,0.5) -- node[inside, scale=0.7] {$x_3^{n-3}$} (-2,1);
            \draw (-0.5,-0.5) -- node[inside, scale=0.7] {$x_3^{n-3}$} (-2,-1);

            \draw (3,1) -- node[inside] {$c_5$} (2,1) -- node[inside] {$\pi_1$} (2,-1) -- node[inside] {$c_5$} (3,-1) -- node[sloped, below] {$c_4x_3\pi_1x_3^{-1}c_4$} (3,1);
            \draw (4,4) -- node[inside] {$x_2^{n-2}$} (3,1);
            \draw (4,-4) -- node[inside] {$x_2^{n-2}$} (3,-1);
            \draw (0.5,0.5) -- node[inside, scale=0.7] {$x_3^{n-2}$} (2,1);
            \draw (0.5,-0.5) -- node[inside, scale=0.7] {$x_3^{n-2}$} (2,-1);
        \end{tikzpicture}
    \end{equation*}
    The highlighted regions are exactly the subdiagrams we wanted to cancel as they are all formed by sets of dipoles, The two right pairs are not complete sets of dipoles though (shown by the non-highlighted parts in the highlighted region).
    
    Collapsing the four pairs yields
    \begin{equation*}
        \begin{tikzpicture}
            \fill[pattern=checkerboard, pattern color=gray!20] (-4,4) -- (-1.5,1.5) -- (-1.5,-1.5) -- (-4,-4) -- cycle;
            \fill[pattern=checkerboard, pattern color=gray!20] (4,4) -- (2,2) -- (2,0.5) -- (1.5,0.5) -- (1.5,-0.5) -- (2,-0.5) -- (2,-2) -- (4,-4) -- cycle;
            \fill[blue!20] (-1.5,1.5) -- (1.5,1.5) to[out=-110, in=110] (1.5,-1.5) -- (-1.5,-1.5) -- cycle;
            \draw (-4,4) -- node[above] {$x_n$} (4,4) -- node[right] {$\pi_0$} (4,-4) -- node[below] {$x_n$} (-4,-4) -- node[left] {$\pi_0$} cycle;
            \draw (-0.5,0.5) -- node[inside] {$x_{n+1}$} (0.5,0.5) -- node[inside] {$\pi_1$} (0.5,-0.5) -- node[inside] {$x_{n+1}$} (-0.5,-0.5) -- node[inside] {$\pi_1$} cycle;
            \draw (-4,4) -- node[inside] {$x_2^{n-3}$} (-2,2) -- node[inside] {$c_5$} (-1.5,1.5) -- node[inside, scale=0.7] {$x_3^{n-3}$} (-0.5,0.5);
            \draw (-4,-4) -- node[inside] {$x_2^{n-3}$} (-2,-2) -- node[inside] {$c_5$} (-1.5,-1.5) -- node[inside, scale=0.7] {$x_3^{n-3}$} (-0.5,-0.5);
            \draw (4,-4) -- node[inside] {$x_2^{n-3}$} (2,-2) -- node[inside] {$c_6$} (1.5,-1.5) -- node[inside, scale=0.7] {$x_3^{n-3}$} (0.5,-0.5);
            \draw (4,4) -- node[inside] {$x_2^{n-3}$} (2,2) -- node[inside] {$c_6$} (1.5,1.5) -- node[inside, scale=0.7] {$x_3^{n-3}$} (0.5,0.5);

            \draw (-2,2) -- node[inside] {$x_3$} (2,2);
            \draw (-1.5,1.5) -- node[inside] {$x_4$} (1.5,1.5);
            \draw (-2,-2) -- node[inside] {$x_3$} (2,-2);
            \draw (-1.5,-1.5) -- node[inside] {$x_4$} (1.5,-1.5);
            
            \draw (2,2) -- node[inside, scale=0.7] {$x_2$} (2,0.5) -- node[inside,scale=0.7] {$c_5$} (1.5,0.5) -- node[inside, scale=0.7] {$x_3$} (1.5,1.5);
            \draw (2,-2) -- node[inside, scale=0.7] {$x_2$} (2,-0.5) -- node[inside,scale=0.7] {$c_5$} (1.5,-0.5) -- node[inside, scale=0.7] {$x_3$} (1.5,-1.5);
            \draw (1.5,0.5) -- node[inside] {$\pi_1$} (1.5,-0.5);
            \draw (2,-0.5) -- node[scale=0.5, sloped, below] {$c_4x_3\pi_1x_3^{-1}c_4^{-1}$} (2,0.5);

            \draw (-1.5,1.5) -- node[inside] {$\pi_1$} (-1.5,-1.5);
            \draw (-2,-2) -- node[scale=0.5, sloped, above] {$c_4x_3\pi_1x_3^{-1}c_4^{-1}$} (-2,2);

            \draw (1.5,1.5) to[out=-110, in=110] node[inside, scale=0.7] {$\pi_1$} (1.5,-1.5);
        \end{tikzpicture}
    \end{equation*}
    We remark that the highlighted region has boundary equation $\pi_1x_4=x_4\pi_1$, which we know
    by \zcref{lem:x4pi1_pi1x4}, can be realised with only $6$ cells. Therefore, we
    can avoid the current highlighted subdiagram. According to \zcref{lem:pi0^ci_pi1},
    the left checkered diagram can be filled with $45n-91$ cells and the right
    one with $45n-62$ cells. The top and bottom regions are both fully in $F$. Their
    boundary equation is
    \begin{equation*}
        x_2^{-(n-3)}x_3x_2^{n-3} = x_n
    \end{equation*}
    which can be written as
    \begin{equation*}
        x_0^{-1}x_1^{-(n-3)}x_0x_0^{-2}x_1x_0^2x_0^{-1}x_1^{n-3}x_0 = x_0^{-(n-1)}x_1x_0^{n-1}
    \end{equation*}
    which in turn simplifies to
    \begin{equation*}
        x_1^{-(n-3)}x_0^{-1}x_1x_0x_1^{n-3} = x_0^{-(n-2)}x_1x_0^{n-2}
    \end{equation*}
    Thus, the length of this boundary is
    \begin{equation*}
        2(n-3) + 3 + 2(n-2) + 1 = 4n - 6
    \end{equation*}
    By Guba's result on the Dehn function of $F$, we have that each of the two regions takes $O((4n-6)^2)$ cells.
    We collect everything to bound the area of $\pi_0x_n=x_n\pi_0$.
    \begin{align*}
        \|\pi_0x_n=x_n\pi_0\| &\leq 2O((4n-6)^2) + (45n-91) + (45n-62)\\
            &\qquad + 2\|x_3c_6=c_5x_4\| + 2\|x_2c_6=c_5x_3\|\\
            &\qquad + 1 + 6\\
            &\leq 2O((4n-6)^2) + 90n-153 + 2(6) + 2(12) + 7\\
            &\leq 2O((4n-6)^2) + 90n -110 = O(n^2)
    \end{align*}
\end{proof}

\subsection{\texorpdfstring{Relations with (mostly) $\pi$-letters}{Relations with (mostly) pi-letters}}

Now we turn to relations containing only $\pi$-letters. We start with the most
simple ones.

\begin{lemma}
    For any integer $n\geq 0$,
    \begin{equation*}
        \|\pi_n^2 = 1\| = 1
    \end{equation*}
\end{lemma}
\begin{proof}
    This holds by conjugation of the relation $\pi_1^{2} = 1$.
\end{proof}

\begin{lemma}
    For any integer $n\geq 0$,
    \begin{equation*}
        \|(\pi_{n+1}\pi_n)^3 = 1\| = O(1)
    \end{equation*}
\end{lemma}
\begin{proof}
    For $n\geq 1$, this holds by conjugation of the relation $(\pi_2\pi_1)^3 = 1$. For $n=0$, we get the relation
    \begin{equation*}
        (\pi_1\pi_0)^3 = c_3(\pi_1^{c_3}\pi_0^{c_3})^3c_3^{-1} =  c_3(\pi_2\pi_1)^3c_3^{-1} = c_3c_3^{-1} =  1
    \end{equation*}
    which requires 5 defining relations according to \zcref{lem:low_cases_pi0_to_pi1}.
\end{proof}

Next, we would like to exchange to $\pi$-letters. To do this economically, we need an
intermediate result. This will also conlclude the proof of \zcref{lem:cipij_pij-1ci}.

\begin{lemma}\label{lem:pi2tox1n-1_pin+1}
    For any integer $n\geq 2$,
    \begin{equation*}
        \|\pi_2^{x_1^{n-1}} = \pi_{n+1}\| = O(n^2)
    \end{equation*}
\end{lemma}
\begin{proof}
    The case $n=2$ is the defining relation $\pi_2^{x_1} = \pi_3$. Let $n\geq 3$.
    This means, we have at least two instances of $x_1$ in the conjugation-exponent.
    To realise the appropriate van Kampen diagram, one can stack smaller van Kampen diagrams with
    boundary equations
    \begin{equation*}
        \pi_j^{x_1} = \pi_{j+1},\qquad j\in\{2,\dots,n\}
    \end{equation*}
    Following \zcref{lem:pijxi_pij+1}, this would result in $\sum_{j=2}^{n}O((j-1)^2) = O(n^3)$ cells.
    Again, we will show that within the stack, sufficiently many cells can be collapsed. To understand
    this, consider two such diagrams stacked upon one another and remark
    that the highlighted regions are mirror images up to one cell.
    This means we can collapse a good portion of the diagram as seen on the right.

    \begin{equation*}
        \begin{tikzpicture}
            \fill[blue!20] (-2,-2) -- (-0.5,-0.5) -- (0.5, -0.5) -- (2,-2) -- (0.5,-3.5) -- (-0.5,-3.5) -- cycle;

            \draw (-2,2) -- node[above] {$\pi_{i+1}$} (2,2) -- node[right] {$x_1$} (2,-2) -- node[below] {$\pi_i$} (-2,-2) -- node[left] {$x_1$} cycle;
            \draw (-2,2) -- node[above = -0.1, sloped] {$x_{i-1}\cdots x_3$} (-0.5,0.5) -- node[above] {$\pi_4$} (0.5,0.5) -- node[above = -0.1, sloped] {$x_3\cdots x_{i-1}$} (2,2);
            \draw (-2,-2) -- node[below = -0.1, sloped] {$x_{i-2}\cdots x_2$} (-0.5,-0.5) -- node[below] {$\pi_3$} (0.5,-0.5) -- node[below = -0.1, sloped] {$x_2\cdots x_{i-2}$} (2,-2);
            \draw (-0.5,0.5) -- node[left] {$x_1$} (-0.5,-0.5);
            \draw (0.5,0.5) -- node[right] {$x_1$} (0.5,-0.5);

            \draw (-2,-2) -- node[above = -0.1, sloped] {$x_{i-2}\cdots x_3$} (-0.5,-3.5) -- node[above] {$\pi_4$} (0.5,-3.5) -- node[above = -0.1, sloped] {$x_3\cdots x_{i-2}$} (2,-2);
            \draw (-2,-6) -- node[below = -0.1, sloped] {$x_{i-3}\cdots x_2$} (-0.5,-4.5) -- node[below] {$\pi_3$} (0.5,-4.5) -- node[below = -0.1, sloped] {$x_2\cdots x_{i-3}$} (2,-6);
            \draw (-0.5,-3.5) -- node[left] {$x_1$} (-0.5,-4.5);
            \draw (0.5,-3.5) -- node[right] {$x_1$} (0.5,-4.5);
            \draw (-2,-2) -- node[left] {$x_1$} (-2,-6) -- node[below] {$\pi_{i-1}$} (2,-6) -- node[right] {$x_1$} (2,-2);
            \node at (3,-2) {$\mbox{\Huge $\leadsto$}$};
        \end{tikzpicture}
        \begin{tikzpicture}
            \fill[blue!20] (-0.5,-0.5) -- (0.5,-0.5) -- (0.5,-2) -- (-0.5,-2) -- cycle;

            \draw (-2,-2) -- node[left] {$x_1$} (-2,2) -- node[above] {$\pi_{i+1}$} (2,2) -- node[right] {$x_1$} (2,-2);
            \draw (-2,2) -- node[above = -0.1, sloped] {$x_{i-1}\cdots x_3$} (-0.5,0.5) -- node[above] {$\pi_4$} (0.5,0.5) -- node[above = -0.1, sloped] {$x_3\cdots x_{i-1}$} (2,2);
            \draw (-0.5,0.5) -- node[left] {$x_1$} (-0.5,-0.5);
            \draw (0.5,0.5) -- node[right] {$x_1$} (0.5,-0.5);

            \draw (-0.5,-0.5) -- node[below] {$\pi_3$} (0.5,-0.5) -- node[right] {$x_2^{-1}$} (0.5,-2) -- node[below] {$\pi_4$} (-0.5,-2) -- node[left] {$x_2^{-1}$} cycle;

            \draw (-2,-6) -- node[below = -0.1, sloped] {$x_{i-3}\cdots x_2$} (-0.5,-4.5) -- node[below] {$\pi_3$} (0.5,-4.5) -- node[below = -0.1, sloped] {$x_2\cdots x_{i-3}$} (2,-6);
            \draw (-0.5,-2) -- node[left] {$x_1$} (-0.5,-4.5);
            \draw (0.5,-2) -- node[right] {$x_1$} (0.5,-4.5);
            \draw (-2,-2) -- node[left] {$x_1$} (-2,-6) -- node[below] {$\pi_{i-1}$} (2,-6) -- node[right] {$x_1$} (2,-2);

            \draw (-2,-2) -- node[below] {$\scriptstyle x_{i-2}\cdots x_3$} (-0.5,-2);
            \draw (0.5,-2) -- node[below] {$\scriptstyle x_3\cdots x_{i-2}$} (2,-2);
        \end{tikzpicture}
    \end{equation*}
    Where every vertical edge is oriented upwards.
    Repeating this pattern, one can get the desired van Kampen diagram.
    \begin{equation*}
        \begin{tikzpicture}
            \draw (-2,5) -- node[above] {$\pi_{n+1}$} (2,5);
            \draw (-2,5) -- node[above=-0.1,sloped] {$x_{n-1}\cdots x_3$} (-0.5,3) -- node[above] {$\pi_4$} (0.5,3) -- node[above=-0.1,sloped] {$x_3\cdots x_{n-1}$} (2,5);

            \draw (-0.5,3) -- node[left, pos=0.2] {$x_1$} node[left, pos=0.5] {$x_2^{-1}$} node[left, pos=0.8] {$x_1$} (-0.5,1);
            \draw (0.5,3) -- node[right,pos=0.2] {$x_1$} node[right,pos=0.5] {$x_2^{-1}$} node[right,pos=0.8] {$x_1$} (0.5,1);
            \draw ($(-0.5,3)!0.33!(-0.5,1)$) -- node[inside, scale=0.8] {$\pi_3$} ($(0.5,3)!0.33!(0.5,1)$);
            \draw ($(-0.5,3)!0.66!(-0.5,1)$) -- node[inside, scale=0.8] {$\pi_4$} ($(0.5,3)!0.66!(0.5,1)$);

            \draw[dotted] (-0.5,1) -- (-0.5,-1);
            \draw[dotted] (0.5,1) -- (0.5,-1);

            \draw (-0.5,-1) -- node[left, pos=0.2] {$x_1$} node[left, pos=0.5] {$x_2^{-1}$} node[left, pos=0.8] {$x_1$} (-0.5,-3);
            \draw (0.5,-1) -- node[right,pos=0.2] {$x_1$} node[right,pos=0.5] {$x_2^{-1}$} node[right,pos=0.8] {$x_1$} (0.5,-3);
            \draw (-0.5,-3) -- node[below] {$\pi_3$} (0.5,-3);
            \draw ($(-0.5,-3)!0.33!(-0.5,-1)$) -- node[inside, scale=0.8] {$\pi_4$} ($(0.5,-3)!0.33!(0.5,-1)$);
            \draw ($(-0.5,-3)!0.66!(-0.5,-1)$) -- node[inside, scale=0.8] {$\pi_3$} ($(0.5,-3)!0.66!(0.5,-1)$);
            \draw (-0.5,-3) -- node[left] {$x_1$} (-0.5,-4) -- node[below] {$\pi_2$} (0.5,-4) -- node[right] {$x_1$} (0.5,-3);
            \draw ($(-0.5,-3)!0.33!(-0.5,-1)$) -- ++(-1.5,0) -- node[left] {$x_1^{n-3}$} (-2,5);
            \draw ($(0.5,-3)!0.33!(0.5,-1)$) -- ++(1.5,0) -- node[right] {$x_1^{n-3}$} (2,5);
        \end{tikzpicture}
    \end{equation*}
    where every vertical edge is oriented upwards.
    The left and right regions of this diagram have boundary equation
    \begin{equation*}
        x_1^{n-3} = (x_2^{-1}x_1)^{n-3}x_3\cdots x_{n-1}
    \end{equation*}
    which can be expanded to
    \begin{equation*}
        x_1^{n-3} = (x_0^{-1}x_1^{-1}x_0x_1)^{n-3}x_0^{-1}(x_0^{-1}x_1)^{n-3}x_0^{n-2}
    \end{equation*}
    which has length $n-3 + 4(n-3) + 1 + 2(n-3) + n-2 = 8n - 22$.
    As the Dehn function on $F$ is quadratic, we have that these two regions can be filled
    with $O((8n-22)^2) = O(n^2)$ cells each. Furthermore, each region with boundary $\pi_{3}^{x_1} = \pi_4$
    can be filled with $6$ cells as seen in \zcref{lem:x1pi3x1_pi4} and every region with
    boundary $\pi_3^{x_2} = \pi_4$ follows directly from the relation $\pi_2^{x_1} = \pi_3$ via conjugation with $x_0$.
    Thus, the middle strip takes $(1+6)(n-3+1) = 7n-14$ cells.
    Lastly, we consider the top region of the diagram. This can be filled by stacking
    cells with boundary $\pi_{i+1}^{x_i} = \pi_{i+2}$ on top of eachother. Each cell is a direct
    consequence of $\pi_2^{x_1} = \pi_3$, thus it takes $n-3$ cells. Combining all of this, we conclude that
    it takes
    \begin{equation*}
        2O(n^2) + 7n-14 + n-3 = O(n^2)
    \end{equation*}
    cells to construct the wanted van Kampen diagram.
\end{proof}

\begin{lemma}\label{lem:pi4pi1_pi1pi4}
    \begin{equation*}
        \|\pi_4\pi_1 = \pi_1\pi_4\| = 11
    \end{equation*}
\end{lemma}
\begin{proof}
    Consider the van Kampen diagram
    \begin{equation*}
        \begin{tikzpicture}
            \fill[blue!20] (-4,4) -- (4,4) -- (2,2) -- (-2,2) -- cycle;
            \fill[blue!20] (-4,-4) -- (4,-4) -- (2,-2) -- (-2,-2) -- cycle;
            \fill[blue!20] (-3.5,2) -- (-2.5,2) -- (-2.5,-2) -- (-3.5,-2) -- cycle;
            \fill[blue!20] (0,2) -- (2,2) -- (2,-2) -- (0,-2) -- cycle;
            \fill[blue!20] (2.5,2) -- (3.5,2) -- (3.5,-2) -- (2.5,-2) -- cycle;
            \fill[pattern=checkerboard,pattern color=gray!30] (-2,2) -- (0,2) -- (0,-2) -- (-2,-2) -- cycle;
            
            \draw (-4,4) -- node[above] {$\pi_1$} (4,4)
            -- node[right] {$\pi_4$} (4,-4)
            -- node[below] {$\pi_1$} (-4,-4)
            -- node[left] {$\pi_4$} cycle;
            \draw (-4,4) -- node[above] {$\pi_2$} (-2,2) -- node[above, pos=0.25] {$x_1$} node[above, pos=0.75] {$\pi_1$} (2,2) -- node[above] {$x_2$} (4,4);
            \draw (-4,-4) -- node[below] {$\pi_2$} (-2,-2) -- node[below, pos=0.25] {$x_1$} node[below, pos=0.75] {$\pi_1$} (2,-2) -- node[below] {$x_2$} (4,-4);
            \draw (-2,2) -- node[right] {$\pi_4$} (-2,-2);
            \draw (0,2) -- node[inside] {$\pi_3$} (0,-2);
            \draw (2,2) -- node[left] {$\pi_3$} (2,-2);
            
            \draw (-3.5,2) -- node[above] {$\pi_1$} (-2.5,2) -- node[left] {$\pi_3$} (-2.5,-2) -- node[below] {$\pi_1$} (-3.5,-2) -- node[inside] {$\pi_3$} cycle;
            \begin{scope}[every node/.style={scale=0.7}]
                \draw (-4,4) -- node[inside] {$x_0$} (-3.5,2);
                \draw (-2,2) -- node[below] {$x_0$} (-2.5,2);
                \draw (-2,-2) -- node[above] {$x_0$} (-2.5,-2);
                \draw (-4,-4) -- node[inside] {$x_0$} (-3.5,-2);
            \end{scope}
            \draw (3.5,2) -- node[above] {$x_1$} (2.5,2) -- node[right] {$\pi_3$} (2.5,-2) -- node[below] {$x_1$} (3.5,-2) -- node[inside] {$\pi_3$} cycle;
            \begin{scope}[every node/.style={scale=0.7}]
                \draw (4,4) -- node[inside] {$x_0$} (3.5,2);
                \draw (2,2) -- node[below] {$x_0$} (2.5,2);
                \draw (2,-2) -- node[above] {$x_0$} (2.5,-2);
                \draw (4,-4) -- node[inside] {$x_0$} (3.5,-2);
            \end{scope}
        \end{tikzpicture}
    \end{equation*}
    where the highlighted regions are cells, the empty ones are definitions, and the checkered region
    can be filled with 6 cells according to \zcref{lem:x1pi3x1_pi4}.
    In total, we count $11$ cells.
\end{proof}

\begin{lemma}\label{lem:pinpi1_pi1pin}
    For any integer $n\geq 3$,
    \begin{equation*}
        \|\pi_n\pi_1 = \pi_1\pi_n\| = O(n^2)
    \end{equation*}
\end{lemma}
\begin{proof}
    The case $n=3$ is a relation of $V$, the case $n=4$ is \zcref{lem:pi4pi1_pi1pi4}.
    Suppose $n\geq 5$ and consider the van Kampen diagram
    \begin{equation*}
        \begin{tikzpicture}
            \draw (-2,2) -- node[above] {$\pi_n$} (2,2)
            -- node[right] {$\pi_1$} (2,-2)
            -- node[below] {$\pi_n$} (-2,-2)
            -- node[left] {$\pi_1$} cycle;
            \draw (-2,2) -- node[inside, scale=0.7] {$x_3^{n-4}$} (-1,1) -- node[above] {$\pi_4$} (1,1) -- node[inside, scale=0.7] {$x_3^{n-4}$} (2,2);
            \draw (-2,-2) -- node[inside, scale=0.7] {$x_3^{n-4}$} (-1,-1) -- node[below] {$\pi_4$} (1,-1) -- node[inside, scale=0.7] {$x_3^{n-4}$} (2,-2);
            \draw (-1,-1) -- node[left] {$\pi_1$} (-1,1);
            \draw (1,-1) -- node[right] {$\pi_1$} (1,1);
        \end{tikzpicture}
    \end{equation*}
    The left and right regions can be filled with $n-4$ cells each through the relation $x_3\pi_1=\pi_1 x_3$. The central region takes 11 cells
    as seen in \zcref{lem:pi4pi1_pi1pi4}. It remains to fill the top and
    bottom regions. But these can be filled with a conjugation of \zcref{lem:pi2tox1n-1_pin+1}.
    Specifically, \zcref{lem:pi2tox1n-1_pin+1} gives a van Kampen diagram with quadratic area
    and boundary equation
    \begin{equation*}
        \pi_2^{x_1^{m-1}} = \pi_{m+1}
    \end{equation*}
    If we conjugate with $x_0^{2}$, we get no new cells, but a boundary equation
    \begin{equation*}
        \pi_4^{x_3^{m-1}} = \pi_{m+3}
    \end{equation*}
    Setting $n=m+3$, we get
    \begin{equation*}
        \pi_4^{x_3^{n-4}} = \pi_{n}
    \end{equation*}
    with quadratic area. In total, we can construct the needed van Kampen diagram with area
    \begin{equation*}
        2(n-4) + 11 + 2O(n^2) = O(n^2)
    \end{equation*}
\end{proof}
Just like before, if we try to tackle the same relation but with a $\pi_0$-letter,
things get more complicated.

\begin{lemma}\label{pi0pin}
    For any integer $n\geq 2$,
    \begin{equation*}
            \|\pi_0\pi_n = \pi_n\pi_0\| = O(n^2)
    \end{equation*}
\end{lemma}
\begin{proof}
    We construct the following van Kampen diagram.
    \begin{equation}\label{eq:skeleton_pi0pin_pinpi0}
        \begin{tikzpicture}
            \fill[blue!20] (-3,3) -- (-1,1) -- (-1,-1) -- (-3,-3) -- cycle;
            \fill[blue!20] (3,3) -- (1,1) -- (1,-1) -- (3,-3) -- cycle;
            \draw (-3,3) -- node[above] {$\pi_n$} (3,3) -- node[right] {$\pi_0$} (3,-3) -- node[below] {$\pi_n$} (-3,-3) -- node[left] {$\pi_0$} cycle;
            \draw (-3,3) -- node[inside] {$c_{n+2}$} (-1,1);
            \draw (3,3) -- node[inside] {$c_{n+2}$} (1,1);
            \draw (3,-3) -- node[inside] {$c_{n+2}$} (1,-1);
            \draw (-3,-3) -- node[inside] {$c_{n+2}$} (-1,-1);
            \draw (-1,1) -- node[above] {$\pi_{n+1}$} (1,1) -- node[right] {$\pi_1$} (1,-1) -- node[below] {$\pi_{n+1}$} (-1,-1) -- node[left] {$\pi_1$} cycle;
        \end{tikzpicture}
    \end{equation}
    If $n=2,3$, we have the following diagrams.
    \begin{equation*}
        \begin{tikzpicture}
            \draw (-5,1.5) -- node[above] {$\pi_2$} (-2,1.5) -- node[right] {$\pi_0$} (-2,-1.5) -- node[below] {$\pi_2$} (-5,-1.5) -- node[left] {$\pi_0$} cycle;
            \draw (-5,1.5) -- node[inside] {$c_4$} (-4,0.5);
            \draw (-2,-1.5) -- node[inside] {$c_4$} (-3,-0.5);
            \draw (-5,-1.5) -- node[inside] {$c_4$} (-4,-0.5);
            \draw (-2,1.5) -- node[inside] {$c_4$} (-3,0.5);
            \draw (-4,0.5) -- node[inside] {$\pi_3$} (-3,0.5) -- node[inside] {$\pi_1$} (-3,-0.5) -- node[inside] {$\pi_3$} (-4,-0.5) -- node[inside] {$\pi_1$} cycle;

            \draw (5,1.5) -- node[above] {$\pi_3$} (2,1.5) -- node[left] {$\pi_0$} (2,-1.5) -- node[below] {$\pi_3$} (5,-1.5) -- node[right] {$\pi_0$} cycle;
            \draw (5,1.5) -- node[inside] {$c_5$} (4,0.5);
            \draw (2,-1.5) -- node[inside] {$c_5$} (3,-0.5);
            \draw (5,-1.5) -- node[inside] {$c_5$} (4,-0.5);
            \draw (2,1.5) -- node[inside] {$c_5$} (3,0.5);
            \draw (4,0.5) -- node[inside] {$\pi_4$} (3,0.5) -- node[inside] {$\pi_1$} (3,-0.5) -- node[inside] {$\pi_4$} (4,-0.5) -- node[inside] {$\pi_1$} cycle;
        \end{tikzpicture}
    \end{equation*}
    Following \zcref{lem:low_cases_pi0_to_pi1}, we can determine that
    \begin{equation*}
        \|\pi_0\pi_2=\pi_2\pi_0\| \leq 1 + 2\|\pi_0^{c_4}=\pi_1\| + 2\|\pi_2^{c_4}=\pi_3\| = 1 + 16 + 4 = 21
    \end{equation*}
    and
    \begin{equation*}
        \|\pi_0\pi_3=\pi_3\pi_0\| \leq \|\pi_1\pi_4 + \pi_4\pi_1\| + 2\|\pi_0^{c_5}=\pi_1\| + 2\|\pi_3^{c_5}=\pi_4\| = 11 + 88 + 16 = 115
    \end{equation*}
    For higher values of $n$, we can employ a similar method to \zcref{lem:pi0^ci_pi1}
    by adapting our current diagram for $c_{n+2}\pi_{n+1}=\pi_nc_{n+2}$ (see \zcref{lem:cipij_pij-1ci}).
    Instead of using the ``skeleton''
    \begin{equation*}
        \begin{tikzpicture}
            \draw (-2,2) -- node[above] {$\pi_{n+1}$} (2,2) -- node[right] {$c_{n+2}$} (2,-2) -- node[below] {$\pi_n$} (-2,-2) -- node[left] {$c_{n+2}$} cycle;
            \draw (-2,-2) -- node[above, scale=0.7] {$x_0^{n-1}$} (-1,-1.5);
            \draw (2,-2) -- node[above, scale=0.7] {$x_0^{n-1}$} (1,-1.5);
            \draw (-2,2) -- node[inside] {$x_1^{n-1}$} (-1,1);
            \draw (2,2) -- node[inside] {$x_1^{n-1}$} (1,1);
            \draw (-1,1) -- node[inside] {$\pi_2$} (1,1) -- node[inside] {$c_3$} (1,-1.5) -- node[inside] {$\pi_1$} (-1,-1.5) -- node[inside] {$c_3$} cycle;
        \end{tikzpicture}
    \end{equation*}
    we introduce $x_2$- and $x_3$-letters that will cancel nicely with cells in the highlighted subdiagrams in diagram (\ref{eq:skeleton_pi0pin_pinpi0}).
    \begin{equation*}
        \begin{tikzpicture}
            \draw (-2,2) -- node[above] {$\pi_{n+1}$} (2,2) -- node[right] {$c_{n+2}$} (2,-2) -- node[below] {$\pi_n$} (-2,-2) -- node[left] {$c_{n+2}$} cycle;
            \draw (-2,-2) -- node[above, scale=0.7] {$x_2^{n-3}$} (-1,-1.5);
            \draw (2,-2) -- node[above, scale=0.7] {$x_2^{n-3}$} (1,-1.5);
            \draw (-2,2) -- node[inside] {$x_3^{n-3}$} (-1,1);
            \draw (2,2) -- node[inside] {$x_3^{n-3}$} (1,1);
            \draw (-1,1) -- node[inside] {$\pi_4$} (1,1) -- node[inside] {$c_5$} (1,-1.5) -- node[inside] {$\pi_3$} (-1,-1.5) -- node[inside] {$c_5$} cycle;
        \end{tikzpicture}
    \end{equation*}
    Plugging this in the diagram (\ref{eq:skeleton_pi0pin_pinpi0}) together with the skeletons of the diagrams supplied by
    \zcref{lem:pi0^ci_pi1}, we obtain the following.
    
    \begin{equation*}
        \begin{tikzpicture}[scale=0.85]
            \fill[pattern=checkerboard, pattern color=gray!30] (-5,5) -- (-1.5,4) -- (-1.5,1.5) -- (-4,1.5) -- cycle;
            \fill[pattern=checkerboard, pattern color=gray!30] (5,5) -- (1.5,4) -- (1.5,1.5) -- (4,1.5) -- cycle;
            \fill[pattern=checkerboard, pattern color=gray!30] (5,-5) -- (1.5,-4) -- (1.5,-1.5) -- (4,-1.5) -- cycle;
            \fill[pattern=checkerboard, pattern color=gray!30] (-5,-5) -- (-1.5,-4) -- (-1.5,-1.5) -- (-4,-1.5) -- cycle;
            \draw (-5,5) -- node[above] {$\pi_n$} (5,5) -- node[right] {$\pi_0$} (5,-5) -- node[below] {$\pi_n$} (-5,-5) -- node[left] {$\pi_0$} cycle;
            \draw (-5,5) -- node[inside] {$c_{n+2}$} (-1.5,1.5);
            \draw (5,5) -- node[inside] {$c_{n+2}$} (1.5,1.5);
            \draw (5,-5) -- node[inside] {$c_{n+2}$} (1.5,-1.5);
            \draw (-5,-5) -- node[inside] {$c_{n+2}$} (-1.5,-1.5);
            \draw (-1.5,1.5) -- node[inside] {$\pi_{n+1}$} (1.5,1.5) -- node[inside] {$\pi_1$} (1.5,-1.5) -- node[inside] {$\pi_{n+1}$} (-1.5,-1.5) -- node[inside] {$\pi_1$} cycle;

            \draw (-1.5,4) -- node[inside] {$\pi_3$} (1.5,4) -- node[inside] {$c_5$} (1.5,3) -- node[inside] {$\pi_4$} (-1.5,3) -- node[inside] {$c_5$} cycle;
            \draw (-5,5) -- node[inside] {$x_2^{n-3}$} (-1.5,4);
            \draw (5,5) -- node[inside] {$x_2^{n-3}$} (1.5,4);
            \draw (-1.5,1.5) -- node[inside] {$x_3^{n-3}$} (-1.5,3);
            \draw (1.5,1.5) -- node[inside] {$x_3^{n-3}$} (1.5,3);

            \draw (-1.5,-4) -- node[inside] {$\pi_3$} (1.5,-4) -- node[inside] {$c_5$} (1.5,-3) -- node[inside] {$\pi_4$} (-1.5,-3) -- node[inside] {$c_5$} cycle;
            \draw (-5,-5) -- node[inside] {$x_2^{n-3}$} (-1.5,-4);
            \draw (5,-5) -- node[inside] {$x_2^{n-3}$} (1.5,-4);
            \draw (-1.5,-1.5) -- node[inside] {$x_3^{n-3}$} (-1.5,-3);
            \draw (1.5,-1.5) -- node[inside] {$x_3^{n-3}$} (1.5,-3);

            \draw (-3,1.5) -- node[inside] {$c_5$} (-4,1.5)
                -- node[inside, pos=0.15] {$c_4$}
                    node[inside, pos=0.3] {$x_3$}
                    node[inside, pos=0.5] {$\pi_1$}
                    node[inside, pos=0.7] {$x_3$}
                    node[inside, pos=0.85] {$c_4$} (-4,-1.5)
                -- node[inside] {$c_5$} (-3,-1.5)
                -- node[inside] {$\pi_1$} cycle;
            \draw (-1.5,1.5) -- node[inside] {$x_3^{n-3}$} (-3,1.5);
            \draw (-1.5,-1.5) -- node[inside] {$x_3^{n-3}$} (-3,-1.5);
            \draw (-5,5) -- node[inside] {$x_2^{n-3}$} (-4,1.5);
            \draw (-5,-5) -- node[inside] {$x_2^{n-3}$} (-4,-1.5);

            \draw (3,1.5) -- node[inside] {$c_5$} (4,1.5)
                -- node[inside, pos=0.15] {$c_4$}
                    node[inside, pos=0.3] {$x_3$}
                    node[inside, pos=0.5] {$\pi_1$}
                    node[inside, pos=0.7] {$x_3$}
                    node[inside, pos=0.85] {$c_4$} (4,-1.5)
                -- node[inside] {$c_5$} (3,-1.5)
                -- node[inside] {$\pi_1$} cycle;
            \draw (1.5,1.5) -- node[inside] {$x_3^{n-3}$} (3,1.5);
            \draw (1.5,-1.5) -- node[inside] {$x_3^{n-3}$} (3,-1.5);
            \draw (5,5) -- node[inside] {$x_2^{n-3}$} (4,1.5);
            \draw (5,-5) -- node[inside] {$x_2^{n-3}$} (4,-1.5);
        \end{tikzpicture}
    \end{equation*}
    The checkered regions are pairs of mirror images. Thus they can be collapsed. We do so, and add more detail to the remaining regions.
    \begin{equation*}
        \begin{tikzpicture}
            \fill[blue!20] (-5,5) -- (5,5) -- (1,4.75) -- (-1,4.75) -- cycle;
            \fill[blue!20] (-5,-5) -- (5,-5) -- (1,-4.75) -- (-1,-4.75) -- cycle;
            \fill[pattern=checkerboard, pattern color=gray!40] (-1,-1) rectangle (1,1);
            \fill[pattern=dots, pattern color=gray!50] (5,5) -- (1,4.75) -- (1,3.25) -- (3,3) -- cycle;
            \fill[pattern=dots, pattern color=gray!50] (-5,5) -- (-1,4.75) -- (-1,3.25) -- (-3,3) -- cycle;
            \fill[pattern=dots, pattern color=gray!50] (5,-5) -- (1,-4.75) -- (1,-3.25) -- (3,-3) -- cycle;
            \fill[pattern=dots, pattern color=gray!50] (-5,-5) -- (-1,-4.75) -- (-1,-3.25) -- (-3,-3) -- cycle;
            \draw (-5,5) -- node[above] {$\pi_n$} (5,5) -- node[right] {$\pi_0$} (5,-5) -- node[below] {$\pi_n$} (-5,-5) -- node[left] {$\pi_0$} cycle;
            \draw (-5,5) -- node[above right, pos=0.4] {$x_2^{n-5}$} node[inside, scale=0.7, pos=0.75] {$x_2$} node[inside, scale=0.7, pos=0.925] {$x_2$} (-2,2) -- node[inside] {$c_5$} (-1,1) -- node[inside, scale=0.5] {$x_3^{n-3}$} (-0.5,0.5);
            \draw (5,5) -- node[above left, pos=0.4] {$x_2^{n-5}$} node[inside, scale=0.7, pos=0.75] {$x_2$} node[inside, scale=0.7, pos=0.925] {$x_2$} (2,2) -- node[inside] {$c_5$} (1,1) -- node[inside, scale=0.5] {$x_3^{n-3}$} (0.5,0.5);
            \draw (5,-5) -- node[below left, pos=0.4] {$x_2^{n-5}$} node[inside, scale=0.7, pos=0.75] {$x_2$} node[inside, scale=0.7, pos=0.925] {$x_2$} (2,-2) -- node[inside] {$c_5$} (1,-1) -- node[inside, scale=0.5] {$x_3^{n-3}$} (0.5,-0.5);
            \draw (-5,-5) -- node[below right, pos=0.4] {$x_2^{n-5}$} node[inside, scale=0.7, pos=0.75] {$x_2$} node[inside, scale=0.7, pos=0.925] {$x_2$} (-2,-2) -- node[inside] {$c_5$} (-1,-1) -- node[inside, scale=0.5] {$x_3^{n-3}$} (-0.5,-0.5);

            \draw (-2,2) -- node[inside] {$\pi_3$} (2,2)
                -- node[inside, scale=0.7, pos=0.15] {$c_4$}
                    node[inside, scale=0.7, pos=0.35] {$x_3$}
                    node[inside, scale=0.7, pos=0.5] {$\pi_1$}
                    node[inside, scale=0.7, pos=0.65] {$x_3$}
                    node[inside, scale=0.7, pos=0.85] {$c_4$} (2,-2) -- node[inside] {$\pi_3$} (-2,-2)
                -- node[inside, scale=0.7, pos=0.15] {$c_4$}
                    node[inside, scale=0.7, pos=0.35] {$x_3$}
                    node[inside, scale=0.7, pos=0.5] {$\pi_1$}
                    node[inside, scale=0.7, pos=0.65] {$x_3$}
                    node[inside, scale=0.7, pos=0.85] {$c_4$} cycle;

            \draw (-1,1) -- node[inside] {$\pi_4$} (1,1) -- node[inside] {$\pi_1$} (1,-1) -- node[inside] {$\pi_4$} (-1,-1) -- node[inside] {$\pi_1$} cycle;
            \draw (-0.5,0.5) -- node[inside, scale=0.7] {$\pi_{n+1}$} (0.5,0.5) -- node[inside, scale=0.7] {$\pi_1$} (0.5,-0.5) -- node[inside, scale=0.7] {$\pi_{n+1}$} (-0.5,-0.5) -- node[inside, scale=0.7] {$\pi_1$} cycle;

            \begin{scope}[every node/.style={scale=0.7}]
                \draw (1,1) -- node[inside] {$c_5$} (1.5,0.5) -- node[inside] {$x_0$} (2,1);
                \draw (1,-1) -- node[inside] {$c_5$} (1.5,-0.5) -- node[inside] {$x_0$} (2,-1);
                \draw (1.5,0.5) -- node[inside] {$\pi_2$} (1.5,-0.5);
                \draw (-1,1) -- node[inside] {$c_5$} (-1.5,0.5) -- node[inside] {$x_0$} (-2,1);
                \draw (-1,-1) -- node[inside] {$c_5$} (-1.5,-0.5) -- node[inside] {$x_0$} (-2,-1);
                \draw (-1.5,0.5) -- node[inside] {$\pi_2$} (-1.5,-0.5);

                \draw (2,1) -- node[inside] {$x_3$} (2.5,2);
                \draw (2.5,2.5) -- node[inside] {$c_4$} (3,2.5) -- node[inside] {$x_0$} (3,2) -- node[inside] {$c_5$} (2.5,2) -- node[inside] {$c_5$} cycle;
                \draw[dotted] (3,2.5) -- (3.5,3);
                \draw (3.5,3.5) -- node[inside] {$c_4$} (4,3.5) -- node[inside] {$x_0$} (4,3) -- node[inside] {$c_5$} (3.5,3) -- node[inside] {$c_5$} cycle;
                \draw (4.5,2) -- node[inside] {$x_3$} (4,3.5) -- node[inside] {$c_4$} ++(-0.5,0);
                \draw (5,5) -- node[inside] {$c_5$} (4.5,2) -- node[inside] {$\pi_1$} (4.5,-2) -- node[inside] {$c_5$} (5,-5);

                \draw (2,-1) -- node[inside] {$x_3$} (2.5,-2);
                \draw (2.5,-2.5) -- node[inside] {$c_4$} (3,-2.5) -- node[inside] {$x_0$} (3,-2) -- node[inside] {$c_5$} (2.5,-2) -- node[inside] {$c_5$} cycle;
                \draw[dotted] (3,-2.5) -- (3.5,-3);
                \draw (3.5,-3.5) -- node[inside] {$c_4$} (4,-3.5) -- node[inside] {$x_0$} (4,-3) -- node[inside] {$c_5$} (3.5,-3) -- node[inside] {$c_5$} cycle;
                \draw (4.5,-2) -- node[inside] {$x_3$} (4,-3.5) -- node[inside] {$c_4$} ++(-0.5,0);
            \end{scope}
            \begin{scope}[every node/.style={scale=0.7}]
                \draw (-2,1) -- node[inside] {$x_3$} (-2.5,2);
                \draw (-2.5,2.5) -- node[inside] {$c_4$} (-3,2.5) -- node[inside] {$x_0$} (-3,2) -- node[inside] {$c_5$} (-2.5,2) -- node[inside] {$c_5$} cycle;
                \draw[dotted] (-3,2.5) -- (-3.5,3);
                \draw (-3.5,3.5) -- node[inside] {$c_4$} (-4,3.5) -- node[inside] {$x_0$} (-4,3) -- node[inside] {$c_5$} (-3.5,3) -- node[inside] {$c_5$} cycle;
                \draw (-4.5,2) -- node[inside] {$x_3$} (-4,3.5);
                \draw (-5,5) -- node[inside] {$c_5$} (-4.5,2) -- node[inside] {$\pi_1$} (-4.5,-2) -- node[inside] {$c_5$} (-5,-5);

                \draw (-2,-1) -- node[inside] {$x_3$} (-2.5,-2);
                \draw (-2.5,-2.5) -- node[inside] {$c_4$} (-3,-2.5) -- node[inside] {$x_0$} (-3,-2) -- node[inside] {$c_5$} (-2.5,-2) -- node[inside] {$c_5$} cycle;
                \draw[dotted] (-3,-2.5) -- (-3.5,-3);
                \draw (-3.5,-3.5) -- node[inside] {$c_4$} (-4,-3.5) -- node[inside] {$x_0$} (-4,-3) -- node[inside] {$c_5$} (-3.5,-3) -- node[inside] {$c_5$} cycle;
                \draw (-4.5,-2) -- node[inside] {$x_3$} (-4,-3.5);
            \end{scope}

            \draw (-2.5,2.5) -- node[inside] {$\pi_4$} (2.5,2.5);
            \draw (-3,3) -- node[inside] {$\pi_5$} (3,3);
            \begin{scope}[every node/.style={scale=0.7}]
                \draw (-3,3) -- node[inside, rectangle] {$x_3$} (-1,3.25);
                \draw (3,3) -- node[inside, rectangle] {$x_3$} (1,3.25);
                \draw (-1,3.25) -- node[inside] {$\pi_4$} (1,3.25);
                \draw (-1,3.25) -- node[inside] {$x_2$} (-1,3.5) -- node[inside] {$\pi_5$} (1,3.5) -- node[inside] {$x_2$} (1,3.25);
                \draw (-1,3.5) -- node[inside] {$x_3$} (-1,3.75) -- node[inside] {$\pi_4$} (1,3.75) -- node[inside] {$x_3$} (1,3.5);
                \draw (-1,3.75) -- node[inside] {$x_2$} (-1,4) -- node[inside] {$\pi_5$} (1,4) -- node[inside] {$x_2$} (1,3.75);
                \node at (0,4.3) {$\vdots$};
                \draw (-1,4.5) -- node[inside] {$x_2$} (-1,4.75) -- node[inside] {$\pi_5$} (1,4.75) -- node[inside] {$x_2$} (1,4.5) -- node[inside] {$\pi_4$} cycle;
            \end{scope}
            \draw (-5,5) -- node[sloped, below] {$x_{n-2}x_{n-3}\cdots x_4$} (-1,4.75);
            \draw (5,5) -- node[sloped, below] {$x_4 \cdots x_{n-3}x_{n-2}$} (1,4.75);

            \draw (-2.5,-2.5) -- node[inside] {$\pi_4$} (2.5,-2.5);
            \draw (-3,-3) -- node[inside] {$\pi_5$} (3,-3);
            \begin{scope}[every node/.style={scale=0.7}]
                \draw (-3,-3) -- node[inside, rectangle] {$x_3$} (-1,-3.25);
                \draw (3,-3) -- node[inside, rectangle] {$x_3$} (1,-3.25);
                \draw (-1,-3.25) -- node[inside] {$\pi_4$} (1,-3.25);
                \draw (-1,-3.25) -- node[inside] {$x_2$} (-1,-3.5) -- node[inside] {$\pi_5$} (1,-3.5) -- node[inside] {$x_2$} (1,-3.25);
                \draw (-1,-3.5) -- node[inside] {$x_3$} (-1,-3.75) -- node[inside] {$\pi_4$} (1,-3.75) -- node[inside] {$x_3$} (1,-3.5);
                \draw (-1,-3.75) -- node[inside] {$x_2$} (-1,-4) -- node[inside] {$\pi_5$} (1,-4) -- node[inside] {$x_2$} (1,-3.75);
                \node at (0,-4.15) {$\vdots$};
                \draw (-1,-4.5) -- node[inside] {$x_2$} (-1,-4.75) -- node[inside] {$\pi_5$} (1,-4.75) -- node[inside] {$x_2$} (1,-4.5) -- node[inside] {$\pi_4$} cycle;
            \end{scope}
            \draw (-5,-5) -- node[sloped, above] {$x_{n-2}x_{n-3}\cdots x_4$} (-1,-4.75);
            \draw (5,-5) -- node[sloped, above] {$x_4 \cdots x_{n-3}x_{n-2}$} (1,-4.75);
        \end{tikzpicture}
    \end{equation*}
    The checkered region is obsolete as its boundary is one of the defining relations of $V$.
    The highlighted regions can be filled as shown in \zcref{lem:pi2tox1n-1_pin+1}. Each
    takes $n-5$ cells. The dotted regions have a boundary in $F$, specifically
    \begin{equation*}
        (x_3^{-1}x_2)^{n-5}x_4x_5\cdots x_{n-2} = x_2^{n-5}
    \end{equation*}
    which can be written as
    \begin{equation*}
        (x_0^{-2}x_1^{-1}x_0x_1x_0)^{n-5}x_0^{-3}(x_1x_0^{-1})^{n-5}x_0^{1}x_0^{n-3} = x_0^{-1}x_1^{n-5}x_0
    \end{equation*}
    and simplified to
    \begin{equation*}
        x_0^{-1}(x_0^{-1}x_1^{-1}x_0x_1)^{n-5}x_0x_0^{-3}(x_1x_0^{-1})^{n-5}x_0^{n-2} = x_0^{-1}x_1^{n-5}x_0
    \end{equation*}
    \begin{equation*}
        \iff (x_0^{-1}x_1^{-1}x_0x_1)^{n-5}x_0^{-2}(x_1x_0^{-1})^{n-5}x_0^{n-3} = x_1^{n-5}
    \end{equation*}
    We can read that this has length
    \begin{equation*}
        4(n-5)+2 + 2(n-5) + (n-3) + n-5 = 8n-36
    \end{equation*}
    Thus each dotted region requires $O((8n-36)^2)$ cells to fill. Now we count
    all the cells in the diagram.
    \begin{align*}
        \|\pi_0\pi_n=\pi_n\pi_0\| &\leq 1 + 2\|\pi_3^{c_5}=\pi_4\|\\
            &\qquad + 2(\|\pi_3^{x_2}=\pi_4\|+\|\pi_4^{x_2}=\pi_5\|)\\
            &\qquad + 2(n-5)\Big(\|\pi_4^{x_3}=\pi_5\| + \|\pi_4^{x_2}=\pi_5\|\Big)\\
            &\qquad + 2(n-5) + 2(45(n-2)+44)\\
            &\qquad + 4O((8n-36)^2)\\
            &\leq 1 + 2 + 2(1+6) + 2(n-5)(1+6)\\
            &\qquad + 2(n-5) + 2(45(n-2)+44) + 4O((8n-36)^2)\\
            &\leq 106n-155 + 4O((8n-36)^2) = O(n^2)
    \end{align*}
\end{proof}

\begin{lemma}\label{lem:piipij_pijpii}
    For any integers $j-2\geq i\geq 0$,
    \begin{equation*}
        \|\pi_i\pi_j = \pi_j\pi_i\| = O(n^2)
    \end{equation*}
    with $n=j-i+1$.
\end{lemma}
\begin{proof}
    \zcref{pi0pin} handles the case $i=0$.
    If $i>0$, this is again a very similar argument to previous constructions. Using conjugations, we can reduce
    the relation $\pi_i\pi_j = \pi_j\pi_i$ to
    \begin{equation*}
        \pi_1\pi_{j-i+1} = \pi_{j-i+1}\pi_1
    \end{equation*}
    at no cost. Then it remains to apply \zcref{lem:pinpi1_pi1pin} with $n=j-i+1$
\end{proof}

We end this subsection with relations that do not have the same amount of
letters on both sides of the equality.

\begin{lemma}\label{lem:pinxn+1_xnpin+1pin}
    For any integer $n\geq0$
    \begin{equation*}
        \|\pi_nx_{n+1} = x_n\pi_{n+1}\pi_n\| = 8
    \end{equation*}
\end{lemma}
\begin{proof}
    If $n\geq 2$, then this follows directly from conjugation of a standard relation.
    If $n=0$, we have the van Kampen diagram
    \begin{equation*}
        \begin{tikzpicture}
            \fill[blue!20] (72:1) -- (72:2) -- (144:2) -- (144:1) -- cycle;
            \fill[blue!20] (0:1) -- (0:2) -- (288:2) -- (288:1) -- cycle;
            \fill[blue!20] (0:1) -- (72:1) -- (144:1) -- (216:1) -- (288:1) -- cycle;
            \fill[pattern=checkerboard, pattern color = gray!50] (0:1) -- (0:2) -- (72:2) -- (72:1) -- cycle;
            \draw (0:2) -- node[right] {$\pi_0$} (72:2)
                -- node[above] {$x_1$} (144:2)
                -- node[left] {$\pi_0$} (216:2)
                -- node[below] {$x_0$} (288:2)
                -- node[right] {$\pi_1$} cycle;
            \draw (0:1) -- node[right] {$\pi_1$} (72:1)
                -- node[above] {$x_2$} (144:1)
                -- node[left] {$\pi_1$} (216:1)
                -- node[below] {$x_1$} (288:1)
                -- node[right] {$\pi_2$} cycle;
            \draw (0:1) -- node[above] {$\scriptstyle c_3$} (0:2);
            \draw (72:1) -- node[left] {$\scriptstyle c_3$} (72:2);
            \draw (144:1) -- node[below] {$\scriptstyle c_2$} (144:2);
            \draw (216:1) -- node[below] {$\scriptstyle c_2$} (216:2);
            \draw (288:1) -- node[right] {$\scriptstyle c_3$} (288:2);
        \end{tikzpicture}
    \end{equation*}
    Where each highlighted region is 1 cell and the checkered region can be obtained in
    5 cells (\zcref{lem:low_cases_pi0_to_pi1}). Finally, if $n=1$, we have a standard relation of the
    presentation.
\end{proof}

\begin{lemma}\label{lem:pinxn_xn+1pinpin+1}
    For any integer $n>0$
    \begin{equation*}
        \|\pi_nx_n = x_{n+1}\pi_n\pi_{n+1}\| = 11
    \end{equation*}
\end{lemma}
\begin{proof}
    Starting from the relation in \zcref{lem:pinxn+1_xnpin+1pin}, we multiply on the left and on the right with good choices of $\pi_n$ and $\pi_{n+1}$, and apply the fact that $\|\pi_n^2 = 1\| = 1,\; \forall n\geq 1$.
    \begin{align*}
        \pi_nx_{n+1} &= x_n\pi_{n+1}\pi_n\\
        \iff \pi_n^2x_{n+1}\pi_n\pi_{n+1} &= \pi_nx_n\pi_{n+1}\pi_n^2\pi_{n+1}\\
        \iff 1x_{n+1}\pi_n\pi_{n+1} &= \pi_nx_n\pi_{n+1}1\pi_{n+1}\\
        \iff x_{n+1}\pi_{n}\pi_{n+1} &= \pi_nx_n1
    \end{align*}
    This cost 3 relations, thus,
    \begin{equation*}
        \|\pi_nx_n = x_{n+1}\pi_n\pi_{n+1}\| = 3 + \|\pi_nx_{n+1} = x_n\pi_{n+1}\pi_n\| \underset{\text{\zcref{lem:pinxn+1_xnpin+1pin}}}{=} 3+8 = 11
    \end{equation*}
\end{proof}

\subsection{\texorpdfstring{Remaining relations with $c$-letters and $\pi$-letters}{Remaining relations with c-letters and pi-letters}}
The only relevant relations that are left, are those containing $c$-letters and
$\pi$-letters of lower subscript. Here things get uglier. The author has not found a way
to conserve the ideal $O(n^2)$ bound we have so far.

\begin{lemma}\label{lem:c1pi0_pi0pi1c2}
    \begin{equation*}
        \|c_2\pi_0 = \pi_0\pi_1c_2^2\| = 4
    \end{equation*}
\end{lemma}
\begin{proof}
    The necessary relations are given by the following van Kampen diagram.
    \begin{equation*}
        \begin{tikzpicture}
            \draw[<-] (-60:2) arc (-50:25:1.55) node (a) {}; \node at (-30:1.5) {$c_2$};
            \draw[<-] (a.center) arc (25:90:1.55) node (b) {}; \node at (30:1) {$\pi_1$};
            \draw[<-] (b.center) arc (90:160:1.55) node (c) {}; \node at (150:1) {$c_2$};
            \draw[->] (c.center) arc (160:195:1.55) node (d) {}; \node at (205:1.15) {$\pi_1$};
            \draw[<-] (d.center) arc (195:230:1.55); \node at (230:1.6) {$c_2$};

            \node[scale=0.7] at (195:1.75) {$\pi_1$};

            \fill[blue!20] (90:2) arc (90:165:2) -- (c.center) arc (160:90:1.55) -- cycle;
            \fill[pattern=dots, pattern color=gray!60] (15:2) arc (15:90:2) -- (b.center) arc (90:25:1.55) -- cycle;
            \fill[pattern=dots, pattern color=gray!60] (c.center) arc (160:230:1.55) arc (240:165:2);

            \draw[->] (-60:2) arc (-60:-22:2); \node at (-41:2.2) {$c_2$};
            \draw[->] (-22:2) arc (-22:15:2); \node at (-3:2.2) {$c_2$};
            \draw[<-] (15:2) arc (15:90:2); \node at (52.5:2.2) {$\pi_0$};
            \draw[<-] (90:2) arc (90:165:2); \node at (127.5:2.2) {$c_2$};
            \draw[->] (165:2) arc (165:240:2); \node at (202.5:2.2) {$\pi_0$};
            \draw[<-] (240:2) arc (-120:-60:2); \node at (-90:1.8) {$\pi_1$};

            \draw[->] (240:2) arc (-160:-20:1.07); \node at (-90:2.7) {$\pi_1$};

            \draw[->] (15:2) -- node[above, scale=0.7] {$c_2$} (a.center);
            \draw[->] (165:2) -- node[above, scale=0.7] {$c_2$} (c.center);
            \draw[->] (90:2) -- node[inside, scale=0.7] {$c_2$} (b.center);

            \draw[<-] (b.center) arc (90:160:1.55);

            \draw[<-] (c.center) arc (50:-55:0.57);
        \end{tikzpicture}
    \end{equation*}
    The dotted regions are definitions, and the highlighted region is neither a cell nor a definition, but just drawn as such for
    visual purposes.
\end{proof}

\begin{lemma}\label{lem:cnpi0_pi0dotspin-1cn2}
    For any integer $n\geq 1$
    \begin{equation*}
        \|c_n\pi_0 = \pi_0\cdots \pi_{n-1}c_n^2\| = O(n^3)
    \end{equation*}
\end{lemma}
\begin{proof}
    If $n=2$, we use \zcref{lem:c1pi0_pi0pi1c2}. For higher values of $n$,
    consider the van Kampen diagram
    \begin{equation*}
        \begin{tikzpicture}[scale=0.88]
            \fill[pattern=dots, pattern color=gray!50] (-5,0) -- (-5,2) -- (5,2) -- (5,0) -- cycle;
            \draw[->] (-5,2) -- node[above] {$\pi_0$} (5,2);
            \draw[->] (-5,-2) -- node[below,pos=0.05] {$\pi_0$}
                                node[below,pos=0.15] {$\pi_1$}
                                node[below,pos=0.25] {$\pi_2$}
                                node[below,pos=0.65] {$\pi_{n-2}$}
                                node[below,pos=0.8] {$\pi_{n-1}\pi_n$}
                                node[below,pos=0.95] {$c_{n+1}$} (5,-2);
            \draw[<->] (-5,-2) -- node[inside,pos=0.25] {$x_n$} node[inside,pos=0.75] {$c_n$} (-5,2);
            \draw[<->] (5,-2) -- node[inside,pos=0.25] {$x_n$} node[inside,pos=0.75] {$c_n$} (5,2);
            \draw[->] (-5,0) -- node[above,pos=0.05] {$\pi_0$}
                                node[above,pos=0.15] {$\pi_1$}
                                node[above,pos=0.25] {$\pi_2$}
                                node[above,pos=0.65] {$\pi_{n-2}$}
                                node[above,pos=0.8] {$\pi_{n-1}$}
                                node[above,pos=0.95] {$c_n$} (5,0);
            \draw[->] (-4,0) -- node[inside] {$x_n$} (-4,-2);
            \draw[->] (-3,0) -- node[inside] {$x_n$} (-3,-2);
            \draw[->] (4,0) -- node[inside] {$x_{n-1}$} (4,-2);
            \draw[->] (2,0) -- node[inside] {$x_n$} (2,-2);
            \draw[->] (1,0) -- node[inside] {$x_n$} (1,-2);

            \draw[->] (-5,-2) to[out=130,in=-130] node[left] {$c_{n+1}$} (-5,2);
            \draw[->] (5,-2) to[out=50,in=-50] node[right] {$c_{n+1}$} (5,2);
        \end{tikzpicture}
    \end{equation*}
    and notice that the dotted region can be filled with same construction as the outer part inductively.
    In order to make the subscripts readable, we have opted to draw the diagram with boundary equation
    \begin{equation*}
        c_{n+1}\pi_0 = \pi_0\cdots\pi_n c_{n+1}^2
    \end{equation*}
    One could almost finish the proof here if lowering the amount of cells is not of interest.
    But to get the $O(n^3)$ complexity, one needs to leverage what happens when multiple (shrinking)
    copies of this van Kampen diagram are placed together like \textit{Matryoshka} dolls.
    Two copies would give the
    following van Kampen diagram.
    \begin{equation*}
        \begin{tikzpicture}[scale=0.88]
            \fill[blue!20] (2,-2) -- (4,-2) -- (4,1) -- (1,1) -- (1,0) -- (2,0) -- cycle;
            \fill[pattern=checkerboard, pattern color=gray!20] (4,-2) -- (5,-2) -- (5,0) to[out=147,in=-90] (4.5,1) -- (4,1) -- cycle;

            \draw[->] (-5,2) -- node[above] {$\pi_0$} (5,2);
            \draw[->] (-5,-2) -- node[below,pos=0.05] {$\pi_0$}
                                node[below,pos=0.15] {$\pi_1$}
                                node[below,pos=0.25] {$\pi_2$}
                                node[below,pos=0.55] {$\pi_{n-3}$}
                                node[below,pos=0.65] {$\pi_{n-2}$}
                                node[below,pos=0.8] {$\pi_{n-1}\pi_n$}
                                node[below,pos=0.95] {$c_{n+1}$} (5,-2);
            \draw[<->] (-5,-2) -- node[inside,pos=0.25] {$x_n$} node[inside,pos=0.75] {$c_n$} (-5,2);
            \draw[<->] (5,-2) -- node[inside,pos=0.25] {$x_n$} node[inside,pos=0.75] {$c_n$} (5,2);
            \draw[->] (-5,0) -- node[below,pos=0.05] {$\pi_0$}
                                node[below,pos=0.15] {$\pi_1$}
                                node[below,pos=0.25] {$\pi_2$}
                                node[below,pos=0.55] {$\pi_{n-3}$}
                                node[below,pos=0.65] {$\pi_{n-2}$}
                                node[below,pos=0.8] {$\pi_{n-1}$}
                                node[below,pos=0.95] {$c_n$} (5,0);
            \draw[->] (-4,0) -- node[inside] {$x_n$} (-4,-2);
            \draw[->] (-3,0) -- node[inside] {$x_n$} (-3,-2);
            \draw[->] (4,0) -- node[inside, rectangle, inner sep=1] {$x_{n-1}$} (4,-2);
            \draw[->] (2,0) -- node[inside] {$x_n$} (2,-2);
            \draw[->] (1,0) -- node[inside] {$x_n$} (1,-2);
            \draw[->] (0,0) -- node[inside] {$x_n$} (0,-2);

            \draw[->] (-5,-2) to[out=130,in=-130] node[left] {$c_{n+1}$} (-5,2);
            \draw[->] (5,-2) to[out=50,in=-50] node[right] {$c_{n+1}$} (5,2);

            \draw[->] (-4.5,1) -- node[above,pos=0.03] {$\scriptstyle  \pi_0$}
                                node[above,pos=0.11] {$\scriptstyle \pi_1$}
                                node[above,pos=0.22] {$\scriptstyle \pi_2$}
                                node[above,pos=0.55] {$\scriptstyle \pi_{n-3}$}
                                node[above,pos=0.78] {$\scriptstyle \pi_{n-2}$}
                                node[above,pos=0.97, scale=0.5] {$c_{n-1}$} (4.5,1);
            \draw[<->] (-5,0) to[out=33,in=-33] node[inside, scale=0.5,pos=0.25] {$x_{n-1}$} node[right,pos=0.75] {$\scriptstyle c_{n-1}$} (-5,2);
            \draw[<->] (5,0) to[out=147,in=-147] node[inside, scale=0.5,pos=0.25] {$x_{n-1}$} node[left,pos=0.75] {$\scriptstyle c_{n-1}$} (5,2);

            \draw[->] (-4,1) -- node[inner sep=0pt, fill=white] {$\scriptstyle x_{n-1}$} (-4,0);
            \draw[->] (-3,1) -- node[inner sep=0pt, fill=white] {$\scriptstyle x_{n-1}$} (-3,0);

            \draw[->] (0,1) -- node[inner sep=0pt, fill=white] {$\scriptstyle x_{n-1}$} (0,0);
            \draw[->] (1,1) -- node[inner sep=0pt, fill=white] {$\scriptstyle x_{n-1}$} (1,0);
            \draw[->] (4,1) -- node[inner sep=0pt, fill=white] {$\scriptstyle x_{n-2}$} (4,0);
        \end{tikzpicture}
    \end{equation*}
    If we would continue to stack these copies, we can distinguish multiple regions.
    \begin{enumerate}
        \item The two regions on the sides of the diagram consist of strings of diagrams
        with boundary equations
        \begin{equation*}
            x_ic_{i+1} = c_i,\qquad i\in\{2,3,\dots,n\}
        \end{equation*}
        \item The checkered region consists of a string of diagrams with boundary equations
        \begin{equation*}
            x_{i-1}c_{i+1} = c_ix_i,\qquad i\in\{2,3,\dots,n\}
        \end{equation*}
        \item The highlighted region consists of a string of diagrams with boundary equations
        \begin{equation*}
            \pi_{i-1}x_{i-1} = x_i\pi_{i-1}\pi_i,\qquad i\in\{2,3,\dots,n\}
        \end{equation*}
        \item The remaining central region is a gridded ``triangle'' consisting of diagrams with boundary equations
        \begin{equation*}
            \pi_jx_i = x_i\pi_j,\qquad i\in\{2,3,\cdots,n\},\; j\in\{0,1,\dots,n-2\},\; i-j\geq 2
        \end{equation*}
    \end{enumerate}
    We will have to discuss these cases one by one.
    \begin{enumerate}
        \item The regions on the sides each consists of diagrams from \zcref{lem:cn_xncn+1}.
        They can be drawn as
        \begin{equation*}
            \begin{tikzpicture}[scale=0.8]
                \draw[->] (-5,-2) to[out=130,in=-130] node[left] {$c_{n+1}$} (-5,2);
                \draw[->] (5,-2) to[out=50,in=-50] node[right] {$c_{n+1}$} (5,2);
                \draw (-5,-2) -- node[below, sloped] {$x_n\cdots x_3x_2$} (-4,1) -- node[above] {$c_2$} (-5,2);
                \draw ($(-5,-2)!0.1!(-4,1)$) to[out=130,in=-130] (-5,2);
                \draw ($(-5,-2)!0.2!(-4,1)$) to[out=130,in=-130] (-5,2);
                \draw ($(-5,-2)!0.3!(-4,1)$) to[out=130,in=-130] (-5,2);
                \draw ($(-5,-2)!0.8!(-4,1)$) to (-5,2);
                \draw ($(-5,-2)!0.9!(-4,1)$) to (-5,2);
                \node at (-5,0.5) {$\cdots$};
                
                \draw (5,-2) -- node[below, sloped] {$x_2x_3\cdots x_n$} (4,1) -- node[above] {$c_2$} (5,2);
                \draw ($(5,-2)!0.1!(4,1)$) to[out=50,in=-50] (5,2);
                \draw ($(5,-2)!0.2!(4,1)$) to[out=50,in=-50] (5,2);
                \draw ($(5,-2)!0.3!(4,1)$) to[out=50,in=-50] (5,2);
                \draw ($(5,-2)!0.8!(4,1)$) to (5,2);
                \draw ($(5,-2)!0.9!(4,1)$) to (5,2);
                \node at (5,0.5) {$\cdots$};
            \end{tikzpicture}
        \end{equation*}
        thus requiring exactly $2(n-1)$ cells in total.
        \item The stacking of these subdiagrams is similar to previous arguments. The reader should be familiar with
        them by now so we only give the result after cancelling as many cells as possible.
        \begin{equation*}
            \begin{tikzpicture}
                \fill[blue!20] (-3,4) -- (3,4) -- (3,3) -- (-3,3) -- cycle;
                \fill[blue!20] (1,2) -- (3,3) -- (3,-5) -- (1,-4) -- cycle;
                \fill[blue!20] (-1,-4) rectangle (1,-3);
                \fill[blue!20] (-1,-1) rectangle (1,0);
                \fill[blue!20] (-1,1) rectangle (1,2);

                \draw (-3,4) -- node[above] {$c_2$} (3,4) -- node[right] {$x_2$} (3,3) -- node[above] {$c_3$} (-3,3) -- node[left] {$x_1$} cycle;
                \draw (3,3) -- node[above,sloped] {$x_3x_4\cdots x_n$} (3,-5) -- node[below] {$c_{n+1}$} (-3,-5) -- node[above,sloped] {$x_{n-1}\cdots x_3x_2$} (-3,3);

                \draw[->] (-1,2) -- node[left] {$x_1$} (-1,1) -- node[left] {$x_0^{-1}$} (-1,0) -- node[left] {$x_1$} (-1,-1);
                \draw[->] (1,2) -- node[right] {$x_2$} (1,1) -- node[right] {$x_1^{-1}$} (1,0) -- node[right] {$x_2$} (1,-1);
                \draw (-1,2) -- node[above] {$c_2$} (1,2);
                \draw (-1,1) -- node[above] {$c_3$} (1,1);
                \draw (-1,0) -- node[above] {$c_2$} (1,0);
                \draw (-1,-1) -- node[above] {$c_3$} (1,-1);

                \draw[dotted] (-1,-1) -- (-1,-2);
                \draw[dotted] (1,-1) -- (1,-2);

                \draw[->] (-1,-2) -- node[left] {$x_0^{-1}$} (-1,-3) -- node[left] {$x_1$} (-1,-4);
                \draw[->] (1,-2) -- node[right] {$x_1^{-1}$} (1,-3) -- node[right] {$x_2$} (1,-4);
                \draw (-1,-2) -- node[above] {$c_3$} (1,-2);
                \draw (-1,-3) -- node[above] {$c_2$} (1,-3);
                \draw (-1,-4) -- node[above] {$c_3$} (1,-4);

                \draw (-3,3) -- node[above] {$x_0$} (-1,2);
                \draw (3,3) -- node[above] {$x_1$} (1,2);
                \draw (-3,-5) -- node[above] {$x_0^{n-2}$} (-1,-4);
                \draw (3,-5) -- node[above] {$x_1^{n-2}$} (1,-4);
            \end{tikzpicture}
        \end{equation*}
        The regions containing cells are highlighted. The left strip does not contain any cells. It is purely made out of definitions. In the middle strip, we can count $n-1$ cells.
        The right region has a boundary equation that is completely in $F$, indeed it reads
        \begin{equation*}
            (x_1^{-1}x_2)^{n-2}x_1^{n-2} = x_3x_4\dots x_n
        \end{equation*}
        We rewrite this to find the length with respect to $\{x_0,x_1\}$.
        \begin{equation*}
            (x_1^{-1}x_0^{-1}x_1x_0)^{n-2}x_1^{n-2} = x_0^{-1}(x_0^{-1}x_1)^{n-2}x_0^{n-1}
        \end{equation*}
        with length
        \begin{equation*}
            4(n-2) + (n-2) + 1 +2(n-2) + (n-1) = 8n-14
        \end{equation*}
        such that we can conclude that this region has area bounded by $O((8n-14)^2) = O(n^2)$.
        In turn, this shows that the whole stack of diagrams has area bounded by $n-1 + O(n^2) = O(n^2)$.
        \item The highlighted region is a string of diagrams from \zcref{lem:pinxn_xn+1pinpin+1}, each requiring
        $O(1)$ cells. Stacking them as in the main diagram gives
        \begin{equation*}
            \begin{tikzpicture}
                \draw (-5,4) -- node[above] {$\pi_1$} (4,4) -- node[right] {$x_1$} (4,3) -- node[right] {$x_2$} (4,2) -- node[right] {$x_3$} (4,1) -- node[right] {$\vdots$} (4,-2) -- node[right] {$x_{n-2}$} (4,-3) -- node[right] {$x_{n-1}$} (4,-4);
                \draw (-5,4) -- node[left] {$x_2$} (-5,3) -- node[above] {$\pi_1$} (-4,3) -- node[above] {$\pi_2$} (4,3);
                \draw (-4,3) -- node[left] {$x_3$} (-4,2) -- node[above] {$\pi_2$} (-3,2) -- node[above] {$\pi_3$} (4,2);
                \draw (-3,2) -- node[left] {$x_4$} (-3,1) -- node[above] {$\pi_3$} (-2,1) -- node[above] {$\pi_4$} (4,1);
                \node at (0,0) {$\ddots$};
                \draw (4,-4) -- node[below] {$\pi_{n-1}\pi_n$} (3,-4) -- node[left] {$x_n$} (3,-3);
                \draw (4,-3) -- node[above] {$\pi_{n-1}$} (3,-3) -- node[above] {$\pi_{n-2}$} (2,-3) -- node[left] {$x_{n-1}$} (2,-2) -- node[above] {$\pi_{n-2}$} (4,-2);
            \end{tikzpicture}
        \end{equation*}
        totalling $(n-1)O(1) = O(n)$ cells.
        \item The central triangular region can further be split into two parts: the left most column and the
        rest of the triangle. To understand the rest of the triangle, we need to consider the following van Kampen diagram
        where 4 ``square'' subdiagrams are combined.
        
        \begin{equation*}\label{eq:four_big_squares}
            \begin{tikzpicture}[scale=0.9]
                \fill[blue!20] (-5,0) -- (-4,1) -- (-1,1) -- (0,0) -- (1,1) -- (4,1) -- (5,0) -- (4,-1) -- (1,-1) -- (0,0) -- (-1,-1) -- (-4,-1) -- cycle;
                \fill[pattern=checkerboard,pattern color=gray!20] (0,5) -- (1,4) -- (1,1) -- (0,0) -- (1,-1) -- (1,-4) -- (0,-5) -- (-1,-4) -- (-1,-1) -- (0,0) -- (-1,1) -- (-1,4) -- cycle;
                \draw (-5,5) -- node[above,pos=0.25] {$\pi_k$} node[above,pos=0.75] {$\pi_{k+1}$} (5,5)
                    -- node[right,pos=0.25] {$x_{m-1}$} node[right,pos=0.75] {$x_m$} (5,-5) 
                    -- node[below,pos=0.25] {$\pi_{k+1}$} node[below,pos=0.75] {$\pi_k$} (-5,-5)
                    -- node[left,pos=0.25] {$x_m$} node[left,pos=0.75] {$x_{m-1}$} cycle;
                \draw (-5,0) -- node[inside, fill=blue!20,pos=0.25] {$\pi_k$} node[inside, fill=blue!20,pos=0.75] {$\pi_{k+1}$} (5,0);
                \draw (0,5) -- node[inside,pos=0.25, rectangle, scale=0.7] {$x_{m-1}$} node[inside,pos=0.75, rectangle, scale=0.7] {$x_{m}$} (0,-5);

                \draw (-4,4) -- node[above] {$\pi_1$} (-1,4)
                    -- node[inside, rectangle] {$x_{m-k}$} (-1,1)
                    -- node[below] {$\pi_1$} (-4,1)
                    -- node[inside, rectangle] {$x_{m-k}$} cycle;
                \draw (-3,3) -- node[above] {$\pi_1$} (-2,3)
                    -- node[inside, scale=0.7] {$x_4$} (-2,2)
                    -- node[below] {$\pi_1$} (-3,2)
                    -- node[inside, scale=0.7] {$x_4$} cycle;
                \draw (-5,5) -- node[right,pos=0.25] {$x_0^{k-1}$} node[inside,scale=0.5, pos=0.75] {$x_3^{m-k-4}$} (-3,3);
                \draw (-5,0) -- node[right,pos=0.25] {$x_0^{k-1}$} node[inside, scale=0.5, pos=0.75] {$x_3^{m-k-4}$} (-3,2);
                \draw (0,5) -- node[left,pos=0.25] {$x_0^{k-1}$} node[inside,pos=0.75, scale=0.5] {$x_3^{m-k-4}$} (-2,3);
                \draw (0,0) -- node[left,pos=0.25] {$x_0^{k-1}$} node[inside,pos=0.75, scale=0.5] {$x_3^{m-k-4}$} (-2,2);
                
                \draw (4,4) -- node[above] {$\pi_1$} (1,4)
                    -- node[inside, scale=0.7, rectangle] {$x_{m-k-1}$} (1,1)
                    -- node[below] {$\pi_1$} (4,1)
                    -- node[inside, scale=0.7, rectangle] {$x_{m-k-1}$} cycle;
                \draw (3,3) -- node[above] {$\pi_1$} (2,3)
                    -- node[inside, scale=0.7] {$x_4$} (2,2)
                    -- node[below] {$\pi_1$} (3,2)
                    -- node[inside, scale=0.7] {$x_4$} cycle;
                \draw (0,5) -- node[right,pos=0.25] {$x_0^{k}$} node[inside, scale=0.5, pos=0.75] {$x_3^{m-k-5}$} (2,3);
                \draw (0,0) -- node[right,pos=0.25] {$x_0^{k}$} node[inside, scale=0.5, pos=0.75] {$x_3^{m-k-5}$} (2,2);
                \draw (5,5) -- node[left,pos=0.25] {$x_0^{k}$} node[inside, scale=0.5, pos=0.75] {$x_3^{m-k-5}$} (3,3);
                \draw (5,0) -- node[left,pos=0.25] {$x_0^{k}$} node[inside, scale=0.5, pos=0.75] {$x_3^{m-k-5}$} (3,2);

                \draw (-4,-4) -- node[below] {$\pi_1$} (-1,-4)
                    -- node[inside, scale=0.7, rectangle] {$x_{m-k+1}$} (-1,-1)
                    -- node[above] {$\pi_1$} (-4,-1)
                    -- node[inside, scale=0.7, rectangle] {$x_{m-k+1}$} cycle;
                \draw (-3,-3) -- node[below] {$\pi_1$} (-2,-3)
                    -- node[inside, scale=0.7] {$x_4$} (-2,-2)
                    -- node[above] {$\pi_1$} (-3,-2)
                    -- node[inside, scale=0.7] {$x_4$} cycle;
                \draw (-5,0) -- node[right,pos=0.25] {$x_0^{k-1}$} node[inside, scale=0.5, pos=0.75] {$x_3^{m-k-3}$} (-3,-2);
                \draw (-5,-5) -- node[right,pos=0.25] {$x_0^{k-1}$} node[inside, scale=0.5, pos=0.75] {$x_3^{m-k-3}$} (-3,-3);
                \draw (0,0) -- node[left,pos=0.25] {$x_0^{k-1}$} node[inside, scale=0.5, pos=0.75] {$x_3^{m-k-3}$} (-2,-2);
                \draw (0,-5) -- node[left,pos=0.25] {$x_0^{k-1}$} node[inside, scale=0.5, pos=0.75] {$x_3^{m-k-3}$} (-2,-3);

                \draw (4,-4) -- node[below] {$\pi_1$} (1,-4)
                    -- node[inside, scale=0.7, rectangle] {$x_{m-k}$} (1,-1)
                    -- node[above] {$\pi_1$} (4,-1)
                    -- node[inside, scale=0.7, rectangle] {$x_{m-k}$} cycle;
                \draw (3,-3) -- node[below] {$\pi_1$} (2,-3)
                    -- node[inside, scale=0.7] {$x_4$} (2,-2)
                    -- node[above] {$\pi_1$} (3,-2)
                    -- node[inside, scale=0.7] {$x_4$} cycle;
                \draw (0,0) -- node[right,pos=0.25] {$x_0^{k}$} node[inside, scale=0.5, pos=0.75] {$x_3^{m-k-4}$} (2,-2);
                \draw (0,-5) -- node[right,pos=0.25] {$x_0^{k}$} node[inside, scale=0.5, pos=0.75] {$x_3^{m-k-4}$} (2,-3);
                \draw (5,0) -- node[left,pos=0.25] {$x_0^{k}$} node[inside, scale=0.5, pos=0.75] {$x_3^{m-k-4}$} (3,-2);
                \draw (5,-5) -- node[left,pos=0.25] {$x_0^{k}$} node[inside,scale=0.5, pos=0.75] {$x_3^{m-k-4}$} (3,-3);
            \end{tikzpicture}
        \end{equation*}
        with $6\leq m\leq n$, $1\leq k\leq n-6$, and $k+3\leq m$.
        The highlighted regions (which actually do not contain any cells) can be collapsed. The checkered region
        can also be reduced. We obtain the van Kampen diagram
        \begin{equation*}
            \begin{tikzpicture}[scale=0.9]
                \fill[pattern=checkerboard, pattern color=gray!20] (0,4) -- (1,4) -- (1,-4) -- (0,-4) -- cycle;
                \fill[pattern=dots] (-4,0) -- (-3,2) -- (-2,2) -- (0,0) -- (1,0) -- (2,2) -- (3,2) -- (4,0) -- (3,-2) -- (2,-2) -- (1,0) -- (0,0) -- (-2,-2) -- (-3,-2) -- cycle;
                \draw (-5,5) -- node[above,pos=0.25] {$\pi_k$} node[above,pos=0.75] {$\pi_{k+1}$} (5,5)
                    -- node[right,pos=0.25] {$x_{m-1}$} node[right,pos=0.75] {$x_m$} (5,-5) 
                    -- node[below,pos=0.25] {$\pi_{k+1}$} node[below,pos=0.75] {$\pi_k$} (-5,-5)
                    -- node[left,pos=0.25] {$x_m$} node[left,pos=0.75] {$x_{m-1}$} cycle;
                \draw (0,5) -- node[inside, scale=0.5, pos=0.05] {$x_0^{k-1}$} node[inside, scale=0.5, pos=0.95] {$x_0^{k-1}$} (0,-5);

                \draw (-4,4) -- node[above] {$\pi_1$} (0,4)
                    -- node[inside, scale=0.7, rectangle] {$x_{m-k}$} (0,0)
                    -- (-4,0)
                    -- node[inside, rectangle] {$x_{m-k}$} cycle;
                \draw (-3,3) -- node[above] {$\pi_1$} (-2,3)
                    -- node[inside, scale=0.7] {$x_4$} (-2,2)
                    -- node[below=2pt, fill=white, inner sep=1pt] {$\pi_1$} (-3,2)
                    -- node[inside, scale=0.7] {$x_4$} cycle;
                \draw (-5,5) -- node[inside, scale=0.7,pos=0.25] {$x_0^{k-1}$} node[inside,pos=0.75, scale=0.5] {$x_3^{m-k-4}$} (-3,3);
                \draw (-4,0) -- node[inside, scale=0.5] {$x_3^{m-k-4}$} (-3,2);
                \draw (0,4) -- node[inside, scale=0.5] {$x_3^{m-k-4}$} (-2,3);
                \draw (0,0) -- node[inside, scale=0.5] {$x_3^{m-k-4}$} (-2,2);
                
                \draw (4,4) -- node[above] {$\pi_1$} (1,4)
                    -- node[inside, scale=0.7, rectangle] {$x_{m-k-1}$} (1,0)
                    -- (4,0)
                    -- node[inside, rectangle] {$x_{m-k-1}$} cycle;
                \draw (3,3) -- node[above] {$\pi_1$} (2,3)
                    -- node[inside, scale=0.7] {$x_4$} (2,2)
                    -- node[below=2pt, fill=white, inner sep=1pt] {$\pi_1$} (3,2)
                    -- node[inside] {$\scriptstyle x_4$} cycle;
                \draw (0,4) -- node[above] {$x_0$} (1,4) -- node[inside, scale=0.5] {$x_3^{m-k-5}$} (2,3);
                \draw (1,0) -- node[inside, scale=0.5] {$x_3^{m-k-5}$} (2,2);
                \draw (5,5) -- node[inside, scale=0.7,pos=0.25] {$x_0^{k}$} node[inside, scale=0.5,pos=0.75] {$x_3^{m-k-5}$} (3,3);
                \draw (4,0) -- node[inside, scale=0.5] {$x_3^{m-k-5}$} (3,2);

                \draw (-4,-4) -- node[below] {$\pi_1$} (0,-4)
                    -- node[inside, scale=0.7, rectangle] {$x_{m-k+1}$} (0,0)
                    -- (-4,0)
                    -- node[inside, scale=0.7, rectangle] {$x_{m-k+1}$} cycle;
                \draw (-3,-3) -- node[below] {$\pi_1$} (-2,-3)
                    -- node[inside, scale=0.7] {$x_4$} (-2,-2)
                    -- node[above=2pt, fill=white, inner sep=1pt] {$\pi_1$} (-3,-2)
                    -- node[inside, scale=0.7] {$x_4$} cycle;
                \draw (-4,0) -- node[inside, scale=0.5] {$x_3^{m-k-3}$} (-3,-2);
                \draw (-5,-5) -- node[inside, scale=0.7,pos=0.25] {$x_0^{k-1}$} node[inside, scale=0.5, pos=0.75] {$x_3^{m-k-3}$} (-3,-3);
                \draw (0,0) -- node[inside, scale=0.5] {$x_3^{m-k-3}$} (-2,-2);
                \draw (0,-4) -- node[inside, scale=0.5] {$x_3^{m-k-3}$} (-2,-3);

                \draw (4,-4) -- node[below] {$\pi_1$} (1,-4)
                    -- node[inside, scale=0.7, rectangle] {$x_{m-k}$} (1,0)
                    -- (4,0)
                    -- node[inside, rectangle] {$x_{m-k}$} cycle;
                \draw (3,-3) -- node[below] {$\pi_1$} (2,-3)
                    -- node[inside, scale=0.7] {$x_4$} (2,-2)
                    -- node[above=2pt, fill=white, inner sep=1pt] {$\pi_1$} (3,-2)
                    -- node[inside, scale=0.7] {$x_4$} cycle;
                \draw (1,0) -- node[inside, scale=0.5] {$x_3^{m-k-4}$} (2,-2);
                \draw (0,-4) -- node[below] {$x_0$} (1,-4) -- node[inside, scale=0.5] {$x_3^{m-k-4}$} (2,-3);
                \draw (4,0) -- node[inside, scale=0.5] {$x_3^{m-k-4}$} (3,-2);
                \draw (5,-5) -- node[inside, scale=0.7,pos=0.25] {$x_0^{k}$} node[inside, scale=0.5,pos=0.75] {$x_3^{m-k-4}$} (3,-3);
                \draw (-5,0) -- node[above,pos=0.05] {$x_0^{k-1}$} node[inside, rectangle,pos=0.25] {$\pi_1$} node[above,pos=0.55] {$x_0$} node[inside, rectangle,pos=0.75] {$\pi_1$} node[above,pos=0.95]{$x_0^{k}$} (5,0);
            \end{tikzpicture}
        \end{equation*}
        Now we observe that the top and bottom part of the dotted regions are almost mirror images.
        So, we can further reduce; in this case to 1 cell each. We finally get the diagram
        \begin{equation*}
            \begin{tikzpicture}[every circle node/.style={inner sep=-1pt,fill=white}]
                \fill[pattern=dots] (-2,0) rectangle (-1,-1);
                \fill[pattern=dots] (3,0) rectangle (2,-1);
                \draw (-4,4) -- node[above,pos=0.22] {$\pi_k$} node[above,pos=0.72] {$\pi_{k+1}$} (5,4)
                    -- (5,-5) node[right,pos=0.22] {$x_{m-1}$} node[right,pos=0.72] {$x_m$}
                    -- (-4,-5) node[below,pos=0.28] {$\pi_{k+1}$} node[below,pos=0.78] {$\pi_k$}
                    -- node[left,pos=0.28] {$x_m$} node[left,pos=0.78] {$x_{m-1}$} cycle;
                \draw (0,4) -- node[circle] {$\scriptstyle x_0^{k-1}$} (0,3)
                    -- node[inner sep=1pt,fill=white,pos=0.25] {$x_{m-k}$} node[inner sep=1pt,fill=white,pos=0.75] {$x_{m-k+1}$} (0,-4)
                    -- node[circle] {$\scriptstyle x_0^{k-1}$} (0,-5);
                \draw (-2,0) -- node[circle] {$\pi_1$} (-1,0);
                \draw (2,0) -- node[circle] {$\pi_1$} (3,0);
                \draw (0,3) -- node[above] {$x_0$} (1,3);
                \draw (0,-4) -- node[below] {$x_0$} (1,-4);
                    
                \draw (0,3) -- node[above] {$\pi_1$} (-3,3)
                    -- node[inner sep=1pt,fill=white,pos=0.22] {$x_{m-k}$} node[inner sep=1pt,fill=white,pos=0.72] {$x_{m-k+1}$} (-3,-4)
                    -- node[below] {$\pi_1$} (0,-4);
                \draw (1,0) -- node[inner sep=1pt,fill=white] {$x_{m-k-1}$} (1,3)
                    -- node[above] {$\pi_1$} (4,3)
                    -- node[inner sep=1pt,fill=white,pos=0.22] {$x_{m-k-1}$} node[inner sep=1pt,fill=white,pos=0.72] {$x_{m-k}$} (4,-4)
                    -- node[below] {$\pi_1$} (1,-4)
                    -- node[inner sep=1pt,fill=white] {$x_{m-k}$} cycle;

                \draw (-2,1) -- node[above] {$\pi_1$} (-1,1)
                    -- node[circle,pos=0.17] {$x_4$} node[circle] {$x_3$} node[circle,pos=0.83] {$x_4$} (-1,-2)
                    -- node[below] {$\pi_1$} (-2,-2)
                    -- node[circle,pos=0.17] {$x_4$} node[circle] {$x_3$} node[circle,pos=0.83] {$x_4$} cycle;
                \draw (-2,-1) -- node[circle] {$\pi_1$} (-1,-1);
                \draw (3,1) -- node[above] {$\pi_1$} (2,1)
                    -- node[circle,pos=0.17] {$x_4$} node[circle] {$x_3$} node[circle,pos=0.83] {$x_4$} (2,-2)
                    -- node[below] {$\pi_1$} (3,-2)
                    -- node[circle,pos=0.17] {$x_4$} node[circle] {$x_3$} node[circle,pos=0.83] {$x_4$} cycle;
                \draw (2,-1) -- node[circle] {$\pi_1$} (3,-1);

                \draw (-4,4) -- node[circle] {$x_0^{k-1}$} (-3,3) -- node[inside, scale=0.5] {$x_3^{m-k-4}$} (-2,1);
                \draw (5,4) -- node[circle] {$x_0^{k}$} (4,3) -- node[inside, scale=0.5] {$x_3^{m-k-5}$} (3,1);
                \draw (-4,-5) -- node[circle] {$x_0^{k-1}$} (-3,-4) -- node[inside, scale=0.5] {$x_3^{m-k-3}$} (-2,-2);
                \draw (5,-5) -- node[circle] {$x_0^{k}$} (4,-4) -- node[inside, scale=0.5] {$x_3^{m-k-4}$} (3,-2);

                \draw (0,3) -- node[inside, scale=0.5] {$x_3^{m-k-4}$} (-1,1);
                \draw (1,3) -- node[inside, scale=0.5] {$x_3^{m-k-5}$} (2,1);
                \draw (0,-4) -- node[inside, scale=0.5] {$x_3^{m-k-3}$} (-1,-2);
                \draw (1,-4) -- node[inside, scale=0.5] {$x_3^{m-k-4}$} (2,-2);
            \end{tikzpicture}
        \end{equation*}

        \begin{figure}
            \includegraphics[scale=0.8]{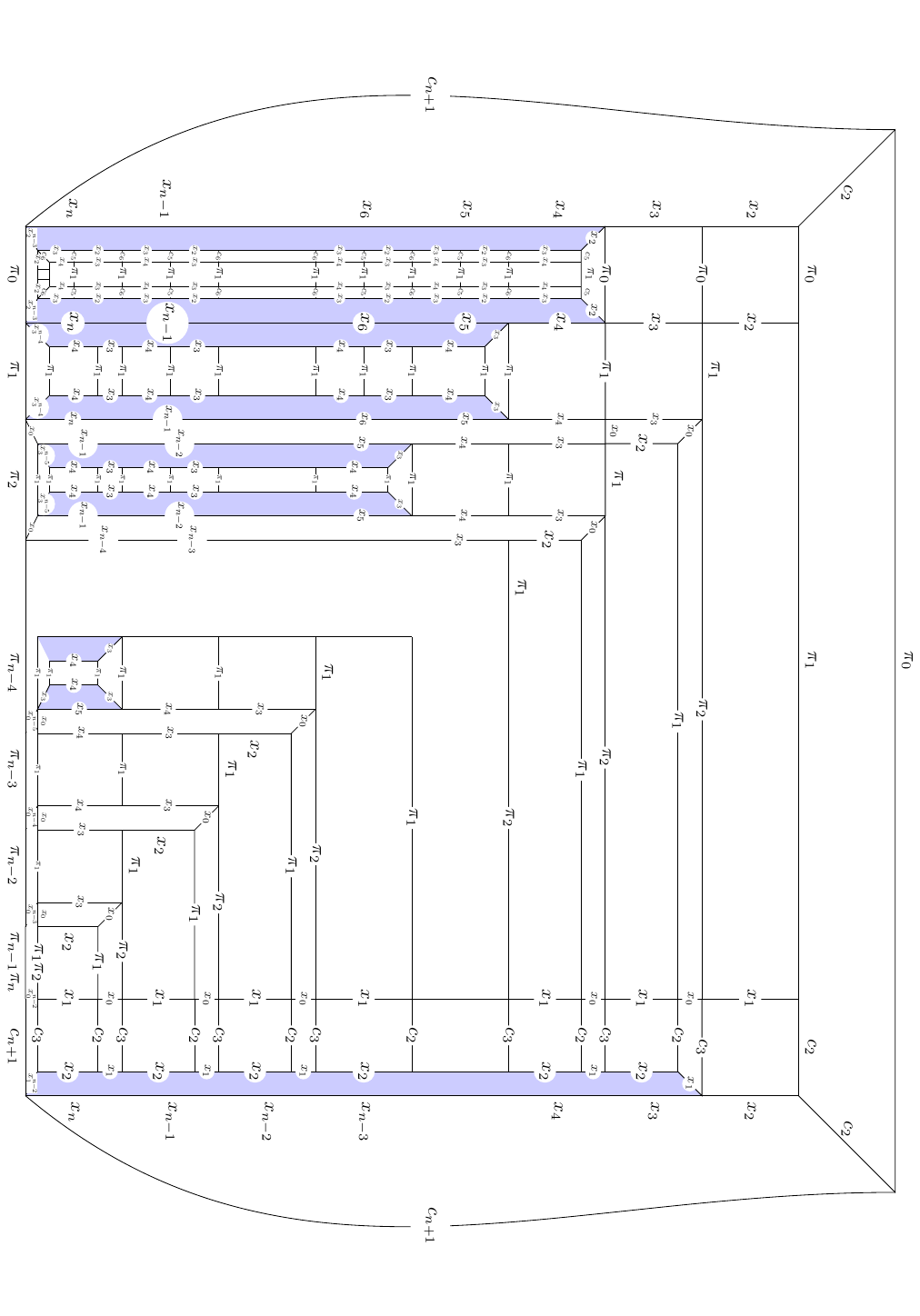}
            \caption{The total diagram with regions in $F$ highlighted.\label{fig:biggestdiagram}}
        \end{figure}

        Two very small stacks of 3 squares appear (the one containing the dotted square). Now, instead of considering four ``big squares''
        as in (\ref{eq:four_big_squares}), we work out what happens when the whole triangle-shaped stack is formed.
        The result can be seen on \zcref{fig:biggestdiagram}.
        The highlighted regions are the only parts that are not formally filled in and contain at least one cell. Their
        respective boundary equations are in $F$, and thus we will use Guba's result on $F$ to describe their combined area.
        each highlighted area in the triangular stack has boundary equation
        \begin{equation*}
            (x_3^{-1}x_4)^{k+1} = x_0^{-1}x_4(x_3^{-1}x_4)^kx_3^kx_0,\quad k\in\{0,\dots,n-5\}
        \end{equation*}
        Writing this in the generators of $F$ gives
        \begin{equation*}
            (x_1^{-1}x_0^{-1}x_1x_0)^{k+1}x_1^{k+1} = x_0^{-2}x_1x_0(x_1^{-1}x_0^{-1}x_1x_0)^kx_1^kx_0
        \end{equation*}
        which has length $10k+10$
        Thus, the combined area of the highlighted regions is
        \begin{equation*}
            2\sum_{k=0}^{n-5} O((10k+10)^2) = O(n^3)
        \end{equation*}
        The left most column of the triangular stack consists of diagrams with boundary
        equation
        \begin{equation*}
            \pi_0x_i = x_i\pi_0,\qquad i\in\{2,3,\dots,n\}
        \end{equation*}
        According to \zcref{lem:pi0xn_xnpi0} each of these costs less then $O(n^2)$.
        Therefore, the whole column can be realised with at most $O(n^3)$ cells.
    \end{enumerate}
    In conclusion, each of the four regions can be filled with at most $O(n^3)$ cells; this has also been worked
    out in \zcref{fig:biggestdiagram}.
\end{proof}

\begin{lemma}\label{lem:cn2pi0_pin-1dotscn}
    For any integer $n>1$
    \begin{equation*}
        \|c_n^2\pi_0 = \pi_{n-1}\cdots \pi_0 c_n\| = O(n^3)
    \end{equation*}
\end{lemma}
\begin{proof}
    This follows directly from \zcref{lem:cnpi0_pi0dotspin-1cn2}, repeated conjugacy and the fact that the $\pi_i's$ are all involutions.
\end{proof}

\begin{lemma}\label{lem:cn3pi0_pin-1cn3}
    For any integer $n\geq 1$
    \begin{equation*}
        \|c_n^3\pi_0 = \pi_{n-1}c_n^3\| = O(n^3)
    \end{equation*}
\end{lemma}
\begin{proof}
    If $n=1$, then we get the equality $c_1^3\pi_0 = \pi_0c_1^3$ which can be shown with two copies of the relation
    $c_1^3=1$. Now consider the cases $n\geq 2$ and construct the following van Kampen diagram.
    \begin{equation*}
        \begin{tikzpicture}
            \draw (-6,0) -- node[right, pos=0.25] {$c_n$} node[right, pos=0.75] {$c_n$} (-6,2) -- node[above] {$\pi_0$} (6,2) -- node[left] {$c_n$} (6,0);
            \draw (-6,0) -- node[right] {$c_n$} (-6,-2) -- node[below] {$\pi_{n-2}$} (-4,-2) -- node[inside] {$c_n$} (-4,0) -- node[inside] {$\pi_{n-1}$} cycle;
            \draw (-4,-2) -- node[below] {$\pi_{n-3}$} (-2,-2) -- node[inside] {$c_n$} (-2,0) -- node[inside] {$\pi_{n-2}$} (-4,0);
            \draw (-2,-2) -- node[below] {$\pi_{n-4}$} (0,-2) -- node[inside] {$c_n$} (0,0) -- node[inside] {$\pi_{n-3}$} (-2,0);
            \draw (2,-2) -- node[below] {$\pi_0$} (4,-2) -- node[inside] {$c_n$} (4,0) -- node[inside] {$\pi_1$} (2,0) -- node [inside] {$c_n$} cycle;
            \draw (4,-2) -- node[below] {$\pi_0\pi_1\cdots \pi_{n-1}$} (6,-2) -- node[left, pos=0.25] {$c_n$} node[left, pos=0.75] {$c_n$} (6,0) -- node[inside] {$\pi_0$} (4,0);
            \draw[dotted] (0,0) -- (2,0);
            \draw[dotted] (0,-2) -- (2,-2);
        \end{tikzpicture}
    \end{equation*}
    The bottom label reads $\pi_{n-2}\pi_{n-3}\cdots\pi_0\pi_0\cdots\pi_{n-2}\pi_{n-1}$. Since the $\pi$-letters are involutions, this word can be reduced to $\pi_{n-1}$.
    This leads to the required boundary equation.
    The top region is handled by \zcref{lem:cn2pi0_pin-1dotscn} in $O(n^3)$ cells.
    The square regions at the bottom can each be filled with $O(\max\{j+2,(n-j)-1\}^2)$ cells each (for $j\in\{1,\dots,n-1\}$) according to \zcref{lem:cipij_pij-1ci},
    except for the bottom right region, which can be filled with \zcref{lem:cnpi0_pi0dotspin-1cn2} in $O(n^3)$. Tallying all of these
    areas, one obtains a total of $O(n^3)$.
\end{proof}

With \zcref{lem:cnpi0_pi0dotspin-1cn2,lem:cn2pi0_pin-1dotscn,lem:cn3pi0_pin-1cn3},
we have the necessary building blocks to explore what happens in expressions of
the form $c_n^m\pi_k$. We split the discussion of these in multiple lemmas.

\begin{lemma}\label{lem:cnmipi_k_pik-mcnm}
    For any integers $0\leq m\leq k<n$ and $m<n+2$
    \begin{equation*}
        \|c_n^m\pi_k = \pi_{k-m}c_n^m\| = O(mn^2)\leq O(n^3)
    \end{equation*}
\end{lemma}
\begin{proof}
    For $m=0$, this is trivial. For $m\geq 1$, this is $m$ repetitions of \zcref{lem:cipij_pij-1ci}. That is:
    \begin{align*}
        \|\pi_{k-m}^{c_n^m} = \pi_k\|& \leq \|\pi_{k-m}^{c_n} = \pi_{k-m+1}\| + \|\pi_{k-m+1}^{c_n} = \pi_{k-m+2}\| + \cdots + \|\pi_{k-1}^{c_n} = \pi_k\|\\
            & \leq \sum_{j=0}^{m-1} O(\max\{n-(k-m+j)+2,k-m+j-1\}^2)\\
            & \leq \sum_{j=0}^{m-1} O(n^2) = O(mn^2)\leq O(n^3)
    \end{align*}
\end{proof}

Now, if instead of having $m\leq k$ as in \zcref{lem:cnmipi_k_pik-mcnm}, we have $m=k+1$, we get the following.
\begin{lemma}\label{lem:cnmpim-1_pi0dotspin-1cnm+1}
    For any integers $0<m<n+1$
    \begin{equation*}
        \|c_n^m\pi_{m-1} = \pi_0\cdots\pi_{n-1}c_n^{m+1}\| = O(n^3)
    \end{equation*}
\end{lemma}
\begin{proof}
    Suppose that $n=1$. This forces $m=1$ too, and we get the equation
    \begin{equation*}
        c_1\pi_0 = \pi_0c_1^2
    \end{equation*}
    Consider the following van Kampen diagram
    \begin{equation*}
        \begin{tikzpicture}
            \fill[blue!20] (216:3) -- (288:3) -- (288:2) -- (216:2) -- cycle;
            \fill[pattern=dots, pattern color=gray!50] (72:3) -- (144:3) -- (144:2) -- ($(72:2)!0.5!(144:2)$) -- cycle;
            \fill[pattern=dots, pattern color=gray!50] (288:3) -- (288:2) -- (0:2) -- (0:3) -- cycle;
            \draw (0:3) -- node[right] {$c_1$} (72:3)
                -- node[above] {$\pi_0$} (144:3)
                -- node[left] {$c_1$} (216:3)
                -- node[below] {$\pi_0$} (288:3)
                -- node[right] {$c_1$} cycle;
            \draw (144:2) -- node[left] {$c_2$} (216:2)
                -- node[below, pos=0.25] {$\pi_1$} node[below, pos=0.75] {$\pi_0$} (288:2)
                -- node[right] {$c_2$} (0:2);
            \draw (144:2) -- node[inside] {$c_2$} (144:3);
            \draw (216:3) -- node[inside] {$x_0$} (216:2);
            \draw (288:3) -- node[inside] {$x_1$} (288:2);
            \draw (0:3) -- node[inside] {$x_0$} (0:2);

            \draw (72:3) -- node[inside] {$c_2$} ($(72:2)!0.5!(144:2)$) -- node[above] {$\pi_1$} (144:2);
            \draw (72:3) -- node[inside] {$c_2$} ($(72:2)!0.5!(0:2)$) -- node[right] {$c_2$} (0:2);

            \draw (72:3) -- node[right] {$c_2$} (288:2);
            \draw (72:3) -- node[right] {$\pi_1$} ($(72:3)!0.5!(216:2)$) -- node[inside] {$c_2$} ($(216:2)!0.5!(288:2)$);
            \draw (288:2) to[out=115,in=35] node[above] {$\pi_0$} ($(288:2)!0.5!(216:2)$);
        \end{tikzpicture}
    \end{equation*}
    with boundary equation $c_1^2\pi_0 = \pi_0c_1$. We count $6$ cells (the dotted regions are definitions) in this diagram plus the
    highlighted region, which takes $8$ cells according to \zcref{lem:pinxn+1_xnpin+1pin}.
    To obtain the wanted boundary equation $c_1\pi_0 = \pi_0c_1^2$, we need to
    add two more $(\pi_0^2=1)$-cells, bringing the total of relations to $16$.

    Now suppose that $n\geq2$. Using \zcref{lem:cnmipi_k_pik-mcnm}, we shuffle $m-1$ copies of $c_n$
    over the $\pi$-letter at the cost of $O(mn^2)$ relations. Thus, we get that
    \begin{align*}
        \|c_n^m\pi_{m-1} = \pi_0\cdots\pi_{n-1}c_n^{m+1}\| &\leq O(mn^2) + \|c_n\pi_0c_n^{m-1} = \pi_0\cdots \pi_{n-1}c_n^2c_n^{m-1}\|\\
                &\leq O(mn^2) + \|c_n\pi_0 = \pi_0\cdots \pi_{n-1}c_n^2\| \\
                &\leq O(mn^2) + O(n^3)\\
                &\leq O(n^3)
    \end{align*}
\end{proof}

Now we consider the case $m=k+2$.
\begin{lemma}\label{mr-lemma-3}
    For any integers $2\leq m<n+2$,
    \begin{equation*}
        \|c_n^m\pi_{m-2} = \pi_{n-1}\cdots\pi_{0}c_n^{m-1}\| = O(n^3)
    \end{equation*}
\end{lemma}
\begin{proof}
    If $n=1$, than $m=2$ and we have
    \begin{equation*}
        c_1^2\pi_0 = \pi_0c_1
    \end{equation*}
    which requires $14$ relations as proven by the van Kampen diagram in \zcref{lem:cnmpim-1_pi0dotspin-1cnm+1}.
    For higher values of $n$, we again use \zcref{lem:cnmipi_k_pik-mcnm}
    to shuffle $m-2$ copies of $c_n$
    over the $\pi$-letter at the cost of $O(mn^2)$ relations. This results in
    \begin{align*}
        \|c_n^m\pi_{m-2} = \pi_{n-1}\cdots\pi_{0}c_n^{m}\| &\leq O(mn^2) + \|c_n^2\pi_0c_n^{m-2} = \pi_{n-1}\cdots\pi_0c_nc_n^{m-2}\|\\
            &\leq O(mn^2) + \|c_n^2\pi_0 = \pi_{n-1}\cdots\pi_0c_n\|\\
            &\leq O(mn^2) + O(n^3)\\
            &\leq O(n^3)
    \end{align*}
\end{proof}

\begin{lemma}\label{mr-lemma-4}
    For any integers $m,n,k$ satisfying $0\leq m<n+2$, $0\leq k<n$, $m>k+2$,
    \begin{equation*}
        \|c_n^m\pi_k = \pi_{k+n+2-m}c_n^m\| = O(n^3)
    \end{equation*}
\end{lemma}
\begin{proof}
    If $n=2$, this forces $k=0$, $m=3$ or,
    \begin{equation*}
        c_2^{3}\pi_0 = \pi_1c_2^{3}
    \end{equation*}
    which can be proven in $6$ relations as (the highlighted regions are definitions).
    \begin{equation*}
        \begin{tikzpicture}
            \fill[blue!20] (-4,1.5) -- (-2.5,0) -- (-1,1.5) -- (-1,-1.5) -- (1,-1.5) -- (1,1.5) -- (2.5,0) -- (4,1.5) -- cycle;
            \draw (-4,1.5) -- node[above] {$c_2$} (-1,1.5) -- node[above] {$\pi_0$} (1,1.5) -- node[above] {$c_2$} (4,1.5);
            \draw (-4,-1.5) -- node[below] {$c_2$} (-1,-1.5) -- node[below] {$\pi_1$} (1,-1.5) -- node[below] {$c_2$} (4,-1.5);
            \foreach \x in {-4,-1,1,4} {
                \draw (\x,1.5) -- node[inside] {$c_2$} (\x,-1.5);
            }
            \draw (1,-1.5) -- node[inside] {$c_1$} (2.5,0) -- node[inside] {$x_0$} (4,1.5);
            \draw (2.5,0) -- node[inside, pos=0.25] {$c_1$} node[inside, pos=0.75] {$x_1$} (1,1.5);
            \draw ($(2.5,0)!0.5!(1,1.5)$) -- node[inside] {$c_1$} (1,-1.5);
            \draw (-1,-1.5) -- node[inside] {$c_1$} (-2.5,0) -- node[inside] {$x_0$} (-4,1.5);
            \draw (-2.5,0) -- node[inside, pos=0.25] {$c_1$} node[inside, pos=0.75] {$x_1$} (-1,1.5);
            \draw ($(-2.5,0)!0.5!(-1,1.5)$) -- node[inside] {$c_1$} (-1,-1.5);
        \end{tikzpicture}
    \end{equation*}
    To prove the general case, we use a string of previous lemmas.
    \begin{align*}
        \|c_n^m\pi_k = \pi_{k+n+2-m}c_n^m\| &\underset{\text{\ref{lem:cnmipi_k_pik-mcnm}}}{\leq} O(n^3) + \|c_n^{m-k}\pi_0c_n^k = \pi_{k+n+2-m}c_n^m\|\\
            &\underset{\text{\ref{lem:cn3pi0_pin-1cn3}}}{\leq} 2O(n^3) + \|c_n^{m-k-3}\pi_{n-1}c_n^{k+3} = \pi_{k+n+2-m}c_n^m\| \\
            &\underset{\text{\ref{lem:cnmipi_k_pik-mcnm}}}{\leq} 3O(n^3) = O(n^3)
    \end{align*}
\end{proof}

Combining the four previous lemmas in one.
\begin{lemma}\label{lem:cnmpik}
    For any integers $0\leq m<n+2$, and $0\leq k<n$,
    \begin{equation*}
        c_n^m\pi_k = \begin{cases*}
            \pi_{k-m}c_n^m & $m\leq k$\\
            \pi_0\cdots\pi_{n-1}c_n^{m+1} & $m=k+1$\\
            \pi_{n-1}\cdots\pi_0c_n^{m-1} & $m=k+2$\\
            \pi_{k+n+2-m}c_n^m & $m>k+2$
        \end{cases*}
    \end{equation*}
    in at most $O(n^3)$ relations.
\end{lemma}
\begin{proof}
    The four cases are handled by
    \zcref{lem:cnmipi_k_pik-mcnm,lem:cnmpim-1_pi0dotspin-1cnm+1,mr-lemma-3,mr-lemma-4},
    respectively
\end{proof}

\section{\texorpdfstring{$\pi$-words}{pi-words}}\label{sec:pi_words}
We have made all the preperations for handling simple relations in $V$. One difficulty remains:
As shown above, ``switching'' certain letters can generate extra $\pi$-letters.
Thus, we need a way to control these. We observe that the letters $\{\pi_0,\pi_1,\pi_2,\dots\}$
generate a permutation group inside $V$. Leveraging this well-known structure will be key
for bounding the Dehn function.

\subsection{A rewriting system over the infinite symmetric group.}
The infinite symmetric group $\mathcal{S}_\infty$ is the group with presentation
\begin{equation*}
    \langle \pi_i,\; i \in \mathbb{N}\mid \pi_i^2 = (\pi_i\pi_{i+1})^3 = 1,\; \pi_i\pi_j = \pi_j\pi_i\; (|i-j|\geq 2)\rangle
\end{equation*}
Now consider the string rewriting system.
\begin{equation}\label{def:rewriting_system}
    \begin{aligned}
        \pi_i^2 &\to 1\quad (i\geq 0)\\
        \pi_j\pi_i &\to \pi_i\pi_j,\quad (j\geq i+2)\\
        (\pi_{i+k}\pi_{i+k-1}\cdots\pi_i)\pi_{i+k} &\to \pi_{i+k-1}(\pi_{i+k}\pi_{i+k-1}\cdots\pi_i),\quad (i\geq 0,\;k\geq 1).
    \end{aligned}
\end{equation}

The following lemma from \cite{GubaFTV} shows that this string rewriting system is
a good choice for $\mathcal{S}_\infty$.
\begin{lemma}\label{lem:string_rewriting_leads_to_irreducible}
    The string rewriting system in \ref{def:rewriting_system} is terminating and confluent.
\end{lemma}
\begin{proof}
    See \cite[Lemma 14]{GubaFTV}.
\end{proof}
We will now leverage this string rewriting to count the amount of relations needed to reduce a word
containing only $\pi$-letters. Let us introduce some terminology to aid our future arguments.
\begin{definition}
    A $\pi$-\textbf{word} is a word in the letters $\{\pi_0,\pi_1,\dots\}$. The \textbf{rank} of a $\pi$-word $w$ is
    the highest subscript appearing in $w$ and is denoted $\rho(w)$. The rank of the empty word
    is defined to be $-1$.
\end{definition}
\begin{definition}\label{def:irreducible_pi_word}
    A $\pi$-word $w$ is called \textbf{irreducible} if no rewriting rule in (\ref{def:rewriting_system}) can be applied. By \zcref{lem:string_rewriting_leads_to_irreducible}, every $\pi$-word can be reduced to a unique irreducible $\pi$-word which we denote $\overline{w}$.
\end{definition}

\begin{definition}
    A $\pi$-word with $\pi$-letters $\pi_{k_1}\pi_{k_2}\cdots \pi_{k_n}$
    is an \textbf{upstream} if the subscripts $k_1,k_2,\dots k_n$ are
    an increasing sequence of consecutive numbers. If the subscripts are a 
    decreasing sequence of consecutive numbers, then the $\pi$-word is
    a \textbf{downstream}. For example, $\pi_1\pi_2\pi_3$ is an upstream, but
    $\pi_5\pi_4\pi_2$ is not a downstream.
\end{definition}

\begin{lemma}\label{lem:shape_of_irreducible_pi_word}
    An irreducible $\pi$-word is a sequence of downstreams with strictly increasing ranks.
    That is, an irreducible $\pi$-word $w$  of rank $n$ is of the form
    \begin{equation*}
        (\pi_{\alpha_1}\pi_{\alpha_1-1}\cdots\pi_{\alpha_1-\beta_1})(\pi_{\alpha_2}\pi_{\alpha_2-1}\cdots\pi_{\alpha_2-\beta_2})\cdots(\pi_{\alpha_k}\pi_{\alpha_k-1}\cdots\pi_{\alpha_k-\beta_k})
    \end{equation*}
    with $\beta_i\leq \alpha_i,\;\forall i$ and $\alpha_1<\alpha_2<\cdots<\alpha_k=n$.
\end{lemma}
\begin{proof}
    We first remark that such a word is indeed irreducible. In a segment
    \begin{equation*}
        (\pi_{\alpha_m}\pi_{\alpha_m-1}\cdots\pi_{\alpha_m-\beta_m})
    \end{equation*}
    no rewriting rules of the form $\pi_i\pi_j\to\pi_j\pi_i$ 
    can be applied as adjacent $\pi$-letters have indices that differ only by $1$.
    As $\alpha_m<\alpha_{m+1}$, rewriting rules of the form
    \begin{equation}\label{eq:long_rewriting}
        (\pi_{i+k}\pi_{i+k-1}\cdots\pi_{i+1}\pi_i)\pi_{i+k}\to \pi_{i+k-1}(\pi_{i+k}\pi_{i+k-1}\cdots\pi_{i+1}\pi_i)
    \end{equation}
    cannot be applied either. Lastly, it is clear that $\pi_i^2\to 1$ cannot occur.

    Now, we show that any irreducible $\pi$-word $w$ is of this form. Let $\pi_\gamma$
    be a letter in $w$. If this is not the last letter of $w$, call the 
    subscript of the next letter $\delta$. If $\delta\leq\gamma-2$, a rewriting rule can
    be applied, if $\delta=\gamma$ too, so $\delta\in\{\gamma-1\}\cup\{\gamma+1,\gamma+2,\dots,n\}$.
    Now let $(\pi_{\alpha_m}\pi_{\alpha_{m-1}}\cdots\pi_{\alpha_m-\beta_m}\pi_\gamma)$
    be a subword of $w$. We already know that
    \begin{equation*}
        \gamma\in\{\alpha_m-\beta_m-1\}\cup\{\alpha_m-\beta_m+1,\alpha_m-\beta_m+2,\cdots,n\}
    \end{equation*}
    But if $\alpha_m-\beta_m+1\leq \gamma\leq \alpha_m$, a rewriting of the form (\ref{eq:long_rewriting})
    can be performed, which causes $\gamma$ to be restricted to
    \begin{equation*}
        \gamma\in\{\alpha_m-\beta_m-1\}\cup\{\alpha_m+1,\alpha_m+2,\cdots,n\}
    \end{equation*}
\end{proof}

\begin{corollary}\label{cor:pi_length_of_irr_pi_word}
    For any irreducible $\pi$-word $w$ of rank $n$,
    \begin{equation*}
        |w|_\infty\leq \frac{(n+1)(n+2)}{2} = O(n^2)
    \end{equation*}
    where we recall the definition of $|w|_\infty$ from \zcref{def:length_inf_gen_set}.
\end{corollary}
\begin{proof}
    The longest (with respect to the $|\cdot|_\infty$ length) irreducible $\pi$-word of rank $n$ is
    \begin{equation*}
        \pi_0(\pi_1\pi_0)(\pi_2\pi_1\pi_0)\cdots (\pi_n\pi_{n-1}\cdots \pi_0)
    \end{equation*}
    It has length
    \begin{equation*}
        1 + 2 + 3 + \cdots + (n+1) = \frac{(n+1)(n+2)}{2}
    \end{equation*}
\end{proof}

\subsection{\texorpdfstring{Naively handling $\pi$-words}{Naively handling pi-words}}
Let us take a first jab at making any $\pi$-word irreducible.
We will see in \zcref{lem:irreducible_pi_word} that because
``switching'' two $\pi$-letters costs $O(n^2)$ and because an irreducible $\pi$-word of rank
$n$ can have up to $O(n^2)$ $\pi$-letters,
the number of relations spikes drastically.
But first, we prove a useful result.

\begin{lemma}\label{lem:pi_through_irr_pi_word}
    For any irreducible $\pi$-word $\overline{v}$ of rank at most $n$, and for any integer $j\leq n$,
    \begin{equation*}
        \|\overline{v}\pi_j = \overline{v\pi_j}\| = |\overline{v}|_{\infty}O(n^2) = O(n^4)
    \end{equation*}
\end{lemma}
\begin{proof}
    To make $\overline{v}\pi_j$ irreducible, one needs to shuffle $\pi_j$ into the correct spot in the
    word via the rules of our rewriting system. Furthermore, we need to check that
    doing so does not `disrupt' other letter in $\overline{v}$, causing the need for more
    applications of rewriting rules.
    The word $\overline{v}$ is irreducible of rank $m\leq n$, so according to
    \zcref{lem:shape_of_irreducible_pi_word}, it ends with a downstream of rank $m$.
    Therefore we can write:
    \begin{equation*}
        \overline{v}\pi_j = \overline{v'}(\pi_m\pi_{m-1}\cdots\pi_{m-\alpha})\pi_j,\qquad \alpha\leq m,\;\rho(\overline{v'})<n
    \end{equation*}
    We distinguish multiple cases:
    \begin{enumerate}
        \item[(A)] If $j<m-\alpha-1$, we can apply \zcref{lem:piipij_pijpii} $(\alpha+1)$ times to obtain the following.
        \begin{equation*}
            \overline{v'}(\pi_m\pi_{m-1}\cdots\pi_{m-\alpha})\pi_j = \overline{v'}\pi_j(\pi_m\pi_{m-1}\cdots\pi_{m-\alpha})
        \end{equation*}
        \item[(B)] If $j=m-\alpha-1$, then $\overline{v}\pi_j$ is already
        irreducible.
        \item[(C)] If $j=m-\alpha$, we just need to cancel the last two letters.
        \begin{equation*}
            \overline{v'}(\pi_m\pi_{m-1}\cdots\pi_{m-\alpha})\pi_{m-\alpha} = \overline{v'}(\pi_m\pi_{m-1}\cdots\pi_{m-\alpha-1})
        \end{equation*}
        \item[(D)] If $m-\alpha+1\leq j\leq m$, we need an extra step:
        \begin{align*}
           \overline{v'}(\pi_m\pi_{m-1}\cdots \pi_{m-\alpha})\pi_j
                &\underset{\text{(\ref{lem:piipij_pijpii})}}{=} \overline{v'}(\pi_m\cdots\pi_{j+1})\pi_j\pi_{j-1}\pi_j(\pi_{j-2}\cdots\pi_{m-\alpha})\\
                &= \overline{v'}(\pi_m\cdots\pi_{j+1})\pi_{j-1}\pi_j\pi_{j-1}(\pi_{j-2}\cdots\pi_{m-\alpha})\\
                &\underset{\text{(\ref{lem:piipij_pijpii})}}{=} \overline{v'}\pi_{j-1}(\pi_m\pi_{m-1}\cdots\pi_{m-\alpha})
        \end{align*}
    \end{enumerate}
    Neither of these cases creates something else than a downstream, and the rank of the
    downstreams is conserved (except in case (C) with $\alpha=0$).
    This implies that the only rewriting rules needed to make $\overline{v}\pi_j$
    irreducible, are the ones that shuffle $\pi_j$ into place (or cancel it). Therefore,
    \begin{equation*}
        \|\overline{v}\pi_j = \overline{v\pi_j}\| \leq |\overline{v}|_\infty O(n^2) = O(n^4)
    \end{equation*}
    according to the bounds in \zcref{lem:piipij_pijpii} and \zcref{cor:pi_length_of_irr_pi_word}.
\end{proof}

\begin{lemma}\label{lem:irreducible_pi_word}
    For $w$ a $\pi$-word with $\rho(w)\leq n$,
    \begin{equation*}
        \|w = \overline{w}\| = |w|_\infty^2O(n^2)
    \end{equation*}
\end{lemma}
\begin{proof}
    We work by induction on $|w|_\pi$. If the length is $1$. The theorem is clear.
    Let $\pi_j$ be the last letter of $w$, such that $w=v\pi_j$. We get
    \begin{align*}
        \|w = \overline{w}\| &\leq \|v\pi_j = \overline{v}\pi_j\| + \|\overline{v}\pi_j=\overline{v\pi_j}\|\\
                            &= \|v=\overline{v}\| + \|\overline{v}\pi_j=\overline{v\pi_j}\|\\
                            &= |v|_\infty^2O(n^2) + |\overline{v}|_\infty O(n^2)\\
                            &\leq (|v|_\infty+1)|v|_\infty O(n^2)\\
                            &\leq |w|_\infty^2O(n^2)
    \end{align*}
    Where we have used that $|\overline{v}|_\infty \leq |v|_\infty = |w|_\infty-1$.
\end{proof}

\subsection{\texorpdfstring{Streams of $\pi$}{Streams of pi}}

In general, the bound found in \zcref{lem:irreducible_pi_word}
is not very good. Often this will lead to the order $O(n^6)$.
To keep the promised sextic bound on the Dehn function of $V$, we will need to
avoid repeating the action of making random $\pi$-words irreducible.
There are two main ways of doing this. The first one is simply not making $\pi$-words irreducible, and only worrying
about the irreducibility at the last step. The biggest downside of course is that
we lose control of the ``shape'' of the $\pi$-word---and more importantly, its maximum length.
Therefore, we will adopt a second way of avoiding \zcref{lem:irreducible_pi_word}.
We will meticulously keep track of what can happen with the $\pi$-letters in irreducible
$\pi$-words when they interact with other generators. We will show that in most
circumstances, the irreducibility is only ``slightly disturbed'' and it does not require
the full $O(n^6)$ relations to get it back to irreducible.

\begin{lemma}[upstream]\label{lem:upstream}
    For integers $n,k,m\geq 0$, the expression
    \begin{equation*}
        (\pi_n\pi_{n+1}\cdots\pi_{n+k})x_m
    \end{equation*}
    can be rewritten to
    \begin{numcases}{}
        x_m(\pi_{n+1}\pi_{n+2}\cdots \pi_{n+k+1}) & $m<n$ \tag{A}\\
        x_{m+1}(\pi_n\pi_{n+1}\cdots \pi_{n+k+1}) & $n\leq m\leq n+k$ \tag{C}\\
        x_n(\pi_{n+1}\pi_n)(\pi_{n+2}\pi_{n+1})\cdots (\pi_{n+k+1}\pi_{n+k}) & $m=n+k+1$ \tag{D}\\
        x_m(\pi_n\pi_{n+1}\cdots \pi_{n+k}) & $n+k+1<m$ \tag{B}
    \end{numcases}
    in at most
    \begin{numcases}{}
        kO((n+k-m)^2) \tag{A}\\
        (n+k-m)O((n+k-m)^2) + (m-n+1)O((m-n+3)^2) \tag{C}\\
        (k+1)O(1)\tag{D}\\
        k O((m-n+1)^2) \tag{B}
    \end{numcases}
    relations respectively.
\end{lemma}
\begin{proof}
    Consider the expression $\pi_ix_j$. As we know, depending on the indices $i,j$,
    we get multiple equalities. (see \zcref{lem:pijxi_pij+1,lem:xjpii_piixj,lem:pinxn+1_xnpin+1pin,lem:pinxn_xn+1pinpin+1})
    \begin{align*}
        \pi_ix_j =\left\{
            \begin{array}{lll}
                x_j\pi_{i+1}    & \text{ in } O((i-j)^2)    & 0\leq j<i\\
                x_{i+1}\pi_i\pi_{i+1}& \text{ in } O(1) & j=i\\
                x_i\pi_{i+1}\pi_i& \text{ in } O(1)             & j=i+1\\
                x_j\pi_i        &\text{ in } O((j-i+1)^2)       & i+2\leq j 
            \end{array}\right\}
    \end{align*}
    So consider the following cases
    \begin{enumerate}
        \item[(A)] If $m<n$, the subscript of the $x$-letter is smaller than the subscripts
        of every $\pi$-letter. So, no new $\pi$-letters are created and $x_m$ slides
        through the $\pi$-word.
        \begin{equation*}
            (\pi_n\pi_{n+1}\cdots\pi_{n+k})x_m = x_m(\pi_{n+1}\pi_{n+2}\cdots\pi_{n+k+1})
        \end{equation*}
        This costs at most
        \begin{equation*}
            \sum_{i=0}^{k}O((n+k-m-i)^2) \leq kO((n+k-m)^2)
        \end{equation*}
        relations.
        \item[(B)] If $n+k+2\leq m$, then $m$ is bigger than every subscript in the
        $\pi$-word. Again, no new $\pi$-letters are generated and the following holds
        \begin{equation*}
            (\pi_n\pi_{n+1}\cdots\pi_{n+k})x_m = x_m(\pi_n\pi_{n+1}\cdots\pi_{n+k})
        \end{equation*}
        in at most
        \begin{equation*}
            \sum_{i=0}^k O((m-n-k+1+i)^2) \leq k O((m-n+1)^2)
        \end{equation*}
        relations.
        \item[(C)] If $n\leq m\leq n+k$, first, $x_m$ travels through the $\pi$-word as in case (A),
        until it no longer can. That is, until we are in the situation
        \begin{equation*}
            (\pi_n\pi_{n+1}\cdots\pi_{n+k})x_m = (\pi_n\pi_{n+1}\cdots \pi_{m-1})\pi_mx_m(\pi_{m+2}\pi_{m+3}\cdots\pi_{n+k+1})
        \end{equation*}
        Then, a new $\pi$-letter is created via $\pi_mx_m = x_{m+1}\pi_m\pi_{m+1}$,
        after which the rest is reduced to an instance of case (B). The result is
        \begin{equation*}
            (\pi_n\pi_{n+1}\cdots\pi_{n+k})x_m = x_{m+1}(\pi_n\pi_{n+1}\cdots \pi_{n+k+1})
        \end{equation*}
        Counting the relations, we get at most
        \begin{equation*}
            \sum_{i=0}^{n+k-m-1} O((n+k-m-i)^2) + O(1) + \sum_{i=0}^{m-n}O((3+i)^2)
        \end{equation*}
        \begin{equation*}
            \leq (n+k-m)O((n+k-m)^2) + (m-n+1)O((m-n+3)^2)
        \end{equation*}
        relations.
        \item[(D)] If $m=n+k+1$, things are more complex.
        Indeed, consider the following series of equalities.
        \begin{align*}
            (\pi_n\pi_{n+1}\cdots\pi_{n+k})x_{n+k+1} &= (\pi_n\cdots \pi_{n+k-1})x_{n+k}(\pi_{n+k+1}\pi_{n+k})\\
                &= (\pi_n\cdots \pi_{n+k-2})x_{n+k-1}(\pi_{n+k}\pi_{n+k-1})(\pi_{n+k+1}\pi_{n+k})\\
                &\vdots\\
                &= x_n(\pi_{n+1}\pi_n)(\pi_{n+2}\pi_{n+1})\cdots(\pi_{n+k+1}\pi_{n+k})
        \end{align*}
        A lot of extra $\pi$-letters are created. The cost of these equalities is
        quite low though, we count at most $(k+1)O(1)$ of them.
    \end{enumerate}
\end{proof}

\begin{lemma}[downstream]\label{lem:downstream}
    For integers $n,k,m\geq 0$, the expression
    \begin{equation*}
        (\pi_{n+k}\cdots\pi_{n+1}\pi_n)x_m
    \end{equation*}
    can be rewritten to
    \begin{numcases}{}
            x_m(\pi_{n+k+1}\cdots\pi_{n+2}\pi_{n+1}) & $m<n$ \tag{A}\\
            x_{n+k+1}(\pi_{n+k}\pi_{n+k+1})\cdots(\pi_{n+1}\pi_{n+2})(\pi_n\pi_{n+1}) & $m=n$\tag{D}\\
            x_{m-1}(\pi_{n+k+1}\cdots \pi_{n+1}\pi_n) & $n< m\leq n+k+1$\tag{C}\\
            x_m(\pi_{n+k}\cdots\pi_{n+1}\pi_n) & $n+k+1<m$\tag{B}
    \end{numcases}
    in at most
    \begin{numcases}{}
        (k+1)O((n+k-m)^2)\tag{A}\\
        (k+1)O(1)\tag{D}\\
        (m-n-1)O((m-n)^2) + (n+k-m+1)O((n+k-m+1)^2)\tag{C}\\
        (k+1)O((m-n+1)^2)\tag{B}
    \end{numcases}
    relations respectively.
\end{lemma}
\begin{proof}
    Consider the expression $\pi_ix_j$. As we know, depending on the indices $i,j$,
    we get multiple equalities. (see \zcref{lem:pijxi_pij+1,lem:xjpii_piixj,lem:pinxn+1_xnpin+1pin,lem:pinxn_xn+1pinpin+1})
    \begin{align*}
        \pi_ix_j =\left\{
            \begin{array}{lll}
                x_j\pi_{i+1}    & \text{ in } O((i-j)^2)    & 0\leq j<i\\
                x_{i+1}\pi_i\pi_{i+1}& \text{ in } O(1) & j=i\\
                x_i\pi_{i+1}\pi_i& \text{ in } O(1)             & j=i+1\\
                x_j\pi_i        &\text{ in } O((j-i+1)^2)       & i+2\leq j 
            \end{array}\right\}
    \end{align*}
    So consider the following cases
    \begin{enumerate}
        \item[(A)] If $m<n$, the subscript of the $x$-letter is smaller than the subscripts
        of every $\pi$-letter. So, no new $\pi$-letters are created and $x_m$ slides
        through the $\pi$-word.
        \begin{equation*}
            (\pi_{n+k}\cdots\pi_{n+1}\pi_n)x_m = x_m(\pi_{n+k+1}\cdots\pi_{n+2}\pi_{n+1})
        \end{equation*}
        This costs at most
        \begin{equation*}
            \sum_{i=0}^k O((n+i-m)^2)\leq (k+1)O((n+k-m)^2)
        \end{equation*}
        relations.
        \item[(B)] If $n+k+2\leq m$, then $m$ is bigger than every subscript in the
        $\pi$-word. Again, no new $\pi$-letters are generated and the following holds
        \begin{equation*}
            (\pi_{n+k}\cdots\pi_{n+1}\pi_n)x_m = x_m(\pi_{n+k}\cdots\pi_{n+1}\pi_n)
        \end{equation*}
        in
        \begin{equation*}
            \sum_{i=0}^{k}O((m-(n+i)+1)^2)\leq (k+1)O((m-n+1)^2)
        \end{equation*}
        \item[(C)] If $n< m\leq n+k+1$, then $x_m$ first travels through the $\pi$-word as in the
        case (B), until it no longer can. That is, until we are in the situation
        \begin{equation*}
            (\pi_{n+k}\cdots \pi_m)\pi_{m-1}x_m(\pi_{m-2}\cdots \pi_n)
        \end{equation*}
        Then, a new $\pi$-letter is created via $\pi_{m-1}x_m = x_{m-1}\pi_m\pi_{m-1}$,
        after which the rest is reduced to an instance of case (A). The result is
        \begin{equation*}
            (\pi_{n+k}\cdots\pi_{n+1}\pi_n)x_m = x_{m-1}(\pi_{n+k+1}\cdots \pi_{n+1}\pi_n)
        \end{equation*}
        In total this costs at most
        \begin{equation*}
            \sum_{i=0}^{m-n-2}O((m-n-i)^2) + O(1) + \sum_{i=0}^{n+k-m}O((i+1)^2)
        \end{equation*}
        \begin{equation*}
            \leq (m-n-1)O((m-n)^2) + (n+k-m+1)O((n+k-m+1)^2)
        \end{equation*}
        relations.
        \item[(D)] If $m=n$, there is a bit more work. Indeed, consider the following
        series of equalities.
        \begin{align*}
            (\pi_{n+k}\cdots\pi_{n+1}\pi_n)x_n &= (\pi_{n+k}\cdots \pi_{n+1})x_{n+1}(\pi_n\pi_{n+1})\\
                &= (\pi_{n+k}\cdots \pi_{n+2})x_{n+2}(\pi_{n+1}\pi_{n+2})(\pi_n\pi_{n+1})\\
                &\vdots\\
                &= x_{n+k+1}(\pi_{n+k}\pi_{n+k+1})\cdots(\pi_{n+1}\pi_{n+2})(\pi_n\pi_{n+1})
        \end{align*}
        which cost $(k+1)O(1)$ relations.
    \end{enumerate}
\end{proof}

The previous two lemmas show that in almost all cases, at most one extra $\pi$-letter is
generated when an $x$-letter ``passes through'' an upstream or downstream.
The only exception is when a downstream (respectively, upstream) of length $k$ is transformed into
$k$ upstreams (respectively, downstreams) of length $2$. We bundle them together in a simplified lemma for easier reference later.

\begin{corollary}[streams]\label{cor:x_through_stream}
    Let $\Pi$ be a stream of rank at most $n$. Then for any $m\leq n$,
    \begin{equation*}
        \|\Pi x_m = x_M\Pi'\|\leq O(n^3)
    \end{equation*}
    with $M\leq n+1$ and $\Pi'$ a $\pi$-word of rank at most $n+1$.
\end{corollary}
\begin{proof}
    This follows directly from \zcref{lem:upstream,lem:downstream}. 
\end{proof}

In the special case where a downstream (respectively, upstream) of length $k$ is transformed into
$k$ upstreams (respectively, downstreams) of length $2$, we can still convert this back to downstreams (respectively, upstreams).

\begin{lemma}\label{lem:unify_streams}
    For any integers $0\leq n,\;0<k$,
    \begin{equation*}
        (\pi_{n+k}\pi_{n+k+1})\cdots (\pi_{n+1}\pi_{n+2})(\pi_n\pi_{n+1}) = (\pi_{n+k}\cdots\pi_{n+1}\pi_n)(\pi_{n+k+1}\cdots\pi_{n+2}\pi_{n+1})
    \end{equation*}
    Moreover, this costs $O((k+2)^4)$.
\end{lemma}
\begin{proof}
    Consider the van Kampen diagram
    \begin{equation*}
        \begin{tikzpicture}
            \begin{scope}[every node/.style={inside}]
                \draw (0,0) -- node {$\pi_{n+k}$} (2,0)
                            -- node {$\pi_{n+k+1}$} (2,2)
                            -- node {$\pi_{n+k-1}$} (4,2)
                            -- node {$\pi_{n+k}$} (4,4)
                            -- node {$\pi_{n+k-2}$} (6,4);
                \draw (2,0) -- node {$\pi_{n+k-1}$} (4,0)
                            -- node {$\pi_{n+k-2}$} (6,0);
                \draw (10,0) -- node {$\pi_{n+1}$} (12,0)
                                -- node {$\pi_{n}$} (14,0);
                \draw (14,0) -- node {$\pi_{n+k+1}$} (14,2)
                                -- node {$\pi_{n+k}$} (14,4);
                \draw (14,8) -- node {$\pi_{n+3}$} (14,10)
                                -- node {$\pi_{n+2}$} (14,12)
                                -- node {$\pi_{n+1}$} (14,14);
                \draw (10,8) -- node {$\pi_{n+3}$} (10,10)
                                -- node {$\pi_{n+1}$} (12,10)
                                -- node {$\pi_{n+2}$} (12,12)
                                -- node {$\pi_n$} (14,12);
                \draw (4,0) -- node {$\pi_{n+k+1}$} (4,2) -- node {$\pi_{n+k-2}$} (6,2);
                \draw (6,0) -- node {$\pi_{n-k+1}$} (6,2) -- node {$\pi_{n+k}$} (6,4);
                \draw (10,0) -- node {$\pi_{n+k+1}$} (10,2) -- node {$\pi_{n+k}$} (10,4) -- node {$\pi_{n+1}$} (12,4) -- node {$\pi_n$} (14,4);
                \draw (12,0) -- node {$\pi_{n+k+1}$} (12,2) -- node {$\pi_{n+k}$} (12,4);
                \draw (10,2) -- node {$\pi_{n+1}$} (12,2) -- node {$\pi_n$} (14,2);
                \draw (10,8) -- node {$\pi_{n+1}$} (12,8) -- node {$\pi_{n}$} (14,8);
                \draw (12,8) -- node {$\pi_{n+3}$} (12,10) -- node {$\pi_{n}$} (14,10);
            \end{scope}
            \begin{scope}[every path/.style={dotted}]
                \draw (6,4) -- (10,8);
                \draw (6,4) -- (10,4);
                \draw (6,0) -- (10,0);
                \draw (14,4) -- (14,8);
                \draw (10,4) -- (10,8);
            \end{scope}
        \end{tikzpicture}
    \end{equation*}
    where each square region can be realised via \zcref{lem:piipij_pijpii}.
    Via the bound described in the same lemma, and by counting the diagonals,
    we bound the whole area by
    \begin{equation*}
        \sum_{j=1}^k jO((k+3-j)^2)\leq O((k+2)^2)\sum_{j=1}^kj \leq O((k+2)^4)
    \end{equation*}
\end{proof}

We now have a better grasp on the impact of passing $x$-letters through
irreducible $\pi$-words. Our next step is to look at what happens with $c$-letters.

\begin{lemma}\label{lem:ck_through_downstream}
    For any integers $0<k\leq n+2$, $0\leq\beta\leq\alpha\leq n$
    \begin{equation*}
        c_{n+1}^k(\pi_{\alpha}\pi_{\alpha-1}\cdots\pi_{\alpha-\beta})
    \end{equation*}
    \begin{equation*}
        =
        \begin{cases*}
            (\pi_{\alpha+n+3-k}\cdots\pi_{\alpha+n+3-k-\beta})c_{n+1}^k & $\alpha\leq k-3$\\
            (\pi_n\cdots\pi_0)^{\beta}c_{n+1}^{k-\beta} & $k=\alpha+2$\\
            (\pi_{\alpha-k+1}\pi_{\alpha-k+2}\cdots\pi_{\alpha-\beta+n-k+1})c_{n+1}^{k+1} & $\alpha-\beta<k\leq\alpha+1$\\
            (\pi_{\alpha-k}\pi_{\alpha-k-1}\cdots\pi_{\alpha-\beta-k})c_{n+1}^k & $k\leq \alpha-\beta$\\
        \end{cases*}
    \end{equation*}
    in at most $O(n^4)$ relations.
\end{lemma}
\begin{proof}
    Remember the following equalities from \zcref{lem:cnmpik}.
    \begin{equation*}
        c_{n+1}^k\pi_j = \begin{cases*}
            \pi_{j-k}c_{n+1}^k & $k\leq j$\\
            (\pi_0\cdots\pi_n)c_{n+1}^{k+1} & $j=k-1$\\
            (\pi_n\cdots\pi_0)c_{n+1}^{k-1} & $j=k-2$\\
            \pi_{j+n+3-k}c_{n+1}^k & $j\leq k-3$
        \end{cases*}
    \end{equation*}
    each costing at most $O(n^3)$ relations.
    We consider the different cases.
    \begin{enumerate}
        \item[(i)] If $k\leq \alpha-\beta$, we simply get
        \begin{equation*}
            c_{n+1}^k(\pi_{\alpha}\pi_{\alpha-1}\cdots\pi_{\alpha-\beta}) =
            (\pi_{\alpha-k}\pi_{\alpha-k-1}\cdots\pi_{\alpha-\beta-k})c_{n+1}^k
        \end{equation*}
        in at most $\beta O(n^3)\leq O(n^4)$ relations.
        \item[(ii)] If $\alpha\leq k-3$, similarly, we find
        \begin{equation*}
            c_{n+1}^k(\pi_{\alpha}\pi_{\alpha-1}\cdots\pi_{\alpha-\beta}) =
            (\pi_{\alpha+n+3-k}\pi_{\alpha+n+3-k-1}\cdots\pi_{\alpha-\beta+n+3-k})c_{n+1}^k    
        \end{equation*}
        in at most $O(n^4)$ relations.
        \item[(iii)] if $\alpha-\beta<k\leq\alpha+1$, we get the following.
        \begin{align*}
            c_{n+1}^k(\pi_{\alpha}\pi_{\alpha-1}\cdots\pi_{\alpha-\beta}) &=
                (\pi_{\alpha-k}\cdots\pi_0)c_{n+1}^k(\pi_{k-1}\cdots\pi_{\alpha-\beta})\\
                &= (\pi_{\alpha-k}\cdots\pi_0)(\pi_0\cdots\pi_n)c_{n+1}^{k+1}(\pi_{k-2}\cdots\pi_{\alpha-\beta})\\
                &= (\pi_{\alpha-k+1}\cdots\pi_n)c_{n+1}^{k+1}(\pi_{k-2}\cdots\pi_{\alpha-\beta})\\
                &= (\pi_{\alpha-k+1}\cdots\pi_n)(\pi_n\pi_{n-1}\cdots\pi_{\alpha-\beta+n-k+2})c_{n+1}^{k+1}\\
                &= (\pi_{\alpha-k+1}\pi_{\alpha-k+2}\cdots\pi_{\alpha-\beta+n-k+1})c_{n+1}^{k+1}
        \end{align*}
        in at most $O(n^4)$ relations.
        \item[(iv)] If $k=\alpha+2$, the situation is more complicated. We get
        \begin{align*}
            c_{n+1}^k(\pi_{k-2}\pi_{k-3}\cdots\pi_{k-2-\beta})
                &= (\pi_n\cdots\pi_0)c_{n+1}^{k-1}(\pi_{k-3}\pi_{k-4}\cdots\pi_{k-2-\beta})\\
                &= \vdots\\
                &= (\pi_n\cdots\pi_0)^{\beta}c_{n+1}^{k-\beta}
        \end{align*}
        in at most $O(n^4)$ relations.
    \end{enumerate}
\end{proof}

\begin{lemma}\label{lem:c_through_pi_word}
    Let $\overline{\Pi}$ be an irreducible $\pi$-word with $\rho(\overline{\Pi})<n$, and $k\leq n+2$, then
    \begin{equation*}
        \|c_{n+1}^k\overline{\Pi} = \overline{\Pi'}c_{n+1}^m\| = O(n^6)
    \end{equation*}
    for some irreducible $\pi$-word $\overline{\Pi'}$ of rank $\rho(\overline{\Pi'})\leq n$ and with $m\leq n+2$.
\end{lemma}
\begin{proof}
    Since $\overline{\Pi}$ is irreducible, it is a sequence of downstreams of strictly
    increasing rank as described in \zcref{lem:shape_of_irreducible_pi_word}.
    We pass all the $c$-letters through each downstream as in \zcref{lem:ck_through_downstream}.
    Each downstream costs at most $O(n^4)$ relations, and there are at most $n$ downstreams, so
    we do not exceed $O(n^5)$ relations.
    In the wake of this process, we find a $\pi$-word that is not necessarely irreducible anymore.
    Fortunately, it is clear from \zcref{lem:ck_through_downstream} that the rank
    of this $\pi$-word cannot exceed $n$. Thus, we can apply
    \zcref{lem:irreducible_pi_word}, giving us $(n^2)^2O(n^2) = O(n^6)$ relations to make
    the $\pi$-word irreducible.
\end{proof}

\section{\texorpdfstring{Bound on the Dehn function of $V$}{Bound on the Dehn function of V}}\label{sec:dehn_of_V}

The proof of the sextic bound on the Dehn function of $V$ will be by induction on the length of the words.
But this length is not always the most natural metric to employ. We will mainly focus on the
length with respect to the infinite generating set (see \zcref{def:length_inf_gen_set}), and the complexity.

\begin{definition}[complexity]\label{def:complexity}
    For a positive word $w=x_{i_1}x_{i_2}\cdots x_{i_k} \in F$,
    the complexity is given by
    \begin{equation*}
        \xi(w) = \max\{i_k,i_{k-1}+1,i_{k-2}+2,\dots,i_1+(k-1),k\}
    \end{equation*}
    Furthermore, the following properties are easy to verify and will be used freely throughout the rest of the paper.
    \begin{itemize}
        \item $|w|_\infty\leq \xi(w)$
        \item If $x_i$ occurs in $w$, then $i\leq\xi(w)$
        \item $\xi(wx_i) =\max(i,\xi(w)+1)$
        \item for any $j>i\geq 0$, $\xi(vx_jx_iw) = \xi(vx_ix_{j+1}w)$.
    \end{itemize}
    The complexity is also the number of carets in a reduced tree diagram for $w$.
\end{definition}

\begin{definition}
    For a positive word $w=x_{i_1}x_{i_2}\cdots x_{i_k} \in F$, we denote the highest subscript of $w$ by $\mu(w)$.
\end{definition}

\begin{lemma}\label{lem:relations_between_lengths}
    For an MP-word $w \in F$,
    \begin{equation*}
        |w| \leq |w|_\infty + 2\mu(w) - 2
    \end{equation*}
    where $|w|$ is the length of $w$ with respect to $\{x_0,x_1\}$.
\end{lemma}
\begin{proof}
    Let $m:=|w|_\infty$, and write
    \begin{equation*}
        w = x_{j_1}x_{j_2}\cdots x_{j_m},\qquad (j_1\leq j_2\leq\cdots\leq j_m)
    \end{equation*}
    Given the non-decreasing nature of the indices, it is rather easy to determine the length of this word.
    Let $k$ be the number of indices $j_i$ that are $0$, specifically:
    \begin{equation*}
        w = x_0^kx_{j_{k+1}}\cdots x_{j_m}
    \end{equation*}
    This word expands to
    \begin{equation*}
        x_0^{k}x_0^{-(j_{k+1}-1)}x_1x_0^{j_{k+1}-j_{k+2}}x_1\cdots x_0^{j_{m-1}-j_{m}}x_1x_0^{j_m-1}
    \end{equation*}
    which has length
    \begin{align*}
        |k-j_{k+1}+1| + (j_{k+2}-j_{k+1}) + (j_{k+3}-j_{k+2}) + \cdots + (j_m-j_{m-1}) + (m-k) + (j_m-1)\\
        = 2j_m-j_{k+1} + (m-k) + |k-j_{k+1}+1| - 1\\
        =\begin{cases*}
            2j_m-2j_{k+1} + m\\
            2j_m-2k+m-2
        \end{cases*}
    \end{align*}
    where the last equality is obtained considering both cases of the absolute value. Now,
    since by definition $j_{k+1}\geq 1$, both cases are bounded by
    \begin{equation*}
        2j_m+m-2 = 2\mu(w)+|w|_\infty - 2
    \end{equation*}
\end{proof}

The complexity also behaves very nicely when it comes to monotone positivity.
\begin{lemma}
    Let $w$ be a positive word
    in $F$ and let $w'$ be its monotone
    positive form. Then $\xi(w') \leq \xi(w)$.
\end{lemma}
\begin{proof}
    The last bulletpoint of \zcref{def:complexity} shows
    that all relations needed to rewrite a positive word to
    its MP form, do not alter the complexity of said word.
\end{proof}
We will use this property further down the line without mentioning it
explicitly.

For the main result of this paper, we will lean on the following theorem from Guba
\cite[Theorem 3]{GubaFTV}.

\begin{theorem}\label{thm:main_thm_GUBA_FTV}
    For any word $w$ of length $n$ in $V$, there exist MP-words $p,q$,
    and integer $m$, and an irreducible $\pi$-word $\pi$
    such that
    \begin{equation*}
        w = p\pi c_{n+1}^m q^{-1}
    \end{equation*}
    in $V$, with
    \begin{equation*}
        \xi(p),\xi(q)\leq n,\quad \rho(\pi)\leq n,\quad (0\leq m\leq n+2)
    \end{equation*}
\end{theorem}
As stated before, for the proof of the main theorem, we will work by induction.
Showing that the induction step is valid will take a lot of technical lemmas.
In order to motivate the statements of these lemmas and in order to keep some legibility, we will defer a lot of the technicalities to Appendix~\ref{appendix:technical_lemmas}.

\begin{theorem}\label{thm:Dehn_function_of_V}
    For any word $w$ of length $n$ in $V$, we have
    \begin{equation*}
        \|w = p\pi c_{n+1}^mq^{-1}\| = O(n^6)
    \end{equation*}
    with
    \begin{itemize}
        \item $p,q$ MP-words in $F$ with $\xi(p),\xi(q)\leq n$
        \item The exponent of $c_{n+1}$ is $m\in\{0,\dots,n+2\}$.
        \item $\pi$ an irreducible $\pi$-word of rank $\rho(\pi)\leq n$.
    \end{itemize}
\end{theorem}
\begin{proof}
    By \zcref{thm:main_thm_GUBA_FTV}, we already now that this theorem holds if we do not consider the bound on the
    number of relations needed. Thus, we will be able to do an induction solely on the claim
    that the area is bounded by $O(n^6)$.

    Cut $w$ into two words $w_1,w_2$ such that their respective lengths $n_1,n_2$
    are as close as possible to $n/2$. This ensures that $n_1+n_2=n$ and $n_1\in\{n_2-1,n_2,n_2+1\}$.
    Assume via the induction hypothesis that
    \begin{equation*}
        \|w_1 = p_1\pi_1c_{n_1+1}^{m_1}q_1^{-1}\| = O(n^6),\quad \|w_2 = p_2\pi_2c_{n_2+1}^{m_2}q_2^{-1}\| = O(n^6)
    \end{equation*}
    with the knowledge from \zcref{thm:main_thm_GUBA_FTV} that
    \begin{itemize}
        \item $p_1,q_1$ MP-words in $F$ with $\xi(p_1),\xi(q_1)\leq n_1$
        \item $p_2,q_2$ MP-words in $F$ with $\xi(p_2),\xi(q_2)\leq n_2$
        \item The exponent of $c_{n_1+1}$ is $m_1\in\{0,\dots,n_1+2\}$.
        \item The exponent of $c_{n_2+1}$ is $m_1\in\{0,\dots,n_2+2\}$.
        \item $\pi_1$ an irreducible $\pi$-word of rank $\rho(\pi_1)\leq n_1$.
        \item $\pi_2$ an irreducible $\pi$-word of rank $\rho(\pi_2)\leq n_2$.
    \end{itemize}
    Now construct the following sketch of a van Kampen diagram.
    \begin{equation*}
        \begin{tikzpicture}
            \begin{scope}[every node/.style={inside}]
                \draw (-5,0) to[out=75,in=160] node {$w_1$} (0,3);
                \draw (5,0) to[out=105,in=20] node {$w_2$} (0,3);
                \draw (-5,0) -- node[pos=0.125] {$p_1$}
                                node[pos=0.375] {$\pi_1$}
                                node[pos=0.625] {$c_{n_1+1}^{m_1}$}
                                node[pos=0.875] {$q_1^{-1}$} (0,3)
                             -- node [pos=0.125] {$p_2$}
                                node[pos=0.375] {$\pi_2$}
                                node[pos=0.625] {$c_{n_2+1}^{m_2}$}
                                node[pos=0.875] {$q_2^{-1}$} (5,0);
                \draw (-5,0) -- node[pos=0.25] {$p$} node[pos=0.75] {$\pi$} (0,-3)
                                -- node[pos=0.25] {$c_{n+1}^m$} node[pos=0.75] {$q^{-1}$} (5,0);

                \draw ($(0,3)!0.25!(5,0)$) -- node[pos=0.15] {$Q_1^{-1}$} node[pos=0.5] {$c_{n+1}^{s_1}$} node[pos=0.85] {$\pi_1'$} ($(-5,0)!0.5!(0,-3)$);
                \draw ($(0,3)!0.5!(5,0)$) -- node[pos=0.2] {$\mathcal{Q}_1^{-1}$} node[pos=0.7] {$c_{n+1}^t$} (0,-3);
                \draw ($(0,3)!0.75!(5,0)$) -- node {$\mathscr{Q}_1^{-1}$} ($(0,-3)!0.5!(5,0)$);

                \draw ($(-5,0)!0.25!(0,3)$) -- node {$\mathscr{P}_2$} ($(-5,0)!0.5!(0,-3)$);
                \draw ($(-5,0)!0.5!(0,3)$) -- node[pos=0.2] {$\mathcal{P}_2$} node[pos=0.7] {$\pi_2''$} (0,-3);
                \draw ($(-5,0)!0.75!(0,3)$) -- node[pos=0.15] {$P_2$} node[pos=0.5] {$\pi_2'$} node[pos=0.85] {$c_{n+1}^{s_2}$} ($(0,-3)!0.5!(5,0)$);
            \end{scope}
            \begin{scope}[every node/.style={circle, draw, inner sep=0.5pt}]
                \node at (0,2.2) {$1$};
                \node at (-1.3,1.1) {$2$};
                \node at (-2.6,0.2) {$3$};
                \node at (-4,-0.1) {$4$};
                \node at (1.3,1.1) {$5$};
                \node at (2.6,0.2) {$6$};
                \node at (4,-0.1) {$7$};
                \node at (0,-0.8) {$8$};
                \node at (1.4,-1.6) {$9$};
                \node at (-1.4,-1.6) {$10$};
            \end{scope}
        \end{tikzpicture}
    \end{equation*}
    Where all edges are oriented from left to right.
    We will show that this sketch is valid and that each of the $10$ marked subregions
    can be filled with at most $O(n^6)$ relations.
    \begin{enumerate}
        \item[\textbf{Region }\tikz{\node[circle, draw, inner sep=0.5pt] at (0,0) {$1$};}]
            We have
            $q_1,p_2$ MP-words with $\xi(q_1)\leq n_1$, $\xi(p_2)\leq n_2$ and $n_1+n_2=n$.
            By \zcref{lem:qp_pq}, there exist MP-words $P_2,Q_1$ such that
            \begin{equation*}
                q_1^{-1}p_2 = P_2Q_1^{-1}
            \end{equation*}
            with $|P_2|_\infty\leq n_2$, $|Q_1|_\infty\leq n_1$, and $\xi(P_2),\xi(Q_1)\leq n$.
            By \zcref{lem:relations_between_lengths}, we know that the lengths of these words are bounded as follows.
            \begin{align*}
                |q_1|&\leq |q_1|_\infty + 2\mu(q_1)-2\leq n_1+2n_1-2\\
                |p_2|&\leq |p_2|_\infty + 2\mu(p_2)-2\leq n_2+2n_2-2\\
                |Q_1|&\leq |Q_1|_\infty + 2\mu(Q_1)-2\leq n_1+2n-2\\
                |P_2|&\leq |P_2|_\infty + 2\mu(P_2)-2\leq n_2+2n-2
            \end{align*}
            such that the length of the whole boundary is at most
            \begin{equation*}
                4(n_1+n_2) + 4n-8 = 8n-8
            \end{equation*}
            The dehn function of $F$ being quadratic \cite[Theorem 1]{GubaFQuadratic} gives us a way to fill this region with
            at most $O((8n-8)^2)=O(n^2)$ cells.
        \item[\textbf{Region }\tikz{\node[circle, draw, inner sep=0.5pt] at (0,0) {$2$};}]
                We have $c_{n_1+1}^{m_1}P_2$ with $|P_2|_\infty\leq n_2$ and $\xi(P_2)\leq n$.
                By \zcref{lem:cp_pc}, there exists an MP-word $\mathcal{P}_2$ and an integer $s_1$ such that
                \begin{equation*}
                    c_{n_1+1}^{m_1}P_2 = \mathcal{P}_2c_{n+1}^{s_1}
                \end{equation*}
                with
                \begin{equation*}
                    s_1\leq n+1,\qquad |\mathcal{P}_2|_\infty\leq n_2,\qquad \xi(\mathcal{P}_2)\leq n
                \end{equation*}
                The length of this boundary can be found by adding the following bounds together.
                \begin{align*}
                    |c_{n_1+1}^{m_1}| &= m_1(2n_1+1) \leq (n_1+2)(2n_1+1) \leq 2n^2+3n+4\\
                    |P_2| &\leq |P_2|_\infty + 2\mu(P_2)-2\leq n_2+2n-2\\
                    |\mathcal{P}_2| &\leq |\mathcal{P}_2|_\infty + 2\mu(\mathcal{P}_2)-2\leq n_2+2n-2\\
                    |c_{n+1}^{s_1}| &= s_1(2n+1) \leq (n+2)(2n+1) = 2n^2+5n+2\\
                \end{align*}
                where we used \zcref{def:definitions_of_infinite_letters} for the last inequality.
                Combining these, it is clear that the length of the boundary of order $n^2$.
                Thus, \cite[Theorem A]{Migliorini_2025}
                implies that the area is of order $O((n^2)^2) = O(n^4)$.
        \item[\textbf{Region }\tikz{\node[circle, draw, inner sep=0.5pt] at (0,0) {$3$};}]
            The boundary for this region can be realised with \zcref{lem:pip_ppi}. Indeed, we have $\pi_1$ an
            irreducible $\pi$-word of rank at most $n_1$ and $\mathcal{P}_2$ an MP-word with
            $|\mathcal{P}_2|_\infty\leq n_2$, $\mu(\mathcal{P}_2)\leq n$ and $\xi(\mathcal{P}_2)\leq n$.
            \zcref{lem:pip_ppi} tells us that there exists an MP-word $\mathscr{P}_2$ and an irreducible
            $\pi$-word $\pi_1'$ such that
            \begin{equation*}
                \pi_1\mathcal{P}_2 = \mathscr{P}_2\pi_1'
            \end{equation*}
            and
            \begin{equation*}
                \xi(\mathscr{P}_2)\leq n,\quad |\mathscr{P}_2|_\infty\leq n_2,\quad \mu(\mathscr{P}_2)\leq n,\quad \rho(\pi'_1)\leq n
            \end{equation*}
            Now we need to bound the number of relations needed to achieve this.
            \zcref{lem:pip_ppi} strongly uses the relations that rewrite
            \begin{equation*}
                (\pi_{N+k}\cdots\pi_{N+1}\pi_N)x_m
            \end{equation*}
            to
            \begin{numcases}{}
                    x_m(\pi_{N+k+1}\cdots\pi_{N+2}\pi_{N+1}) & $m<N$ \tag{A}\\
                    x_{N+k+1}(\pi_{N+k}\pi_{N+k+1})\cdots(\pi_{N+1}\pi_{N+2})(\pi_N\pi_{N+1}) & $m=N$\tag{D}\\
                    x_{m-1}(\pi_{N+k+1}\cdots \pi_{N+1}\pi_N) & $N< m\leq N+k+1$\tag{C}\\
                    x_m(\pi_{N+k}\cdots\pi_{N+1}\pi_N) & $N+k+1<m$\tag{B}
            \end{numcases}
            Each of the four cases has an
            $x$-letter travel through a stream at the cost of at most $O((N+k)^3)$
            relations according to \zcref{lem:downstream,cor:x_through_stream}.
            In the case that (D) occurs, an additional
            $O((k+2)^4)$ relations are required to render the $\pi$-word irreducible again as in \zcref{lem:unify_streams}.
            In conclusion, the cost of the $i$th $x$-letter passing through the $\pi$-word is bounded by
            \begin{equation*}
                O((k+2)^4) + \sum_{j=1}^{\rho(\pi_1)+(i-1)}O((N+k)^3)\leq O((n+2)^4) + \sum_{j=1}^{\rho(\pi_1)+(i-1)}O(n^3)
            \end{equation*}
            Note the upper limit of the sum index $j$. It represents the highest
            possible rank of the $\pi$-word at the ``time'' that the $i$th
            $x$-letter passes through. Therefore it also denotes the highest possible
            amount of downstreams in the $\pi$-word.

            If we want to pass each $x$-letter in $\mathcal{P}_2$, we need at most
            \begin{equation*}   
                \sum_{i=1}^{|\mathcal{P}_2|_\infty}\Big(O((n+2)^4) + \sum_{j=1}^{\rho(\pi_1)+(i-1)}O(n^3)\Big)
            \end{equation*}
            relations. Using our known bounds for some of these values, we simplify this to
            \begin{equation*}
                \leq \sum_{k=1}^{n_2}\Big(O(n^4) + \sum_{i=1}^{n}O(n^3)\Big)\leq O(n_2n^4)+O(n_2n^4)\leq O(n^5)
            \end{equation*}
            Since we kept the $\pi$-word irreducible throughout these steps, we do not need to
            change it. The positive word might not be MP anymore, but it requires less than $O(n^5)$
            relations to rewrite it if necessary. Indeed, the complexity of the positive words are at most $n$ and the Dehn function of $F$ is quadratic \cite[Theorem 1]{GubaFQuadratic}
        \item[\textbf{Region }\tikz{\node[circle, draw, inner sep=0.5pt] at (0,0) {$5$};}]
            This region is similar to region \tikz{\node[circle, draw, inner sep=0.5pt] at (0,0) {$3$};}.
            Indeed, we start with $Q_1$ an MP-word with $|Q_1|_\infty\leq n_1$ and $\xi(Q_1)\leq n$, and
            $\pi_2$ an irreducible $\pi$-word with $\rho(\pi_2)\leq n_2$. \zcref{lem:Q-1pi_pi_Q-1} gives us
            an MP word $\mathcal{Q}_1$ and an irreducible $\pi$-word $\pi_2'$ such that
            \begin{equation*}
                Q_1^{-1}\pi_2 = \pi_2'\mathcal{Q}_1^{-1}
            \end{equation*}
            with
            \begin{equation*}
                \xi(\mathcal{Q}_1)\leq n,\quad |\mathcal{Q}_1|_\infty\leq n_1,\quad \mu(\mathcal{Q}_1)\leq n,\quad \rho(\pi_2')\leq n
            \end{equation*}
            An analogous argument to \tikz{\node[circle, draw, inner sep=0.5pt] at (0,0) {$3$};}
            shows that it costs at most $O(n^5)$ relations to fill this region.
        \item[\textbf{Region }\tikz{\node[circle, draw, inner sep=0.5pt] at (0,0) {$6$};}]
            From region \tikz{\node[circle, draw, inner sep=0.5pt] at (0,0) {$5$};}, we have $\mathcal{Q}_1$ an MP-word with
            \begin{equation*}
                \xi(\mathcal{Q}_1)\leq n,\quad |\mathcal{Q}_1|_\infty\leq n_1,\quad \mu(\mathcal{Q}_1)\leq n
            \end{equation*}
            and $c_{n_2+1}^{m_2}$. By using the relation
            $c_{n_2+1}^{n_2+3}=1$ from \zcref{lem:c_n^n+2_1}, we can rewrite
            \begin{equation*}
                \mathcal{Q}_1^{-1}c_{n_2+1}^{m_2} = \Big(c_{n_2+1}^{n_2+3-m_1}\mathcal{Q}_1\Big)^{-1}
            \end{equation*}
            Now we can apply \zcref{lem:cp_pc} where we exchange the roles of $n_1$ and $n_2$ to find an
            MP-word $\mathscr{Q}_1$ and an integer $s$ such that
            \begin{equation*}
                \Big(c_{n_2+1}^{n_2+3-m_1}\mathcal{Q}_1\Big)^{-1} = \Big(\mathscr{Q}_1c_{n+1}^s\Big)^{-1} = c_{n+1}^{n+3-s}\mathscr{Q}_1^{-1}
            \end{equation*}
            with
            \begin{equation*}
                s\leq n+2,\quad \xi(\mathscr{Q}_1)\leq n,\quad \mu(\mathscr{Q}_1)\leq n
            \end{equation*}
            Now that we know that the boundary of this region is realiseable, we need to bound its area.
            But this is analogous to bounding the area of region \tikz{\node[circle, draw, inner sep=0.5pt] at (0,0) {$2$};}.
            We conclude a cost of $O(n^4)$.
        \item[\textbf{Region }\tikz{\node[circle, draw, inner sep=0.5pt] at (0,0) {$8$};}]
            We start with $c_{n+1}^{s_1}$ and an irreducible $\pi$-word $\pi_2'$ of rank at most $n$.
            These meet the criteria of \zcref{lem:c_through_pi_word} and therefore
            \begin{equation*}
                \|c_{n+1}^{s_1}\pi_2'=\pi_2''c_{n+1}^{t}\| = O(n^6)
            \end{equation*}
            for some irreducible $\pi$-word $\pi_2''$ of rank at most $n$ and $t\leq n+2$.
        \item[\textbf{Region }\tikz{\node[circle, draw, inner sep=0.5pt] at (0,0) {$9$};}]
            The two terms $c_{n+1}^{t}$ and $c_{n+1}^{s_2}$ are easy to combine and can
            be reduced if necessary via $c_{n+1}^{n+3}=1$.
            \begin{equation*}
                c_{n+1}^t c_{n+1}^{s_2} = c_{n+1}^{t+s_2} = \begin{cases*}
                    c_{n+1}^{t+s_2} & $t+s_2\leq n+2$\\
                    c_{n+1}^{t+s_2-(n+2)} & $t+s_2> n+2$
                \end{cases*}
            \end{equation*}
            This last case is the only one that needs any relations. According to
            \zcref{lem:c_n^n+2_1} it costs $O(n^3)$ relations.
        \item[\textbf{Region }\tikz{\node[circle, draw, inner sep=0.5pt] at (0,0) {$10$};}]
            This region can be realised by making the word $\pi_1'\pi_2''$ irreducible.
            We recall that $\rho(\pi_1'),\rho(\pi_2'')\leq n$, and thus
            \begin{equation*}
                \rho(\pi_1'\pi_2'')\leq n
            \end{equation*}
            Since no relation applicaple to $\pi$-words can increase the subscripts of $\pi$-letters,
            the irreducible form of this word, which we denote by $\pi$, has rank bounded by $n$ as expected.
            By \zcref{lem:irreducible_pi_word}, the area of this region is bounded by
            \begin{equation*}
                |\pi_1'\pi_2''|_\pi^2 O(n^2)\leq \big(|\pi_1'|_\pi+|\pi_2''|_\pi\big)^2 O(n^2) \leq \big(O(n^2)+O(n^2)\big)^2O(n^2) = O(n^6)
            \end{equation*}
            where we used \zcref{cor:pi_length_of_irr_pi_word} to bound $|\pi_1'|_\pi$ and $|\pi_2''|_\pi$.
        \item[\textbf{Region }\tikz{\node[circle, draw, inner sep=0.5pt] at (0,0) {$4$};}]
            We have two MP-words $p_1$ and $\mathscr{P}_2$ with
            \begin{equation*}
                \xi(p_1)\leq n_1,\quad \xi(\mathscr{P}_2)\leq n,\quad |\mathscr{P}_2|_\infty\leq n_2,\quad \mu(\mathscr{P}_2)\leq n
            \end{equation*}
            We know that there exists an MP-word $p$ such that
            \begin{equation*}
                p_1\mathscr{P}_2=p
            \end{equation*}
            Furthermore,
            \begin{equation*}
                \xi(p_1\mathscr{P}_2) = \max\{\xi(\mathscr{P}_2),\xi(p_1)+|\mathscr{P}_2|_\infty\} \leq \max\{n,n_1+n_2\} = n
            \end{equation*}
            It remains to show that this region
            can be filled economically. The length of the boundary is bounded by
            \begin{align*}
                |p_1|+|\mathscr{P}_2| + |p|&\leq |p_1|_\infty+2\mu(p_1)-2+|\mathscr{P}_2|_\infty+2\mu(\mathscr{P}_2)-2+|p|_\infty+2\mu(p)-2\\
                                            &\leq n_1+2n_1-2+n_2+2n-2+n+2n-2\\
                                            &\leq 7n-5
            \end{align*}
            where for the last inequality, we used that $n_1+n_2=n$ and $n_1+n_1\leq n+2$.
            Thus, the area must be bounded by $O((7n-5)^2)=O(n^2)$.
        \item[\textbf{Region }\tikz{\node[circle, draw, inner sep=0.5pt] at (0,0) {$7$};}]
            This last region is analogous to \tikz{\node[circle, draw, inner sep=0.5pt] at (0,0) {$4$};}.
    \end{enumerate}
    Summing up all the bounds on the areas of the regions, we get
    \begin{equation*}
        3O(n^2)+O(n^3)+2O(n^4)+O(n^5)+4O(n^6) = O(n^6)
    \end{equation*}
\end{proof}

\begin{theorem}
    The Dehn function of Thompson's group $V$ is at most of order $O(n^6)$.
\end{theorem}
\begin{proof}
    Take any word $w$ of length $\leq n$ on the alphabet $\{x_0,x_0^{-1},x_1,x_1^{-1},c_1,c_1^{-1},\pi_0,\pi_0^{-1}\}$ that represents the identity in $V$. Replace every occurrence of $c_1^{-1}$ by $c_1^2$ and every occurence of $\pi_0^{-1}$ by $\pi_0$ using the standard relations. This will cost at most $O(n)$ relations and increase the length $w$ with at most $n$. Call the new length $N$.
    By \zcref{thm:main_thm_GUBA_FTV,thm:Dehn_function_of_V}, 
    there exist MP-words $p,q$, an integer $m$ and an irreducible $\pi$-word $\pi$ such that
    \begin{equation*}
        \|w=p\pi c_{N+1}^mq^{-1}\| = O(N^6)
    \end{equation*}
    and $\xi(p),\xi(q)\leq N$.
    Since $w$ represents the identity, we must have that $\pi c_{N+1}^m=p^{-1}q\in F$.
    This is only possible if $m=0$ and $\pi=1$ as words. Thus,
    \begin{equation*}
        \|w=pq^{-1}\| = O(N^6)
    \end{equation*}
    Now,
    \begin{equation*}
        |pq^{-1}|\leq |p| + |q|\leq |p|_\infty + 2\mu(p) - 2 +|q|_\infty + 2\mu(q)-2\leq 6N-4
    \end{equation*}
    Which combined with Guba's result on $F$ \cite[Theorem 1]{GubaFQuadratic}, gives that
    \begin{equation*}
        \|w=1\| = O(N^6) + O(N^2) = O(N^6) = O(n^6)
    \end{equation*}
\end{proof}

\appendix

\section{Technical lemmas}\label{appendix:technical_lemmas}
We gather here, the technical lemmas referenced in \zcref{thm:Dehn_function_of_V}.
We start by recalling a result by Guba which allows to move inverses of $x$-letters through words.

\begin{lemma}\label{lem:x-1_through_positive_word}
    \emph{(\cite[lemma 4]{GubaFTV})}
    For $v$ a positive word and $j\geq0$, there exists a positive word $w$ such that
    \begin{equation*}
        x_i^{-1}v = wx_j^{-\delta}\qquad (\delta=0,1)
    \end{equation*}
    with $j\leq |v|_\infty+i$, $\xi(w)\leq\xi(v)+1$, and $|w|_\infty\leq|v|_\infty$.
\end{lemma}

\begin{lemma}\label{lem:qp_pq}
    Let $p,q$ be MP-words such that $\xi(p)\leq n_2$, $\xi(q)\leq n_1$ and $n_1+n_2=n$.
    Then, there exist MP-words $P,Q$ such that
    \begin{equation*}
        q^{-1}p = PQ^{-1}
    \end{equation*}
    with $|P|_\infty\leq n_2$, $|Q|_\infty\leq n_1$, and $\xi(P),\xi(Q)\leq n$.
\end{lemma}
\begin{proof}
    First, let us show that passing the negative letters in $q^{-1}$ through the MP-word $p$---%
    as done in \zcref{lem:x-1_through_positive_word}---does not alter the MP status of $p$.
    Indeed, we have the equalities
    \begin{equation*}
        x_k^{-1}x_j = \begin{cases*}
            x_{j+1}x_k^{-1} & $k<j$\\
            1 & $k=j$\\
            x_jx_{k-1}^{-1} & $k>j$
        \end{cases*}
    \end{equation*}
    The only case that could potentially disrupt the monotone positivity, is the case
    where $k<j$. But since the negative letter is travelling from left to right
    through the MP-word, all subsequent cases must also fall in this case, thus preserving the
    monotone positivity.
    \zcref{lem:x-1_through_positive_word} also tells us that passing each letter in $q^{-1}$ through $p$,
    will cause $\xi(p)$ to increase with at most $|q|_\infty$. Therefore,
    \begin{equation*}
        \xi(P)\leq \xi(p) + \xi(q)\leq n_1+n_2 = n
    \end{equation*}
    Furthermore $|p|_\infty$ does not increase, thus $|P|_\infty\leq |p|_\infty\leq n_2$.
    By symmetry, the dual holds for $Q$.
\end{proof}

\begin{lemma}\label{lem:cp_pc}
    For any integers $0<m_1,n_1,n_2,n$ such that
    \begin{equation*}
        m_1\leq n_1+2,\quad n_1+n_2=n,
    \end{equation*}
    and for any MP-word $P$ with $l:=|P|_\infty\leq n_2$ and $\xi(P)\leq n$, we have that
    \begin{equation*}
        c_{n_1+1}^{m_1}P = \mathcal{P}c_{n+1}^{s}
    \end{equation*}
    with $\mathcal{P}$ an MP-word and
    \begin{equation*}
        s\leq n+2,\quad \xi(\mathcal{P})\leq n
    \end{equation*}
\end{lemma}
\begin{proof}
    Consider the equalities from \zcref{lem:cm+1sxj}
        \begin{equation}\label{eq:x_through_cs}
        c_{m+1}^sx_j = \begin{cases*}
            x_{j-s}c_{m+2}^s & $j\geq s$ \quad\;\;\;(A)\\
            c_{m+2}^{s+1} & $j=s-1$ (B)\\
            x_{j+m+3-s}c_{m+2}^{s+1} & $j<s-1$  (C)
        \end{cases*}
    \end{equation}
    We can read that any $x$-letter from $P$ will increase the subscript of the $c$-letter by exactly $1$. To obtain the subscript of $n+1$, we first raise the subscript of $c$ to $n+1-|P|_\infty = n-l+1$.
    \begin{align*}
        c_{n_1+1}^{m_1} &= c_{n_1+1}^{m_1-1}x_{n_1+1}c_{n_1+2}\\
            &= c_{n_1+1}^{m_1-1}x_{n_1+1}x_{n_1+2}c_{n_1+3}\\
            &\vdots\\
            &= c_{n_1+1}^{m_1-1}x_{n_1+1}x_{n_1+2}\cdots x_{n-l}c_{n-l+1}\\
            &= x_{n_1+1-(m_1-1)}x_{n_1+2-(m_1-1)}\cdots x_{n-l-(m_1-1)}c_{n-l+1}^{m_1}\tag{\zcref{eq:x_through_cs}(A)}
    \end{align*}
    Next, we need to pass every $x$-letter in $P$ through $c_{n-l+1}^{m_1}$. But we need to be careful. We can only use \zcref{eq:x_through_cs} if the subscript of the $x$-letter is at most the subscript of the $c$-letter. Let $P$ be
    \begin{equation*}
        P = x_{i_1}x_{i_2}\cdots x_{i_l}
    \end{equation*}
    Since $\xi(P)\leq  n$,
    \begin{equation*}
        \max\big\{i_l,i_{l-1}+(1),\cdots i_2+(l-2), i_1+(l-1)\big\}\leq n
    \end{equation*}
    So we can read that $i_1+(l-1)\leq n\implies i_1\leq n-l+1$. This means that the first letter of $P$ can pass all $c$-letters via \zcref{eq:x_through_cs}. This will increase the subscript of said $c$-letters, but cause no issue as by the definition of complexity, the subscript of the next $x$-letter in $P$ is at most 1 more than the previous. Inductively, every $x$-letter in $P$ may pass through. We conclude that
    \begin{equation*}
        c_{n_1+1}^{m_1} = x_{n_1+1-(m_1-1)}x_{n_1+2-(m_1-1)}\cdots x_{n-l-(m_1-1)}\mathcal{P}'c_{n+1}^s\underset{\text{make MP}}{=}\mathcal{P}c_{n+1}^s
    \end{equation*}
    for some positive word $\mathcal{P}'$ in $F$ and $s\leq n+2$, and $\mathcal{P}$ the monotone positive word that can be obtained from $x_{n_1+1-(m_1-1)}x_{n_1+2-(m_1-1)}\cdots x_{n-l-(m_1-1)}\mathcal{P}'$. It remains to show the bounds on $\mathcal{P}$.
    First we tackle the length over the infinite presentation:
    \begin{align*}
        |\mathcal{P}|_\infty &\leq |x_{n_1+1-(m_1-1)}x_{n_1+2-(m_1-1)}\cdots x_{n-l-(m_1-1)}| + |\mathcal{P}'|\\
        &= \Big(n-l-(m_1-1)-(n_1+1-(m_1-1))+1\Big) + |\mathcal{P'}|_\infty\\
        &= n-n_1-l + |\mathcal{P}'|_\infty\\
        &= n_2-|P|_\infty+|\mathcal{P}'|_\infty\leq n_2
    \end{align*}
    where the last inequality follows from the fact that none of the cases in \zcref{eq:x_through_cs} increase the number of $x$-letters.
    Next, we look at the complexity and use the properties at the end of 
    \zcref{def:complexity}:
    \begin{align*}
        \xi(\mathcal{P}) &= \xi\Big(x_{n_1+1-(m_1-1)}x_{n_1+2-(m_1-1)}\cdots x_{n-l-(m_1-1)}\mathcal{P}'\Big)\\
            &= \max\Big\{\xi(\mathcal{P'}),\xi(x_{n_1+1-(m_1-1)}x_{n_1+2-(m_1-1)}\cdots x_{n-l-(m_1-1)})+|\mathcal{P}'|_\infty,|\mathcal{P}|_\infty\Big\}\\
            &= \max\Big\{\xi(\mathcal{P'}),n-l-(m_1-1)+|\mathcal{P}'|_\infty,n_2-l+|\mathcal{P}'|_\infty,|\mathcal{P}|_\infty\Big\}\\
            &\leq \max\Big\{\xi(\mathcal{P'}),n-m_1+1,n_2,|\mathcal{P}|_\infty\Big\}\\
            &\leq \max\Big\{\xi(\mathcal{P}'),n\Big\}
    \end{align*}
    Therefore, it suffices to show that $\xi(\mathcal{P}')\leq n$. This word is created by a sequence of cases (A), (B), and (C) from \zcref{eq:x_through_cs}. Recall that
    \begin{equation*}
        P = x_{i_1}x_{i_2}\cdots x_{i_l},\quad l=|P|_\infty,\quad P \text{ is MP }\implies (i_1\leq i_2\leq \cdots\leq i_l)
    \end{equation*}
    and that we are pushing the $x$-letters through the $c$-letters in the expression $c_{n-l+1}^{m_1}P$.
    We observe first that if at some point case (A) occurs because $i_j$ is greater (or equal)
    than the exponent of the $c$-letter at that step, then all following subscripts of $x$-letters are also
    greater than this exponent, and (B) and (C) cannot occur anymore. This means that the most general sequence
    of steps to obtain $\mathcal{P}'$ is
    \begin{equation*}
        \underbrace{(B/C),(B/C),\dots,(B/C)}_{k\text{ times}},\underbrace{(A),(A),\cdots,(A)}_{l-k\text{ times}}
    \end{equation*}
    But case (B) does not contribute any $x$-letter to $\mathcal{P}'$, and thus
    does not increase the complexity. Since we are looking for an upper bound, we can consider the
    ``worst case scenario'' and work with
    \begin{equation*}
        \underbrace{(C),(C),\dots,(C)}_{k\text{ times}},\underbrace{(A),(A),\cdots,(A)}_{l-k\text{ times}}
    \end{equation*}
    Applying this sequence of cases yields
    \begin{align*}
        c_{n-l+1}^{m_1}P = \overbrace{\Big(x_{i_1+n-l+3-m_1}x_{i_2+n-l+(1)+3-(m_1+1)}\cdots x_{i_k+n-l+(k-1)+3-(m_1+k-1)}\Big)}^{(C)\times k}\\
            \underbrace{\Big(x_{i_{k+1}-(m_1+k)}x_{i_{k+2}-(m_1+k)}\cdots x_{i_l-(m_1+k)}\Big)}_{(A)\times (l-k)}c_{r+1}^{m_1+k}
    \end{align*}
    with the added information:
    \begin{equation}\label{eq:index_knowledge}
        \begin{aligned}
            i_j<m_1+j-2,\text{ for all } j\text{ such that }1\leq j\leq k &\qquad (\text{stemming from (C)$\times k$})\\
            i_j\geq m_1+k,\text{ for all } j\text{ such that } k<j\leq l &\qquad (\text{stemming from (A)$\times(l-k)$})
        \end{aligned}
    \end{equation}
    We prepare the calculation of the complexity of $\mathcal{P}'$ by
    recalling that we know that $\xi(P)\leq n$, and thus
    \begin{equation}\label{eq:complexity_of_P}
        \max\{j_l,j_{l-1}+1,j_{l-2}+2,\dots,j_1+(l-1),l\}\leq n
    \end{equation}
    Now,
    \begin{equation}\label{eq:complexity_of_P_prime}
        \begin{aligned}
            \xi(\mathcal{P}') = \max\Big\{&\big\{\red{i_l}-(m_1+k)+\red{(0)},\red{i_{l-1}}-(m_1+k)+\red{(1)},\dots,\red{i_{l-(l-k-1)}}-(m_1+k)+\red{(l-k-1)}\big\}\\
                                    &\cup\big\{\blue{i_k}+n-l+(k-1)\blue{+3}-(\blue{m_1}+k-1)+(l\blue{-k}),\dots,\\
                                    &\qquad \blue{i_2}+n-l+(1)\blue{+3}-(\blue{m_1}+1)+(l\blue{-2}),\blue{i_1}+n-l\blue{+3-m_1}+(l\blue{-1})\big\},\\
                                    &|\mathcal{P}'_\infty|\Big\}
        \end{aligned}
    \end{equation}
    The first of the two subsets in the maximum can be simplified by remarking that the highlighted terms are bounded above by $n$ through \zcref{eq:complexity_of_P}. For the second internal set of \zcref{eq:complexity_of_P_prime}, we use the first inequalities of \zcref{eq:index_knowledge}. If we rewrite them, we get:
    \begin{equation*}
        i_j-m_1-j+3\leq 0,\; \forall j\colon 1\leq j\leq k
    \end{equation*}
    With these, we can also simplify the second internal set; \zcref{eq:complexity_of_P_prime} becomes
    \begin{equation*}
        \xi(\mathcal{P}')\leq \max\Big\{\big\{n-(m_1+k)\big\}\cup\big\{n\big\},|\mathcal{P}'|_\infty\Big\}\leq n
    \end{equation*}
\end{proof}

\begin{lemma}\label{lem:pip_ppi}
    Let $\overline{\pi}$ be an irreducible $\pi$-word of rank at most $n_1$, and let $\mathcal{P}$
    be a MP-word with $\xi(\mathcal{P})\leq n$ and $|\mathcal{P}|_\infty\leq n_2$.
    Then there exists an MP-word $\mathscr{P}$ and an irreducible $\pi$-word $\overline{\pi'}$ such that
    \begin{equation*}
        \overline{\pi}\mathcal{P} = \mathscr{P}\overline{\pi'}
    \end{equation*}
    and
    \begin{equation*}
        \xi(\mathscr{P})\leq n,\quad |\mathscr{P}|_\infty\leq n_2,\quad \mu(\mathscr{P})\leq n,\quad \rho(\overline{\pi'})\leq n
    \end{equation*}
\end{lemma}
\begin{proof}
    Since we know that we can pass $x$-letters through downstreams
    (as in \zcref{lem:downstream}), one can see that the form $\mathscr {P}\overline{\pi'}$
    with $\mathscr{P}$ an MP-word and $\overline{\pi'}$ an irreducible $\pi$-word, is attainable.
    Indeed, it suffices to pass each $x$-letter and then make the positive word MP and
    the $\pi$-word irreducible.
    Recall the relations in \zcref{lem:downstream} where we rewrote
    \begin{equation}\label{eq:x_through_downstream}
        (\pi_{n+k}\cdots\pi_{n+1}\pi_n)x_m
    \end{equation}
    to
    \begin{numcases}{}
            x_m(\pi_{n+k+1}\cdots\pi_{n+2}\pi_{n+1}) & $m<n$ \tag{A}\\
            x_{n+k+1}(\pi_{n+k}\pi_{n+k+1})\cdots(\pi_{n+1}\pi_{n+2})(\pi_n\pi_{n+1}) & $m=n$\tag{D}\\
            x_{m-1}(\pi_{n+k+1}\cdots \pi_{n+1}\pi_n) & $n< m\leq n+k+1$\tag{C}\\
            x_m(\pi_{n+k}\cdots\pi_{n+1}\pi_n) & $n+k+1<m$\tag{B}
    \end{numcases}
    It is clear from this that the rank of $\overline{\pi}$ can increase by at most $1$ per
    $x$-letter in $\mathcal{P}$. Thus,
    \begin{equation*}
        \rho(\pi')\leq \rho(\pi)+|\mathcal{P}|_\infty \leq n_1 + n_2 = n
    \end{equation*}
    Moreover, the relations do not increase the number of $x$-letters, thus
    \begin{equation*}
        |\mathscr{P}|_\infty\leq |\mathcal{P}|_\infty\leq n_2
    \end{equation*}
    The subscripts of $x$-letters can only increase in case (D). As the rank of the $\pi$-word can
    only reach its theoretical maximum of $n$ during the pass of the last $x$-letter, the highest possible
    subscript for an $x$-letter is $(n-1)+1=n$. Thus $\mu(\mathscr{P})\leq n$.
    Lastly, we look at the complexity of $\mathscr{P}$. If an $x$-letter of
    $\mathcal{P}$ passed through the $\pi$-word via only cases (A), (B), and (C),
    it is clear that the complexity cannot increase. Indeed, the number of $x$-letters
    would not change (for all cases, even (D)), and the subscripts could not increase.
    Therefore, we only need to consider case (D) again. The same argument that
    leads to the bound on $\mu(\mathscr{P})$, can be repeated for every $x$-letter passing.
    This leads to
    \begin{equation*}
        \xi(\mathscr{P})\leq\max\{\mu(\mathscr{P}),\mu(\mathscr{P})-1+(1),\cdots,\mu(\mathscr{P})-|\mathcal{P}|_\infty+(|\mathcal{P}|_\infty),\xi(\mathcal{P}),|\mathcal{P}|_\infty\}\leq n
    \end{equation*}
\end{proof}

\begin{lemma}\label{lem:Q-1pi_pi_Q-1}
    Let $Q$ be an MP-word with $\xi(Q)\leq n$, $|Q|_\infty\leq n_1$, $\mu(Q)\leq n$, and
    let $\overline{\pi}$ be an irreducible $\pi$-word with rank at most $n_2$. Let $n_1+n_2=n$.
    Then, there exists an irreducible $\pi$-word $\overline{\pi'}$ and an MP-word $\mathcal{Q}$ such that
    \begin{equation*}
        Q^{-1}\overline{\pi} = \overline{\pi'}\mathcal{Q}^{-1}
    \end{equation*}
    with
    \begin{equation*}
        \xi(\mathcal{Q})\leq n,\quad |\mathcal{Q}|_\infty\leq n_1,\quad \mu(\mathcal{Q})\leq n,\quad \rho(\overline{\pi'})\leq n
    \end{equation*}
\end{lemma}
\begin{proof}
    We start by rewriting \zcref{eq:x_through_downstream} by multiplying with the appropriate
    $x^{-1}$-letters on the left and right, and adapting the subscripts accordingly. The
    result is that
    \begin{equation*}
        x_m^{-1}(\pi_{n+k}\cdots \pi_{n+1}\pi_n)
    \end{equation*}
    is equal to
    \begin{numcases}{}
        (\pi_{n+k+1}\cdots\pi_{n+2}\pi_{n+1})x_m^{-1} & $m<n$ \tag{A}\\
        (\pi_{n+k+1}\cdots\pi_{n+1}\pi_n)x_{m+1}^{-1} & $n\leq m<n+k+1$\tag{C}\\
        (\pi_{n+k}\pi_{n+k+1})\cdots(\pi_{n+1}\pi_{n+2})(\pi_n\pi_{n+1})x_n^{-1} & $m=n+k+1$\tag{D}\\
        (\pi_{n+k}\cdots\pi_{n+1}\pi_n)x_m^{-1} & $n+k+1<m$\tag{B}
    \end{numcases}
    This shows that we can get every $x^{-1}$-letter to travel through the
    $\pi$-word and obtain the form $\overline{\pi'}\mathscr{Q}^{-1}$ after making the $\pi$-word irreducible
    and making the positive word MP.
    It is also clear that the number of $x$-letters cannot increase, and thus $|\mathcal{Q}|_\infty\leq n_1$.
    The rank of the $\pi$-word increases by at most $1$ per $x$-letter, thus
    \begin{equation*}
        \rho(\overline{\pi'})\leq \rho(\overline{\pi})+|Q|_\infty\leq n_2+n_1 = n
    \end{equation*}
    The only case that increases the subscripts of $x$-letters is case (C). Since this
    case can only occur when the subscript of the $x$-letter is bounded above by the rank
    of the stream, we also conclude that $\mu(\mathcal{Q})\leq n$. We are left with bounding the
    complexity of $\mathcal{Q}$. Since cases (A) and (B) do not alter the subscript of the
    $x$-letter, they also do not change the complexity. Case (D) can only decrease the complexity.
    Therefore, we only consider case (C). This case can only occur when the subscript of the
    $x$-letter is at most the rank of the $\pi$-word at that step in the rewriting. Denote
    \begin{equation*}
        Q^{-1} = x_{j_l}^{-1}\cdots x_{j_2}^{-1}x_{j_1}^{-1}
    \end{equation*}
    By an argument analogous to the one in \zcref{lem:pip_ppi}, when $x_{j_i}^{-1}$ passes through the $\pi$-word, the rank
    can be at most $\rho(\overline{\pi})+(i-1)\leq n_2+i-1$. This implies that the subscript of $x_{j_i}$
    can only increase up to $n_2+i$. Therefore,
    \begin{align*}
        \xi(\mathcal{Q})&\leq \max\Big\{n_2+l,n_2+(l-1),\dots,n_2,\xi(Q)\Big\}\\
                        &\leq \max\Big\{n_2+n_1,\xi(Q)\Big\} = n
    \end{align*}
\end{proof}

\bibliographystyle{apalike}
\bibliography{Dehn_V_biblio.bib}

@article{
    Migliorini_2025,
        title={Thompson’s group {$T$} has quadratic {D}ehn function},
        volume={13},
        DOI={10.1017/fms.2025.10075},
        journal={Forum of Mathematics, Sigma},
        author={Migliorini, Matteo},
        year={2025}, pages={e109}
    }

@misc{WangT,
      title={The {D}ehn function of {R}ichard {T}hompson's group {$T$}}, 
      author={Kun Wang and Zhu-Jun Zheng and Junhuai Zhang},
      year={2015},
      eprint={1511.04148},
      archivePrefix={arXiv},
      primaryClass={math.GR},
      DOI = {https://arxiv.org/abs/1511.04148v2},
      note = "https://arxiv.org/abs/1511.04148v2"
}

@InProceedings{GubaFTV,
        author="Guba, V. S.",
        editor="Birget, Jean-Camille
        and Margolis, Stuart
        and Meakin, John
        and Sapir, Mark",
        title="Polynomial Isoperimetric Inequalities for {R}ichard {T}hompson's Groups {$F, T$}, and {$V$}",
        booktitle="Algorithmic Problems in Groups and Semigroups",
        year="2000",
        publisher="Birkh{\"a}user Boston",
        address="Boston, MA",
        pages="91--120",
        isbn="978-1-4612-1388-8"
}

@article{GubaFQuadratic,
        author = {Guba, V.S.},
        year = {2006},
        month = {02},
        pages = {313-342},
        title = {The {D}ehn function of {R}ichard {T}hompson’s group 
        {$F$}
        is quadratic},
        volume = {163},
        journal = {Inventiones mathematicae},
        doi = {10.1007/s00222-005-0462-z}
}

@article{IntroNotesF,
author = {Cannon, James and Floyd, William and Parry, William},
year = {1996},
month = {},
pages = {215-256},
title = {Introductory notes on {R}ichard {T}hompson's Groups},
volume = {42},
journal = {L'Enseignement Mathématique},
}

@article{BurClearSteinTabackcombMetrpropT,
author = {Burillo, José and Cleary, Sean and Stein, Melanie and Taback, Jennifer},
year = {2008},
month = {09},
pages = {631-652},
title = {Combinatorial and metric properties of {T}hompson’s group {$T$}},
volume = {361},
journal = {Transactions of The American Mathematical Society},
doi = {10.1090/S0002-9947-08-04381-X}
}

@incollection{BridsonGeometryOfWordProblem,
    author = {Bridson, Martin R},
    isbn = {9780198507727},
    title = {The geometry of the word problem},
    booktitle = {Invitations to Geometry and Topology},
    publisher = {Oxford University Press},
    year = {2002},
    month = {10},
    doi = {10.1093/oso/9780198507727.003.0002},
    url = {https://doi.org/10.1093/oso/9780198507727.003.0002},
    eprint = {https://academic.oup.com/book/0/chapter/422757766/chapter-pdf/52608021/isbn-9780198507727-book-part-2.pdf},
}

@book{BradyRileyShortGeomOfWordProblem,
author = {Brady, Noel. and Riley, Tim and Short, Hamish.},
address = {Basel},
booktitle = {The geometry of the word problem for finitely generated groups},
isbn = {9783764379490},
language = {eng},
lccn = {2006936541},
publisher = {Birkhauser},
series = {Advanced courses in mathematics, CRM Barcelona},
title = {The geometry of the word problem for finitely generated groups / Noel Brady, Tim Riley, Hamish Short.},
year = {2007},
}

@article{EgbertVanKampen,
 ISSN = {00029327, 10806377},
 URL = {http://www.jstor.org/stable/2371129},
 author = {Egbert R. Van Kampen},
 journal = {American Journal of Mathematics},
 number = {1},
 pages = {268--273},
 publisher = {Johns Hopkins University Press},
 title = {On Some Lemmas in the Theory of Groups},
 urldate = {2026-07-21},
 volume = {55},
 year = {1933}
}

@book{OfficeHours,
 URL = {http://www.jstor.org/stable/j.ctt1vwmg8g},
 publisher = {Princeton University Press},
 title = {Office Hours with a Geometric Group Theorist},
 urldate = {2024-04-04},
 year = {2017}
}
\end{document}